\documentclass[a4j,11pt,pdflatex]{amsart}
\usepackage[top=25truemm,bottom=30truemm,left=30truemm,right=30truemm]{geometry}
\usepackage{amssymb}
\usepackage{amsmath}
\usepackage[pdftex]{graphicx,color}
\usepackage{float,subfigure}
\newtheorem{Theorem}{Theorem}[section]
\newtheorem{Lemma}[Theorem]{Lemma}
\newtheorem{Proposition}[Theorem]{Proposition}
\newtheorem{Corollary}[Theorem]{Corollary}
\theoremstyle{definition}

\theoremstyle{remark}
\newtheorem{Remark}[Theorem]{Remark}
\newtheorem{Example}{Example}
\numberwithin{equation}{section}
\newcommand{\R}{\mathbb{R}}
\newcommand{\Z}{\mathbb{Z}}
\newcommand{\D}{\mathbb{D}}
\newcommand{\myC}{\mathbb{C}}
\newcommand{\mysl}{\mathfrak{sl}_3 \R}
\newcommand{\SL}{\mathrm{SL}_3 \R}
\newcommand{\slc}{\mathfrak{sl}_3 \myC}
\newcommand{\SLC}{\mathrm{SL}_3 \myC}
\newcommand{\LSL}{\Lambda \SL}
\newcommand{\LSLC}{\Lambda \mathrm{SL}_3 \myC}
\newcommand{\LSLP}{\Lambda^+ \SL}
\newcommand{\LSLN}{\Lambda^- \SL}
\newcommand{\LSLPN}{\Lambda^+_{*} \SL}
\newcommand{\LSLNN}{\Lambda^-_{*} \SL}
\newcommand{\Lsl}{\Lambda \mysl}
\newcommand{\LslP}{\Lambda^+ \mysl}
\newcommand{\LslN}{\Lambda^- \mysl}
\newcommand{\LslPN}{\Lambda^+_{*} \mysl}
\newcommand{\LslNN}{\Lambda^-_{*} \mysl}
\newcommand{\id}{\operatorname{id}}
\newcommand{\diag}{\operatorname{diag}}
\newcommand{\sign}{\operatorname{sign}}
\newcommand{\tr}{\operatorname{tr}}
\newcommand{\pd}[1]{\partial_{#1}}
\begin{document}
\title{Representation formula for discrete indefinite affine spheres}

\author[S.-P.~Kobayashi]{Shimpei Kobayashi}
\address{Department of Mathematics,
Hokkaido University, Sapporo 060-0810, Japan}
\email{shimpei@math.sci.hokudai.ac.jp}
% \thanks{The first author is partially supported by Kakenhi 26400059}

\author[N.~Matsuura]{Nozomu Matsuura}
\address{Department of Education and Creation Engineering,
Kurume Institute of Technology,
Kurume 830-0052, Japan}
\email{nozomu@kurume-it.ac.jp}
% \thanks{The second author is partially supported by Kakenhi 15K04862}

\thanks{This work was partially supported by JSPS KAKENHI
Grant Numbers JP26400059, JP15K04862, JP18K03265 and JP19K03507.
The first named author was also
%partially
supported by
Deutsche Forschungsgemeinschaft-Collaborative Research Center, TRR 109,
``Discretization in Geometry and Dynamics''.
The second named author was also
% partially
supported by the Fukuoka University fund,
grant number 177102.}

\date{\today}
\pagestyle{plain}

\begin{abstract}
We present a representation formula for discrete indefinite affine spheres
via loop group factorizations.
This formula is derived from the Birkhoff decomposition of loop groups
associated with discrete indefinite affine spheres.
In particular we show that a discrete indefinite improper affine sphere
can be constructed from two discrete plane curves.
\end{abstract}

\maketitle

\section*{Introduction}

\nocite{tz1908-1}
\nocite{tz1908-2}
\nocite{tz1909}
\nocite{tz1910-1}

Around 1908 Tzitzeica introduced surfaces in \cite{tz1907}--\cite{tz1910-2},
which are now called proper affine spheres with center at the origin,
with the property that
the Gaussian curvature is proportional to the fourth power
of the support function from the origin.
He observed that this property is invariant under an affine transformation
fixing the origin.
This work is regarded as the source of affine differential geometry
of surfaces,
and gives his name to the structure equation of proper affine spheres.
The reader is referred to %the introductory part of
\cite{MR1729560} for an account of the Tzitzeica equation
within its classical context %of hyperbolic surfaces in centroaffine geometry.
of surface theory in equicentroaffine geometry.
The Tzitzeica equation is now known to be one of the most famous
soliton equations in the theory of integrable systems
(\cite{MR0442516}, \cite{MIKHAILOV198173}, \cite{MR550472},
\cite{MR1774990}).
In fact it is obtained by a so-called $B$-type reduction of
the $2$-dimensional Toda lattice equation (\cite{MR1490247}).
The proper affine sphere can be understood as
an affine geometric analogue of the sphere,
in a sense that its affine normals meet at the origin.
When the affine normals are parallel,
the surface can be regarded as an analogue of the plane,
and is called an improper affine sphere.
The improper affine sphere is described by the Liouville equation,
which is also known to be integrable.

It is a distinctive feature of integrable systems
that we can discretize them while keeping their integrability.
For example,
a discrete Liouville equation was derived in \cite{hirota1979JPSJ46}
using the bilinear techniques,
and the corresponding discrete improper affine sphere
was introduced in \cite{MR1949349}.
As for the Tzitzeica equation,
an integrable discrete model was proposed in \cite{MR1676596},
which can be written into the trilinear equation
in terms of the $\tau$ function.
Their approaches in finding the discrete equations belong to
the theory of discrete differential geometry (DDG),
which investigates the geometric objects that are described by
integrable partial difference equations,
refer to \cite{MR2467378} for a comprehensive introduction to DDG.
We expect that investigating discrete objects may offer
a better understanding way of smooth objects,
and as a consequence of it,
DDG can be applied to practical use
in architecture,
computer vision,
operations research and so on.
See, for instance, \cite{MR2335142},
\cite{MR2486030} and \cite{MR3247006}.

The interrelations existing between integrable systems
and geometry are described by the Gauss-Weingarten formula,
because the moving frames of surfaces give the Lax pairs
of soliton equations.
Besides that,
it is the point that we are able to introduce a natural parameter into
the Lax pair, which is often called the spectral parameter.
Thus we can investigate surfaces from a view point of the loop group theory,
which also helps us in deriving discrete counterparts of surfaces.
Indeed,
on the indefinite affine spheres,
the loop group discretization method
has been demonstrated in \cite{MR1676596}.
Further,
we can make a use of the Birkhoff decomposition of loop groups
so as to give construction methods for special classes of discrete surfaces.
For example,
discrete counterparts for the surfaces with constant negative curvature
can be defined and constructed
via loop group method (\cite{MR97a:39022}, \cite{MR3661533}),
where a discrete analogue of separation of variables for
sine-Gordon equation (\cite{MR581396}) is presented.

In this paper,
we give a construction method for discrete indefinite affine spheres
by using a loop group method.
In particular we show that a discrete indefinite improper affine sphere
can be constructed from two discrete plane curves.
The paper is organized as follows:
in Section \ref{sc:smooth},
after explaining some basic notions of affine differential geometry,
we prepare loop groups associated with affine spheres.
We close the section by rephrasing the representation formula
by Blaschke for improper affine spheres
and illustrating some examples that may have singularities.
In Section \ref{sc:discrete},
we discretize the representation formula,
so that the discrete improper affine spheres,
which may have singularities,
are constructed from two planar discrete curves.

\section{Indefinite affine spheres}\label{sc:smooth}

\subsection{Preliminaries}\label{subsec:prelim}

Let $f$ be an immersion from a domain $\D \subset \R^2$
to the affine space $\left(\R^3, \det\right)$.
Here we use determinant function as a fixed volume element on $\R^3$.
Let $\xi$ be an transversal vector field to $f$,
that is,
for each $\left(x, y\right) \in \D$
the vector $\xi \left(x, y\right)$ never tangent
to the surface $f \left(\D\right)$.
A symmetric bilinear function $h = \left[h_{i j}\right]$ is defined
by the Gauss formula
\begin{align*}
\pd{x}^2 f &= w_{11}^1 \pd{x} f + w_{11}^2 \pd{y} f + h_{11} \xi,\\
\pd{y} \pd{x} f
&= w_{12}^1 \pd{x} f + w_{12}^2 \pd{y} f + h_{12} \xi,\\
\pd{y}^2 f &= w_{22}^1 \pd{x} f + w_{22}^2 \pd{y} f + h_{22} \xi,
\end{align*}
where $\pd{x} = \partial/\partial x$ and $\pd{y} = \partial/\partial y$.
It is easy to check that the rank of $h$ is independent of
the choice of $\xi$.
If the rank of $h$ is $2$,
$h$ can be treated as a nondegenerate metric on $\D$.
This is the basic assumption on which Blaschke \cite{blaschkebook2}
developed the affine differential geometry of surfaces.
We can define canonical transversal vector field
by the properties that
the induced volume element on $\D$ coincides with the volume element
of the affine metric $h$,
namely
\begin{equation}\label{volumecondition}
{\det \left[\pd{x} f, \pd{y} f, \xi\right]}^2
= \left|h_{11} h_{22} - \left(h_{12}\right)^2\right|,
\end{equation}
and both $\pd{x}\xi$ and $\pd{y} \xi$ are tangent to $f \left(\D\right)$.
Such a $\xi$ is unique up to sign,
and is called the \textit{affine normal} field.
The immersion $f$ with the affine normal field $\xi$ is called
the \textit{Blaschke immersion},
and the map $\tilde F\colon \D \to \mathrm{GL}_3 \R$,
$\left(x, y\right) \mapsto \left[\pd{x}f, \pd{y}f, \xi\right]$ is
called a \textit{moving frame} of $f$.
It is known that,
for a Blaschke immersion $f$,
half of the Laplacian $(1/2) \Delta f$ relative to the affine metric $h$
is equal to the affine normal field $\xi$.

For a Blaschke immersion $f$,
the affine shape operator $s = \left[s_{i j}\right]$ is defined by
the Weingarten formula
\begin{align*}
\pd{x} \xi &= - s_{11} \pd{x} f - s_{21} \pd{y} f,\\
\pd{y} \xi &= - s_{12} \pd{x} f - s_{22} \pd{y} f.
\end{align*}
If the affine shape operator $s$ is proportional to the  identity,
that is $s = H \id$,
then the Blaschke immersion $f$ is called an \textit{affine sphere}.
By virtue of the integrability condition,
this function $H$ should be a constant.
If $H=0$ then $f$ is called an \textit{improper} affine sphere,
and if $H \neq 0$ then $f$ is called a \textit{proper} affine sphere.
On use of a scaling transformation of the ambient space
and a change of the orientation $f\mapsto -f$,
we can normalize the constant $H$ to be $- 1$ if $H \neq 0$.
For example,
a graph immersion
\begin{equation*}
f \left(x, y\right) =
\begin{bmatrix}
x\\
y\\
\psi \left(x, y\right)
\end{bmatrix}
\end{equation*}
is an improper affine sphere with the constant affine normal field
$\xi \left(x, y\right) = {}^\mathrm{t} \left[0, 0, 1\right]$
if and only if $\psi$ satisfies the
Monge-Amp\`{e}re equation
\begin{equation}\label{monge-ampere:without-boundary}
\left(\pd{x}^2 \psi\right)
\left(\pd{y}^2 \psi\right)
- \left(\pd{x} \pd{y} \psi\right)^2
= \pm 1.
\end{equation}
In general,
if $f$ is an improper affine sphere,
then the affine normals are parallel in $\R^3$.
If $f$ is a proper affine sphere,
then the affine normals meet at one point in $\R^3$,
which is called the center.

Let $f$ be an affine sphere whose affine metric $h$ has signature
$\left(+, -\right)$.
We call such an $f$ \textit{indefinite affine sphere} in short.
The affine normal field may be expressed as
\begin{equation*}
\xi = - H f + \left(1 + H\right) \xi_0,
\end{equation*}
where $H \in \left\{- 1, 0\right\}$ and
$\xi_0$ is a constant vector.
By an appropriate affine transformation on $\R^3$
we can fix $\xi_0$ to be ${}^\mathrm{t} \left[0, 0, 1\right]$.
We shall employ the asymptotic coordinate systems
with respect to $h$,
and substitute the symbols $\left(x, y\right)$ with $\left(u, v\right)$.
As far as $\det \tilde{F} \neq 0$,
without loss of generality
we can assume that $\det \tilde{F} > 0$,
and define three functions $\omega$, $A$, $B$ by
\begin{equation*}
\omega = \det \tilde F,\quad
A = \det \left[\pd{u} f,\, \pd{u}^2 f,\, \xi\right],\quad
B = \det \left[\pd{v}^2 f,\, \pd{v} f,\, \xi\right].
\end{equation*}
%see for example \cite[pp.~122, 211]{blaschkebook2}.
We rewrite the Gauss-Weingarten formulas as
\begin{equation}\label{eq:FUVtilde}
\pd{u}\tilde F = \tilde F
\begin{bmatrix}
\pd{u} \log \omega & 0 & - H\\
A\, \omega^{-1} & 0 &0\\
0 & \omega & 0
\end{bmatrix},\quad
\pd{v}\tilde F =  \tilde F
\begin{bmatrix}
0 & B \omega^{-1} & 0\\
0 & \pd{v} \log \omega & -H\\
\omega  & 0 & 0
\end{bmatrix}.
\end{equation}
The compatibility condition between these two equations,
namely $\pd{u} \pd{v} \tilde F = \pd{v} \pd{u} \tilde F$,
is given by the three partial differential equations
\begin{gather}
\pd{v}\pd{u}{\log \omega}+ A B \omega^{-2} +H \omega
=\label{tzitzeica-liouville}
0,\\
\pd{v} A = 0,\quad
\pd{u} B = 0.\label{tzitzeica-liouville:aux}
\end{gather}
The equations in \eqref{tzitzeica-liouville:aux} are clearly solved as
$A = A \left(u\right)$ and $B = B \left(v\right)$ respectively.
Equation \eqref{tzitzeica-liouville} is called the \textit{Tzitzeica equation}
if $H = - 1$,
and the \textit{Liouville equation} if $H = 0$.
It is known that general solutions to the Liouville equation
are given by two real functions of one variable,
see the formula \eqref{eq:solution} in Remark \ref{rem:generalsol-liouville}.

Conversely,
a triad $\left(\omega, A, B\right)$ of solutions to
the system \eqref{tzitzeica-liouville}--\eqref{tzitzeica-liouville:aux},
or in other words,
a pair of the affine metric $h = 2 \omega\, du dv$ and
the cubic form $C = A\, du^3 + B\, dv^3$,
gives a unique indefinite affine sphere up to equiaffine transformations.
Since the system \eqref{tzitzeica-liouville}--\eqref{tzitzeica-liouville:aux}
is invariant under a transformation
\begin{equation*}
A \mapsto \lambda^3 A,\quad
B \mapsto \lambda^{-3} B,\quad
\lambda \in \R^{\times} = \R \setminus \left\{0\right\},
\end{equation*}
if $\left(\omega, A, B\right)$ is a triad of solutions
to \eqref{tzitzeica-liouville}--\eqref{tzitzeica-liouville:aux}
then $\left(\omega,\, \lambda^3 A,\, \lambda^{-3} B\right)$ is also a triad
of solutions to the same system.
Therefore,
there exists a family of indefinite affine spheres
$\left\{f^\lambda\right\}$ that is parametrized by $\lambda \in \R^{\times}$,
and $f^1$ is the original affine sphere $f$.
This $1$-parameter family of indefinite affine spheres,
which we call the \textit{associated family} of $f$,
has the property that they have the same affine metric and
the same constant affine mean curvature $(1/2) \tr{s} = H$.

\subsection{Loop group description}

We define a \textit{gauged frame} $F$ of $f^\lambda$ by
\begin{equation}\label{def:extended}
F = \left[\pd{u}f^\lambda,\, \pd{v}f^\lambda,\, \xi^\lambda\right]
\diag \left(\lambda^{-1} \omega^{-1/2},\, \lambda\, \omega^{-1/2},\, 1\right),
\end{equation}
where $\xi^\lambda =-H f^\lambda + \left(1 + H\right) \xi_0$.
For any $\left(\left(u, v\right), \lambda\right) \in \D \times \R^{\times}$,
the frame $F$ takes values in the special linear group $\SL$ and satisfies
the partial differential equations
\begin{equation}\label{PDE:extended}
\pd{u} F = F U,\quad
\pd{v} F = F V,
\end{equation}
where
\begin{equation}\label{PDE:extended-UV}
\begin{split}
U &=
\begin{bmatrix}
\left(1/2\right) \pd{u} \log \omega & 0 & - \lambda H \omega^{1/2}\\
\lambda A \omega^{-1} &  - \left(1/2\right) \pd{u} \log \omega &0\\
0 & \lambda \omega^{1/2} & 0
\end{bmatrix},\\
V &=
\begin{bmatrix}
- \left(1/2\right) \pd{v} \log \omega & \lambda^{-1}B \omega^{-1} & 0\\
0  & \left(1/2\right) \pd{v} \log \omega & - \lambda^{-1} H \omega^{1/2}\\
\lambda^{-1} \omega^{1/2} & 0 & 0
\end{bmatrix}.
\end{split}
\end{equation}
By multiplying $F$ by some constant matrix from the left if necessary,
we can assume that
\begin{equation}\label{eq:initial}
F \left(0, 0, \lambda\right) = \id
\end{equation}
at the base point $(u, v)=(0, 0)$.
The gauged frame $F$ which satisfies the system
\eqref{PDE:extended}--\eqref{PDE:extended-UV}
with initial condition \eqref{eq:initial}
will be called the \textit{extended frame}
of an indefinite affine sphere $f$.

Moreover one can check that the matrices $U = U \left(\lambda\right)$ and
$V = V \left(\lambda\right)$ in \eqref{PDE:extended-UV} satisfy
\begin{align}
- {}^\mathrm{t} U \left(-\lambda\right) T
&=\label{twisted-T-liealg}
T\, U \left(\lambda\right),\quad
- {}^\mathrm{t} V \left(-\lambda\right) T
=T\, V \left(\lambda\right),\\
U \left(q \lambda\right)
&=\label{twisted-Q-liealg}
Q\, U \left(\lambda\right) Q^{-1},\quad
V \left(q \lambda\right)
=
Q\, V \left(\lambda\right) Q^{-1},
\end{align}
where
\begin{equation*}
T =
\begin{bmatrix}
0 & 1 & 0\\
1 & 0 & 0\\
0 & 0 & -H
\end{bmatrix},\quad
Q = \diag \left(q, q^2, 1\right),\quad
q = e^{2 \pi \sqrt{-1}/3}.
\end{equation*}
Therefore $F$ must satisfy
\begin{align}
{{}^\mathrm{t} F \left(-\lambda\right)}^{-1} T
&=\label{twisted-T}
T F \left(\lambda\right),\\
F \left(q \lambda\right)
&=\label{twisted-Q}
Q F \left(\lambda\right) Q^{-1},
\end{align}
and hence the loop algebra and the loop group
can be introduced (\cite{MR1873045}) as
\begin{align*}
\Lsl &= \left\{\Phi\colon \mathbb{S}^1 \to \slc \;\left|\;
\overline{\Phi \left(\overline{\lambda}\right)}
= \Phi \left(\lambda\right),\ %
- {}^\mathrm{t} \Phi \left(-\lambda\right) T
= T\, \Phi \left(\lambda\right),\ %
\Phi \left(q \lambda\right)
= Q\, \Phi \left(\lambda\right) Q^{-1}\right.\right\},\\
\LSL &= \left\{\phi\colon \mathbb{S}^1 \to \SLC \;\left|\;
\overline{\phi \left(\overline{\lambda}\right)}
= \phi \left(\lambda\right),\ %
{{}^\mathrm{t} \phi \left(-\lambda\right)}^{-1} T
= T\, \phi \left(\lambda\right),\ %
\phi \left(q \lambda\right)
= Q\, \phi \left(\lambda\right) Q^{-1}\right.\right\}.
\end{align*}
Here the overlines mean complex conjugate,
and $\slc$ is the Lie algebra of $\SLC$,
that is, $\slc$ is the set of trace-free matrices.
It should be noted that the extended frame $F$ is $\LSL$-valued
function on $\D$,
because $F$,
which is originally defined on $\lambda \in \R^\times$,
can be analytically extended to $\myC^{\times}$.
The subgroups
\begin{align*}
\LSLP &= \left\{\phi \in \LSL \;\big|\;
\phi \left(\lambda\right) =
\textstyle\sum_{k = 0}^\infty \lambda^k \phi_k\right\},\\
\LSLN &= \left\{\phi \in \LSL \;\big|\;
\phi \left(\lambda\right) =
\textstyle\sum_{k = 0}^\infty \lambda^{-k} \phi_k\right\}
\end{align*}
will play important roles in the following discussions,
together with
\begin{align*}
\LSLPN &= \left\{\phi \in \LSLP \;\big|\;
\phi \left(\lambda\right) = \id +
\textstyle\sum_{k = 1}^\infty \lambda^k \phi_k\right\},\\
\LSLNN &= \left\{\phi \in \LSLN \;\big|\;
\phi \left(\lambda\right) = \id +
\textstyle\sum_{k = 1}^\infty \lambda^{-k} \phi_k\right\}.
\end{align*}
Similarly subalgebras $\LslP$, $\LslN$ and $\LslPN$, $\LslNN$ are defined.
We note that $\Phi_\pm \in \Lambda^\pm_\ast \mysl$ have the expansion
$\Phi_\pm \left(\lambda\right) = \sum_{k=1}^\infty \lambda^{\pm k} \Phi_k$,
where double signs correspond.
\begin{Proposition}\label{prop:twisted-Q}
Let $F \left(\lambda\right) = \sum_{k=-\infty}^\infty \lambda^k F_k
\in \LSLC$ satisfy the twisted condition \eqref{twisted-Q}.
Then coefficient matrices of $F \left(\lambda\right)$ are of the form
\begin{equation*}
F_{3l} =\diag \left(\ast, \ast, \ast\right),\quad
F_{3l+1} =
\begin{bmatrix}
0 & 0 & \ast\\
\ast & 0 & 0\\
0 & \ast & 0
\end{bmatrix},\quad
F_{3l+2} =
\begin{bmatrix}
0 & \ast & 0\\
0 & 0 & \ast\\
\ast & 0 & 0
\end{bmatrix}
\end{equation*}
for all integers $l$.
\end{Proposition}
\begin{proof}
Because $Q F \left(\lambda\right) Q^{-1}
= \sum_{k \in \Z} \lambda^k Q F_k Q^{-1}$ and
\begin{gather*}
F \left(q \lambda\right)
= \sum_{k \in \Z} \lambda^k q^k F_k
= \sum_{l \in \Z}
\left(
\lambda^{3l} F_{3l}
+ \lambda^{3l+1} q F_{3l+1}
+ \lambda^{3l+2} q^2 F_{3l+2}
\right),
\end{gather*}
we have $F_{3l} = Q F_{3l} Q^{-1}$,
$q F_{3l+1} = Q F_{3l+1} Q^{-1}$,
and $q^2 F_{3l+2} = Q F_{3l+2} Q^{-1}$.
\end{proof}
We now recall Birkhoff decomposition theorem for the loop group $\LSL$.
\begin{Theorem}[Birkhoff decomposition \cite{MR1873045}, \cite{MR900587}]\label{thm:Birkhoff}
 The respective multiplication maps
\begin{equation*}
\LSLPN \times \LSLN \to \LSL
\quad\text{and}\quad
\LSLNN  \times \LSLP \to \LSL
\end{equation*}
are diffeomorphisms onto its images.
Moreover, the images $\LSLPN \cdot \LSLN $
and $\LSLNN  \cdot \LSLP$ are both open and dense in $\LSL$,
which will be called the \emph{big cells}.
\end{Theorem}
Roughly speaking,
Theorem \ref{thm:Birkhoff} says that for almost all $g \in \LSL$,
there uniquely exist pairs
$\left(g_+, g_-\right) \in \LSLPN \times \LSLN$
and $\left(\tilde{g}_-, \tilde{g}_+\right) \in \LSLNN \times \LSLP$
such that
\begin{equation}\label{eq:DecompositionBirkhoff}
g = g_+ g_- = \tilde g_- \tilde g_+.
\end{equation}
The following theorem has been proven
for indefinite proper affine spheres $\left(H = - 1\right)$
in \cite[Proposition 5.2 and Theorems 7.1, 6.1]{MR1873045}.
Here we show a proof which is valid for both $H=-1$ or $H=0$.
\begin{Theorem}\label{thm:Rep}
Let $f$ be an indefinite affine sphere,
and $\left(u, v\right) \in \D$ be its asymptotic coordinates.
Consider the Birkhoff decompositions for
the extended frame $F$
near $\left(u, v\right) = \left(0, 0\right)$ as
\begin{equation}\label{def:F+,F-,G-,G+}
F = F_+ F_- = G_- G_+,
\end{equation}
where $F_+ \in \LSLPN$, $F_- \in \LSLN$, $G_+ \in \LSLP$
and $G_- \in \LSLNN$.
Then $F_+$ and $G_-$ do not depend on $v$ and $u$ respectively,
and their Maurer-Cartan forms are given as
\begin{equation}\label{ODE:F+andG-}
F_+^{-1} d F_+ = \xi_+,\quad
G_-^{-1} d G_- = \xi_-,
\end{equation}
where
\begin{equation}\label{eq:pot}
\xi_+ = \lambda
\begin{bmatrix}
0 & 0 & -H \alpha\\
\beta & 0 & 0\\
0 & \alpha & 0
\end{bmatrix}
du,\quad
\xi_- = \lambda^{-1}
\begin{bmatrix}
0 & \sigma & 0\\
0 & 0 & -H \rho\\
\rho & 0 & 0
\end{bmatrix}
dv.
\end{equation}
Here the functions $\alpha$, $\beta$ depend only on $u$,
and $\rho$, $\sigma$ only on $v$.
Moreover,
$\alpha$ and $\rho$ have no zeros near the base point
$\left(u, v\right) = \left(0, 0\right)$.
% and the coefficients of the affine cubic differential
% $C = A\, du^3 + B\, dv^3$ are given as
% $A = \alpha^2 \beta$ and $B = \rho^2 \sigma$.

Conversely,
let $\left(\xi_+, \xi_-\right)$ be a pair of $1$-forms as \eqref{eq:pot},
and $\left(F_+, G_-\right)$ be a pair of solutions to
the linear ordinary differential equations
\begin{equation}\label{eq:ODEs}
d F_+ = F_+ \xi_+,\quad
d G_- = G_- \xi_-
\end{equation}
with the initial condition $F_+ \left(0, \lambda\right)
= G_- \left(0, \lambda\right) = \id$.
Define $V_+ \in \LSLPN$ and $V_- \in \LSLN$ by
the Birkhoff decomposition for  $G_-^{-1} F_+$ near $(u, v) =(0, 0)$ as
\begin{equation}\label{def:V+V-}
G_-^{-1} F_+ = V_+ V_-^{-1},
\end{equation}
and write $\hat{F} = F_+ V_- = G_- V_+$.
Then there exists a diagonal matrix
$D = \diag \left(d,\, d^{-1}, 1\right)$ with some non-vanishing
function $d = d \left(u, v\right)$ such that $D_0{}^{-1} \hat{F} D$,
where $D_0 = D|_{(u, v)=(0, 0)}$,
is the extended frame of an indefinite affine sphere $f$
with the cubic differential
$C= \alpha^2 \beta\, du^3 + \rho^2 \sigma\, dv^3$.
In particular,
in case of proper affine spheres $\left(H = - 1\right)$,
the third column of $D_0{}^{-1} \hat{F} D$ directly gives
the position vector of $f$.
\end{Theorem}
\begin{Remark}
The pair of $1$-forms defined in \eqref{eq:pot} will be called the pair of
\textit{normalized potentials} for an indefinite affine sphere.
It should be noted that the resulting indefinite affine sphere
which is constructed from a pair of normalized potentials
would have singularities
where $G_-^{-1} F_+$ is outside of the big-cell
for the Birkhoff decomposition.
\end{Remark}

\begin{proof}
Let $F$ be an extended frame,
and define $F_+$ and $F_-$ by \eqref{def:F+,F-,G-,G+}.
Therefore we have $F_+ = F F_-^{-1}$ and so that
\begin{equation*}
F_+^{-1} \pd{v}F_+
= F_- F^{-1} \pd{v} \left(F F_-^{-1}\right)
= \left(F_- V - \pd{v}F_-\right) F_-^{-1},
\end{equation*}
where $V$ is given by \eqref{PDE:extended-UV}.
Since $V$ takes values in $\LslN$,
the right-hand side takes values in it.
Moreover since $F_+$ takes values in $\LSLPN$,
the left-hand side takes values in $\LslPN$.
Thus we have $F_+^{-1} \pd{v}F_+ = 0$,
which shows that $F_+$ does not depend on $v$.
Similarly $\pd{u} G_- = 0$.

Next,
we compute $F_+^{-1} d F_+$ and $G_-^{-1} d G_-$.
We have
\begin{equation*}
F_+^{-1} d F_+
= F_- F^{-1} d \left(F F_-^{-1}\right)
= \left(F_- U - \pd{u} F_-\right) F_-^{-1} du
\end{equation*}
where $U$ is given by \eqref{PDE:extended-UV}.
Since $U$ has the form $U = U^0 + \lambda U^1$
and $F_-$ takes values in $\LSLN$,
we have
\begin{equation*}
\xi_+ = F_+^{-1} d F_+ = \left(X^0 + \lambda X^1\right) du.
\end{equation*}
Here $X^0 = 0$ because $\xi_+$ should be a $\LslPN$-valued $1$-form.
The twisted condition \eqref{twisted-Q-liealg} implies that
$X^1$ have the form
\begin{equation*}
% X^0 = \diag \left(x_{11}, x_{22}, x_{33}\right),\quad
X^1 =
\begin{bmatrix}
0 & 0 & x_{13}\\
x_{21} & 0 & 0\\
0 & x_{32} & 0
\end{bmatrix},
\end{equation*}
where $x_{ij}$ are some functions in $u$.
Further the twisted condition \eqref{twisted-T-liealg} implies that
$x_{13} = - H x_{32}$.
Thus we have \eqref{eq:pot} on setting $\alpha = x_{32}$ and
$\beta = x_{21}$.
If we write $F_-$ and its inverse as
\begin{equation*}
F_- = I^0 + \lambda^{-1} I^1 + \cdots,\quad
F_-^{-1} = J^0 + \lambda^{-1} J^1 + \cdots,
\end{equation*}
where
$\id = F_- F_-^{-1} =
I^0 J^0 + \lambda^{-1} \left(I^0 J^1 + I^1 J^0\right) + \cdots$,
then we in particular have $I^0 J^0 = \id$.
Further,
from the twisted conditions \eqref{twisted-Q} and \eqref{twisted-T},
it follows that
\begin{equation*}
I^0 = \diag \left(i,\, i^{-1},\, 1\right),\quad
J^0 = \diag \left(i^{-1},\, i,\, 1\right),
\end{equation*}
where $i$ is some function in $\left(u, v\right)$ with no zeros.
Noticing that $\lambda X^1 = \left(F_- U - \pd{u} F_-\right) F_-^{-1}$,
it is easy to see that $X^1$ is computed as
\begin{align*}
% X^0 &= I^0 U^1 J^1 + I^1 U^1 J^0
% + I^0 U^0 J^0
% - \left(\pd{u} I^0\right) J^0,\\
X^1 &= I^0 U^1 J^0 =
\begin{bmatrix}
0 & 0 & -H\, i\, \omega^{1/2}\\
A\, i^{-2} \omega^{-1} & 0 & 0\\
0 & i\, \omega^{1/2} & 0
\end{bmatrix},
\end{align*}
which shows that $\alpha = x_{32} = i\, \omega^{1/2}$ has no zeros.
Similarly we can show that $\rho$ has no zeros.

Conversely let $F_+$ and $G_-$ be the solutions of \eqref{eq:ODEs}
with initial condition $F_+ \left(0, \lambda\right)
= G_- \left(0, \lambda\right) =\id$ and consider the
Birkhoff decomposition near $\left(u, v\right) = \left(0, 0\right)$ as
\eqref{def:V+V-} with $V_+ \in \LSLPN$ and $V_- \in \LSLN$.
Then the Maurer-Cartan form of
$\hat{F} = F_+ V_- = G_- V_+$ is computed as
\begin{gather}
\hat{F}^{-1} d\hat{F}
=\label{eq:MC-Fhat-1}
\left(F_+ V_-\right)^{-1} d \left(F_+ V_-\right)
= V_-^{-1} \left(\xi_+ V_- + d V_-\right),\\
\hat{F}^{-1} d\hat{F}
=\label{eq:MC-Fhat-2}
\left(G_- V_+\right)^{-1} d \left(G_- V_+\right)
= V_+^{-1} \left(\xi_- V_+ + d V_+\right).
\end{gather}
We write
\begin{equation*}
V_- = K_0 + \lambda^{-1} K_1 + \cdots,\quad
V_-^{-1} = L_0 + \lambda^{-1} L_1 + \cdots
\end{equation*}
with the matrices
\begin{equation*}
K_0 = \diag \left(k, k^{-1}, 1\right),\quad
L_0 = \diag \left(k^{-1}, k, 1\right),
\end{equation*}
where $k$ is some function in $\left(u, v\right)$ which has no zeros.
%Moreover,
%from the twisted condition \eqref{twisted-T},
%$l$ satisfies that $\pd{v} l = 0$,
%and further,
%if $H=-1$ then $l^2 = 1$.
Noticing that $V_+$ is $\LSLPN$-valued,
it follows from \eqref{eq:MC-Fhat-1}--\eqref{eq:MC-Fhat-2} that
$\hat{F}^{-1} d\hat{F}$ is given by
\begin{align*}
\hat{F}^{-1} \pd{u} \hat{F} &= \lambda\, L_0
\begin{bmatrix}
0 & 0 & - H \alpha\\
\beta & 0 & 0\\
0 & \alpha & 0
\end{bmatrix}
K_0,\\
%= \hat{U},\\
\hat{F}^{-1} \pd{v} \hat{F} &= \lambda^{-1}
\begin{bmatrix}
0 & \sigma & 0\\
0 & 0 & - H \rho\\
\rho & 0 & 0
\end{bmatrix}
+ L_0\, \pd{v} K_0.
%= \hat{V}.
\end{align*}
We introduce a gauge $D = \diag \left(d, d^{-1}, 1\right)$,
then $F = \hat{F} D$ satisfies that
\begin{align*}
F^{-1} \pd{u} F
% &= D^{-1} \big(\hat{U} D + \pd{u} D\big)\\
&= \lambda
\begin{bmatrix}
0 & 0 & - H \alpha k^{-1} d^{-1}\\
\beta k^2 d^2 & 0 & 0\\
0 & \alpha k^{-1} d^{-1} & 0
\end{bmatrix}
+
\begin{bmatrix}
d^{-1} \pd{u} d & 0 & 0\\
0 & - d^{-1} \pd{u} d & 0\\
0 & 0 & 0
\end{bmatrix},\\
F^{-1} \pd{v} F
% &= D^{-1} \big(\hat{V} D + \pd{v} D\big)\\
&= \lambda^{-1}
\begin{bmatrix}
0 & \sigma d^{-2} & 0\\
0 & 0 & - H \rho d\\
\rho d & 0 & 0
\end{bmatrix}
+
\begin{bmatrix}
d^{-1} \pd{v} d + k^{-1} \pd{v} k & 0 & 0\\
0 & - d^{-1} \pd{v} d - k^{-1} \pd{v} k & 0\\
0 & 0 & 0
\end{bmatrix}.
\end{align*}
We define $\omega$, $A$, $B$ by
\begin{equation*}
\omega = \frac{\alpha \rho}{k},\quad
A = \alpha^2 \beta,\quad
B = \rho^2 \sigma.
\end{equation*}
If necessary changing $u \to - u$ and/or $v \to -v$, we can assume
$\omega>0$, and choose $d = \rho^{-1} \omega^{1/2}$.
It is easy to check that these matrices $F^{-1} \pd{u} F$ and
$F^{-1} \pd{v} F$ coincide with \eqref{PDE:extended-UV}.
Thus $D_0{}^{-1} F$, where $D_0 = D|_{(u, v)=(0, 0)}$,
satisfies $D_0{}^{-1} F \left(0,0,\lambda\right) = \id$,
and hence is the extended frame of some indefinite affine sphere.
\end{proof}

\subsection{Indefinite improper affine spheres}

When $H = 0$,
there exists an integral formula in terms of four
functions of one variable,
which is known as the \textit{Blaschke representation}.
We first show a fundamental lemma.
\begin{Lemma}\label{lem:solution}
Let $H=0$.
Then the pair of solutions $\left(F_+, G_-\right)$ to the system
\eqref{ODE:F+andG-} and \eqref{eq:pot} with the initial condition
$F_+ \left(0, \lambda\right)= G_- \left(0, \lambda\right) = \id$ is
explicitly given by
\begin{equation}\label{eq:ESol}
F_+ =
\begin{bmatrix}
1 & 0 & 0\\
\lambda b & 1 & 0\\
\lambda^2 c & \lambda a & 1
\end{bmatrix},\quad
G_- =
\begin{bmatrix}
1 & \lambda^{-1} s & 0\\
0 & 1  & 0\\
\lambda^{-1} r  & \lambda^{-2} t & 1
\end{bmatrix},
\end{equation}
where $a$, $b$, $c$, $r$, $s$, $t$ are defined as
\begin{gather}
a \left(u\right) = \int_{0}^u \alpha(k)\, d k, \quad
b \left(u\right) = \int_0^u \beta(k) \, d k, \quad
c \left(u\right) = \int_{0}^u a(k) \beta(k)\, dk,\label{def:abc}\\
r \left(v\right) = \int_0^v \rho (k) \, d k,\quad
s \left(v\right) = \int_0^v \sigma (k) \, d k,\quad
t \left(v\right) = \int_{0}^v r(k) \sigma(k)\, dk.\label{def:rst}
\end{gather}
Moreover,
if $1 - b s\neq 0$,
then $V_+ \in \LSLPN$ and $V_- \in \LSLN$,
defined by the Birkhoff decomposition of $G_-^{-1} F_+$ as \eqref{def:V+V-},
are given as
\begin{align}
V_+ &=\label{eq:VPM:P}
\begin{bmatrix}
1 & 0 & 0\\
\lambda b \left(1 - b s\right)^{-1} & 1 & 0\\
\lambda^2 c \left(1 - b s\right)^{-1} &
\lambda \left(a \left(1-bs\right) + c s\right) & 1
\end{bmatrix},\\
V_- &=\label{eq:VPM:M}
\begin{bmatrix}
\left(1 - b s\right)^{-1} & \lambda^{-1} s & 0\\
0 & 1 - b s & 0\\
\lambda^{-1} \big(r + b t \left(1 - b s\right)^{-1}\big) &
\lambda^{-2} t & 1
\end{bmatrix}.
\end{align}
\end{Lemma}
\begin{proof}
It is easy to check that $F_+$, $G_-$
in \eqref{eq:ESol} satisfy \eqref{ODE:F+andG-} and \eqref{eq:pot}
with $H=0$,
and $F_+ \left(0, \lambda\right) = G_- \left(0, \lambda\right) = \id$.
The loops $V_+$, $V_-$ in \eqref{eq:VPM:P}, \eqref{eq:VPM:M} clearly
belong to $\LSLPN$, $\LSLN$ respectively.
Because $F_+ V_- = G_- V_+$,
the decomposition \eqref{def:V+V-} holds.
\end{proof}
The condition $1- b s \neq 0$ in Lemma \ref{lem:solution}
means that $G_-^{-1} F_+$ belongs to the big cell
of the Birkhoff decomposition.
\begin{Theorem}
Let $\alpha$, $\beta$, $\rho$, $\sigma$ be functions in one variable,
and define $a$, $b$, $c$, $r$, $s$, $t$ by
\eqref{def:abc} and \eqref{def:rst}.
Let $F_+$, $G_-$, $V_+$, $V_-$ be the loops
given by \eqref{eq:ESol}, \eqref{eq:VPM:P}, \eqref{eq:VPM:M},
and define $\hat{F}$ by $\hat{F} = F_+ V_- = G_- V_+$.
We assume that $\alpha$, $\rho$, $1 - b s$ have no zeros.
Then there exists a diagonal matrix
$D = \diag \left(d, d^{-1}, 1\right)$ with
some function $d$ such that
$D_0{}^{-1} \hat{F} D$,
where $D_0 = D|_{(u, v) = (0, 0)}$,
is the extended frame of some indefinite improper affine sphere $f$.
The data solving the integrability condition
\eqref{tzitzeica-liouville}--\eqref{tzitzeica-liouville:aux}
with $H = 0$ are given as
\begin{equation}\label{explicit:omegaAB}
\omega = \left(1 - b s\right) \alpha \rho,\quad
A = \alpha^2 \beta,\quad
B = \rho^2 \sigma.
\end{equation}
Moreover, the associated family of $f$ is given by
the representation formula
\begin{equation}\label{formula:representation-improper}
f^\lambda =
\begin{bmatrix}
\lambda a + \lambda^{-2} \left(r s - t\right)\\
\lambda^2 \left(a b - c\right) + \lambda^{-1} r\\
a r - \left(a b - c\right) \left(r s - t\right)
+ \lambda^3 \int_0^u \alpha (k) c (k)\, dk
+ \lambda^{-3} \int_0^v \rho (k) t (k)\, dk
\end{bmatrix},
\end{equation}
where $\lambda \in \R^\times$.
All indefinite improper affine spheres are locally constructed in this way.
\end{Theorem}
\begin{proof}
First a straightforward computation shows that
Maurer-Cartan form of $\hat{F}$ is given by
\begin{align*}
{\hat{F}}^{-1} \pd{u}\hat{F}
&= V_+^{- 1} \pd{u} V_+
=
\begin{bmatrix}
0 & 0 & 0\\
\lambda \beta \left(1 - b s\right)^{-2} & 0 & 0\\
0 & \lambda \alpha \left(1 - b s\right) & 0
\end{bmatrix}
= \hat{U},\\
{\hat{F}}^{-1} \pd{v}\hat{F}
&= V_-^{- 1} \pd{v} V_-
=
\begin{bmatrix}
b \sigma \left(1 - b s\right)^{-1} & \lambda^{-1} \sigma & 0\\
0 & - b \sigma \left(1 - b s\right)^{-1} & 0\\
\lambda^{-1} \rho & 0 & 0
\end{bmatrix}
= \hat{V}.
\end{align*}
Next we take a diagonal gauge $D = \diag \left(d, d^{-1}, 1\right)$.
% \eq{
% D = \diag \left(\frac{\alpha}{\sqrt{\left(1 - b s\right) \alpha \rho\,}},\,
% \frac{\sqrt{\left(1 - b s\right) \alpha \rho\,}}{\alpha},\, 1\right)
% }
% Here we adjust the sign of $\left(1 - b s\right) \alpha \rho$ to be
% positive by change of coordinates if necessary,
% since $1- b s$, $\alpha$, $\rho$ have no zeros.
Then the Maurer-Cartan form of $F= \hat{F} D$ is given by
\begin{align*}
F^{-1} \partial_u F
&= D^{-1} \big(\hat{U} D + \pd{u} D\big)
=
\begin{bmatrix}
\pd{u} \log d & 0 & 0\\
\lambda \beta d^2 \left(1 - b s\right)^{-2} & - \pd{u} \log d & 0\\
0 & \lambda \alpha d^{-1} \left(1 - b s\right) & 0
\end{bmatrix},\\
F^{-1} \partial_v F
&= D^{-1} \big(\hat{V} D + \pd{v} D\big)
=
\begin{bmatrix}
b \sigma \left(1 - b s\right)^{-1} + \pd{v} \log d &
\lambda^{-1} \sigma d^{-2} & 0\\
0 & - b \sigma \left(1 - b s\right)^{-1} - \pd{v} \log d & 0\\
\lambda^{-1} \rho d & 0 & 0
\end{bmatrix}.
\end{align*}
If necessary, changing $u \to - u$ and/or $v \to -v$, we can assume
$\left(1-bs\right) \alpha \rho>0$.
 Then setting \eqref{explicit:omegaAB} and
\begin{equation*}
d = \frac{\sqrt{\left(1 - b s\right) \alpha \rho\,}}{\rho}
= \frac{\omega^{1/2}}{\rho},
\end{equation*}
this system accords with \eqref{PDE:extended}--\eqref{PDE:extended-UV}.
To obtain the representation formula \eqref{formula:representation-improper},
we consider an another diagonal gauge $\tilde{D}= \diag
\big(\lambda\, \omega^{1/2},\, \lambda^{-1} \omega^{1/2},\, 1\big)$ as
introduced in \eqref{def:extended}.
Therefore $\tilde{F} = F \tilde{D}$ satisfies that
\begin{equation*}
\tilde{F}^{-1} \partial_u \tilde{F} =
\begin{bmatrix}
\pd{u} \log \omega &0  &0\\
\lambda^3 A\, \omega^{-1} & 0 & 0\\
0 & \omega & 0
\end{bmatrix},\quad
\tilde{F}^{-1} \partial_v \tilde{F} =
\begin{bmatrix}
0 & \lambda^{-3} B \omega^{-1} &0\\
0 & \pd{v} \log \omega & 0\\
\omega & 0&0
\end{bmatrix}.
\end{equation*}
Thus $\tilde F$ is a family of moving frames of
indefinite improper affine spheres.
The moving frame $\tilde F$ can be computed explicitly as
\begin{equation*}
\tilde F
= G_- V_+ D \tilde{D}
=
\begin{bmatrix}
\lambda \alpha & \lambda^{-2} \rho s & 0\\
\lambda^{2} \alpha b & \lambda^{-1} \rho & 0\\
\left(\lambda^{3} c + b t + r \left(1 - b s\right)\right) \alpha &
\left(\lambda^{-3} t + c s + a \left(1 - b s\right)\right) \rho & 1
\end{bmatrix}.
\end{equation*}
Since the moving frame is defined by $\tilde F
= \left[\partial_u f^\lambda, \partial_v f^\lambda, \xi_0\right]$,
we integrate the first column of $\tilde{F}$ by $u$ and have
\begin{equation*}
f^\lambda =
\begin{bmatrix}
\lambda a\\
\lambda^2 \left(a b - c\right)\\
a r
- \left(r s - t\right) \left(a b - c\right)
+ \lambda^3 \int_0^u \alpha (k) c (k)\, dk
\end{bmatrix}
+
\begin{bmatrix}
x\\
y\\
z
\end{bmatrix},
\end{equation*}
where $x$, $y$, $z$ are some functions in $v$.
Therefore from the second column of $\tilde{F}$,
we have
\begin{equation*}
\begin{bmatrix}
\lambda^{-2} \rho s\\
\lambda^{-1} \rho\\
\left(\lambda^{-3} t + c s + a \left(1 - b s\right)\right) \rho
\end{bmatrix}
= \pd{v} f^\lambda
=
\begin{bmatrix}
x'\\
y'\\
a r' - \left(r s - t\right)' \left(a b - c\right) + z'
\end{bmatrix}.
\end{equation*}
Therefore
\begin{equation*}
x = \lambda^{-2} \left(r s - t\right),\quad
y = \lambda^{-1} r,\quad
z = \lambda^{-3} \int^v_0 \rho (k) t (k)\, dk,
\end{equation*}
which shows \eqref{formula:representation-improper}.
\end{proof}
\begin{Remark}\label{rem:generalsol-liouville}
For given functions $A$ and $B$,
it is known that a general solution to
the Liouville equation \eqref{tzitzeica-liouville}
with $H = 0$ is represented as
\begin{equation}\label{eq:solution}
\omega \left(u, v\right)
= \left(\int_0^u \phi\left(k\right) dk
- \int_0^v \psi\left(k\right) dk\right)
\left(- \frac{A \left(u\right) B \left(v\right)}%
{\phi \left(u\right) \psi \left(v\right)}\right)^{1/2},
\end{equation}
where $\phi$ and $\psi$ are arbitrary functions
with no zeros in one variable.
\end{Remark}
\begin{Corollary}[Representation formula]\label{cor:repformula}
Let $\gamma_1$, $\gamma_2$ be plane curves
defined on intervals $I_1$, $I_2$ respectively.
Assume that both the intervals contain $0$.
% \footnote{Deleted the assumption ``and velocities satisfy
% $\left[1, 0\right] \gamma_1' \left(u\right) \neq 0$ and
% $\left[0, 1\right] \gamma_2' \left(v\right) \neq 0$ for all $u$ and $v$.''}
Then the map
\begin{equation}\label{formula:improper-AF-from-planecurves}
f \left(u, v\right)
=
\begin{bmatrix}
\gamma_1 \left(u\right) + \gamma_2 \left(v\right)\\
z \left(u, v\right)
\end{bmatrix},
\end{equation}
where the height $z$ is defined by
\begin{equation}\label{formula:improper-AF-from-planecurves:z}
z \left(u, v\right)
= \det \left[\gamma_1 \left(u\right), \gamma_2 \left(v\right)\right]
+ \int_0^u \det
\left[\gamma_1 \left(k\right), \gamma_1^\prime \left(k\right)\right] dk
- \int_0^v \det
\left[\gamma_2 \left(k\right), \gamma_2^\prime \left(k\right)\right] dk,
\end{equation}
is an indefinite improper affine sphere
with the affine normal ${}^\mathrm{t} \left[0, 0, 1\right]$,
which is parametrized by the asymptotic coordinates
$\left(u, v\right) \in \D = I_1 \times I_2$.
Its affine metric $h = 2 \omega\, du dv$ and
cubic form $C = A\, du^3 + B\, dv^3$ are given by
\begin{equation*}
\omega = \det \left[\gamma_1' \left(u\right),
\gamma_2' \left(v\right)\right],\quad
A = \det \left[\gamma_1' \left(u\right),
\gamma_1'' \left(u\right)\right],\quad
B = \det \left[\gamma_2'' \left(v\right),
\gamma_2' \left(v\right)\right].
\end{equation*}
The singular set of $f$ is
$S = \left\{\left(u, v\right) \in \D \;\big|\;
\det \left[\gamma_1' \left(u\right),
\gamma_2' \left(v\right)\right] = 0\right\}$.
Moreover the associated family of $f$ is given by the transformation
\begin{equation*}
\gamma_1 \mapsto
\begin{bmatrix}
\lambda & 0\\
0 & \lambda^2
\end{bmatrix} \gamma_1,\quad
\gamma_2 \mapsto
\begin{bmatrix}
\lambda^{-2} & 0\\
0 & \lambda^{-1}
\end{bmatrix} \gamma_2
\end{equation*}
where $\lambda \in \R^\times$.
Conversely all indefinite improper affine spheres can be locally
constructed in this way.
\end{Corollary}
\begin{proof}
First, introducing functions $p=ab-c$ and $q= rs -t$,
we rephrase \eqref{formula:representation-improper} as
\begin{equation}\label{formula:representation-improper-2}
f^\lambda =
\begin{bmatrix}
\lambda a + \lambda^{-2} q \\
\lambda^2 p + \lambda^{-1} r\\
a r - p q
+ \lambda^3 \int_0^u \left(a p' - a' p\right) dk
- \lambda^{-3} \int_0^v \left(q r' - q' r\right) dk
\end{bmatrix},
\end{equation}
where we use the identities
\begin{align*}
\alpha c & = a' c + a \left(a b' - c'\right)\\
 & = a^2 b' - ac' + a' c\\
 & = a p' - a' p,
\end{align*}
and $\rho t = q' r-q r'$.
We note that $a\left(0\right)
= p\left(0\right)
= q\left(0\right)
= r\left(0\right) =0$.
We then consider an equiaffine transformation of $f^{\lambda}$
as
\begin{equation*}
\tilde f^{\lambda}=
 \begin{bmatrix}
 1 & 0 & 0 \\
 0 & 1 & 0 \\
\lambda^{-1} r_0 - \lambda^2 p_0 & \lambda a_0 - \lambda^{-2} q_0 & 1
 \end{bmatrix} f^{\lambda}+
\begin{bmatrix}
 \lambda a_0 + \lambda^{-2} q_0 \\
 \lambda^2 p_0 + \lambda^{-1} r_0 \\
 a_0 r_0 - p_0 q_0
\end{bmatrix},
\end{equation*}
where $a_0, r_0, p_0$ and $q_0$ are some constants.
A straightforward computation shows that
\begin{equation*}
\tilde f^\lambda =
\begin{bmatrix}
\lambda \tilde a + \lambda^{-2} \tilde q \\
\lambda^2 \tilde p + \lambda^{-1} \tilde r\\
\tilde a \tilde r - \tilde p \tilde q
+ \lambda^3 \int_0^u
\left(\tilde a \tilde p' - \tilde a' \tilde p\right) dk
- \lambda^{-3} \int_0^v
\left(\tilde q \tilde r' - \tilde q' \tilde r\right) dk
\end{bmatrix},
\end{equation*}
where $\tilde a = a+ a_0$,
$\tilde p = p + p_0$,
$\tilde q = q + q_0$ and $\tilde r = r + r_0$.
Thus we obtain \eqref{formula:improper-AF-from-planecurves}
on writing
\begin{equation*}
\gamma_1 \left(u\right) =
\begin{bmatrix}
\tilde a \left(u\right)\\
\tilde p \left(u\right)
\end{bmatrix},\quad
\gamma_2 \left(v\right) =
\begin{bmatrix}
\tilde q \left(v\right)\\
\tilde r \left(v\right)
\end{bmatrix}.
\end{equation*}
Since $\gamma_1$ and $\gamma_2$ are arbitrary,
\eqref{formula:improper-AF-from-planecurves}
gives the all improper indefinite affine spheres.
\end{proof}
The formula \eqref{formula:representation-improper-2} is
exactly the same that is represented in \cite[p.~216]{blaschkebook2}.
% , by setting
% $U_2 = - \lambda a$, $V_2 = \lambda^{-2} q$, $U_1 = \lambda^2 p$
% and $V_1 = - \lambda^{-1} s$.
In contrast to the Blaschke's proof which utilized the Lelieuvre's formula,
our proof is based on the decomposition of the extended frame.
\begin{Remark}
The representation formula
\eqref{formula:improper-AF-from-planecurves} is also formulated
in \cite{MR2845277} with their concern in computer vision.
They have given a geometric interpretation of the height function
\eqref{formula:improper-AF-from-planecurves:z} as follows.
Consider the curves $2 \gamma_1$ and $2 \gamma_2$,
and fix two points $2 \gamma_1 \left(u\right)$
and $2 \gamma_2 \left(v\right)$ arbitrarily.
We assume both $u$ and $v$ are positive for simplicity,
and denote by $\Omega$
the region enclosed by the union of four curves
\begin{align*}
C_1\colon &
\left[0, v\right] \ni k \mapsto 2 \gamma_2 \left(k\right),\\
C_2\colon &
\left[0, 1\right] \ni k \mapsto 2 \gamma_2 \left(v\right)
+ k \left(2 \gamma_1 \left(u\right) - 2 \gamma_2 \left(v\right)\right),\\
C_3\colon &
\left[0, u\right] \ni k \mapsto 2 \gamma_1 \left(- k + u\right),\\
C_4\colon &
\left[0, 1\right] \ni k \mapsto 2 \gamma_1 \left(0\right)
+ k \left(2 \gamma_2 \left(0\right) - 2 \gamma_1 \left(0\right)\right).
\end{align*}
Then the value $\left|z \left(u, v\right)\right|$ gives
the area of the region $\Omega$.
We will again mention this fact in a simplified case,
see Example \ref{example:circle,square,cpt}.
\end{Remark}
We illustrate some examples of indefinite improper affine spheres
by using the representation formula
\eqref{formula:improper-AF-from-planecurves}.
The resulting surfaces usually have singularities,
and hence are sometimes called \textit{indefinite improper affine maps},
which were introduced in \cite{MR2546487}
for non-convex improper affine surfaces
as an analogue of convex ones \cite{MR2126213}.
\begin{Example}
Let $P$ and $R$ be smooth functions in one variable.
We substitute graphs
\begin{equation*}
\gamma_1 \left(u\right) =
\begin{bmatrix}
u\\
P' \left(u\right)
\end{bmatrix},\quad
\gamma_2 \left(v\right) =
\begin{bmatrix}
R' \left(v\right)\\
v
\end{bmatrix}
\end{equation*}
into the representation formula \eqref{formula:improper-AF-from-planecurves},
and have an indefinite improper affine sphere
\begin{equation}\label{formula:improper-AF-from-planecurves-graph}
f \left(u, v\right) =
\begin{bmatrix}
u + R' \left(v\right)\\
v + P' \left(u\right)\\
\left(u + R' \left(v\right)\right)
\left(v + P' \left(u\right)\right)
- 2 \left(P \left(u\right) + R \left(v\right)
+ P' \left(u\right) R' \left(v\right)\right)
\end{bmatrix}.
\end{equation}
Its data is $\omega =
1 - P'' \left(u\right) R'' \left(v\right)$,
$A = P''' \left(u\right)$,
$B = R''' \left(v\right)$.
It describes a subclass of indefinite improper affine spheres that
may have singularities.
In view of singularity theory
it is known that a cuspidal cross cap,
which is one of the typical singularities
as well as cuspidal edges or swallowtails,
never appear on indefinite improper affine spheres.
See \cite{MR2546487} and \cite{MR2215151} for details.

Especially we set $R = 0$ so that
we have a smooth indefinite improper affine sphere
\begin{equation}\label{Nomizu-Sasaki:Thm5.1}
f \left(u, v\right) =
\begin{bmatrix}
u\\
w\\
u w - 2 P \left(u\right)
\end{bmatrix},
\end{equation}
where $w = v + P' \left(u\right)$.
Further,
the most simplest choice $P = 0$ gives the hyperbolic paraboloid,
or the choice $P \left(u\right) = \left(1/6\right) u^3$
gives the Cayley surface.
It is known that if the affine metric of
an indefinite improper affine sphere is flat
then it is locally of the form \eqref{Nomizu-Sasaki:Thm5.1}.
\end{Example}
\begin{Example}\label{example:circle,square,cpt}
If $\gamma_1$ is the same as $\gamma_2$,
we write them as $\gamma$,
the formula \eqref{formula:improper-AF-from-planecurves} becomes
\begin{gather}
f \left(u, v\right)
=\label{formula:improper-AF-from-one-planecurve}
\begin{bmatrix}
\gamma \left(u\right) + \gamma \left(v\right)\\
z \left(u, v\right)
\end{bmatrix},\\
z \left(u, v\right)
=\label{formula:improper-AF-from-one-planecurve:z}
\det \left[\gamma \left(u\right), \gamma \left(v\right)\right]
+ \int_v^u
\det \left[\gamma \left(k\right), \gamma' \left(k\right)\right] dk.
\end{gather}
It has the data
\begin{equation*}
\omega = \left|\det \left[\gamma' \left(u\right),
\gamma' \left(v\right)\right]\right|,\quad
A = \det \left[\gamma' \left(u\right),
\gamma'' \left(u\right)\right],\quad
B = - \det \left[\gamma' \left(v\right),
\gamma'' \left(v\right)\right].
\end{equation*}
A geometric interpretation of the function
\eqref{formula:improper-AF-from-one-planecurve:z} is given
in \cite{MR2486030},
which is called the \textit{inner area distance} in their language.
Here we briefly explain what it is.
Consider the curve $2 \gamma$,
and fix two points $G_0$ and $G_1$ on its image arbitrarily.
We write
\begin{equation*}
G_0 = 2 \gamma \left(u\right) = 2
\begin{bmatrix}
a \left(u\right)\\
p \left(u\right)
\end{bmatrix},\quad
G_1 = 2 \gamma \left(v\right) = 2
\begin{bmatrix}
a \left(v\right)\\
p \left(v\right)
\end{bmatrix},
\end{equation*}
and assume $u < v$ for simplicity.
We denote by $\Omega$ the region bounded by
the union of two curves,
the arc $C_1\colon \left[u, v\right] \ni k \mapsto
2 \gamma \left(k\right)$
and the line segment $C_2\colon \left[0, 1\right] \ni k \mapsto
G_1 + k \left(G_0 - G_1\right)$.
Then,
by the Green's theorem,
the area of $\Omega \subset \left(\R^2, \left(x, y\right)\right)$
is computed by the line integral
\begin{align*}
\frac{1}{2} \int_{C_1 + C_2} - y\, dx + x\, dy
=&\; \int_u^v \left(- p \left(k\right) a' \left(k\right)
+ a \left(k\right) p' \left(k\right)\right) dk\\
&+ \int_0^1
- \left(p \left(v\right)
+ k \left(p \left(u\right) - p \left(v\right)\right)\right)
\left(a \left(u\right) - a \left(v\right)\right) dk\\
&+ \int_0^1
\left(a \left(v\right)
+ k \left(a \left(u\right) - a \left(v\right)\right)\right)
\left(p \left(u\right) - p \left(v\right)\right) dk\\
% =&\;
% \int_u^v \det \left(\gamma \left(k\right), \gamma' \left(k\right)\right) dk
% - \det \left(\gamma \left(u\right), \gamma \left(v\right)\right)\\
=&\;
- z \left(u, v\right).
\end{align*}
Namely,
the representation formula
\eqref{formula:improper-AF-from-one-planecurve} says that,
at the midpoint of the line segment connecting
$2 \gamma \left(u\right)$ and $2 \gamma \left(v\right)$,
the height $z \left(u, v\right)$ is given by
the signed area of the region $\Omega$.
We also note that it is found in \cite{MR2486030} that,
when $\gamma$ is closed,
we can introduce new variables $x$, $y$ and $\psi$
by the graph expression of
\eqref{formula:improper-AF-from-one-planecurve} as
\begin{equation*}
\begin{bmatrix}
x\\
y\\
\psi \left(x, y\right)
\end{bmatrix}
=
\begin{bmatrix}
\gamma \left(u\right) + \gamma \left(v\right)\\
z \left(u, v\right)
\end{bmatrix}
\end{equation*}
so as to obtain a solution to the Monge-Amp\`{e}re equation
with the Dirichlet boundary condition
\begin{equation}\label{monge-ampere}
\left(\pd{x}^2 \psi\right)
\left(\pd{y}^2 \psi\right)
- \left(\pd{x} \pd{y} \psi\right)^2
= - 1,\quad
\psi\big|_{\partial \Gamma} = 0,
\end{equation}
where $\Gamma$ denotes the region bounded by
the closed curve $2 \gamma$.
% Here the boundary condition is important.
% As we already mentioned at Section \ref{subsec:prelim},
% it is sufficient for obtaining a solution of
% the Monge-Amp\`{e}re equation without a boundary condition
% \eqref{monge-ampere:without-boundary}
% to rewrite the formula \eqref{formula:improper-AF-from-planecurves-graph}
% as graph immersion.
We can readily verify \eqref{monge-ampere} by
a direct computation as follows.
From the definition of new variables
we have that
\begin{align*}
\pd{u} \psi &= \det \left[\gamma \left(u\right) - \gamma \left(v\right),
\gamma' \left(u\right)\right],\\
\pd{v} \psi &= \det \left[\gamma \left(u\right) - \gamma \left(v\right),
\gamma' \left(v\right)\right],
\end{align*}
and the differential relation $\left[\pd{u}, \pd{v}\right]
= \left[\pd{x}, \pd{y}\right]
\left[\gamma' \left(u\right), \gamma' \left(v\right)\right]$.
This implies that
\begin{align*}
\begin{bmatrix}
\pd{y} \psi\\
\pd{x} \psi
\end{bmatrix}
=
\begin{bmatrix}
1 & 0\\
0 & - 1
\end{bmatrix}
\left(\gamma \left(u\right) - \gamma \left(v\right)\right),
\end{align*}
and hence we have the Hesse matrix of $\psi$ as
\begin{align*}
\left[
\pd{y}
\begin{bmatrix}
\pd{y} \psi\\
\pd{x} \psi
\end{bmatrix},
\pd{x}
\begin{bmatrix}
\pd{y} \psi\\
\pd{x} \psi
\end{bmatrix}
\right]
=
\frac{1}{\det M \left(u, v\right)}\,
M \left(u, v\right)\,
{}^\mathrm{t} M \left(v, u\right),
\end{align*}
where
\begin{equation*}
M \left(u, v\right) =
\begin{bmatrix}
1 & 0\\
0 & - 1
\end{bmatrix}
\left[\gamma' \left(u\right), \gamma' \left(v\right)\right].
\end{equation*}
Therefore the determinant of Hesse matrix of $\psi$ is
identically $- 1$ because $\det {}^\mathrm{t} M \left(v, u\right)
= - \det M \left(u, v\right)$.
Thus the formula \eqref{formula:improper-AF-from-one-planecurve} also
provides us with
a construction method of solutions to \eqref{monge-ampere}.
Now we illustrate some examples
by taking several closed curves $\gamma$.
\begin{enumerate}
\item
First one is given by the circle
\begin{equation*}
\gamma \left(u\right) =
\begin{bmatrix}
\cos u\\
\sin u
\end{bmatrix},
\end{equation*}
which leads to
\begin{equation*}
f \left(u, v\right) =
\begin{bmatrix}
\cos u + \cos v\\
\sin u + \sin v\\
u - v - \sin \left(u - v\right)
\end{bmatrix} =
2
\begin{bmatrix}
\cos x \cos y\\
\cos x \sin y\\
x - \cos x \sin x
\end{bmatrix},
\end{equation*}
where $x = \left(u-v\right)/2$ and $y = \left(u+v\right)/2$.
Its data is
\begin{equation*}
\omega = \sin \left(u - v\right),\quad
A = 1,\quad
B = -1.
\end{equation*}
Therefore $f$ has singularities at
$S = \left\{\left.\left(u,v\right) \in \R^2
\;\right|\;
u \equiv v \pmod{\pi}\right\}$.
\begin{figure}[H]\label{fig:exam2-0}
\includegraphics[height=4cm,keepaspectratio]%
{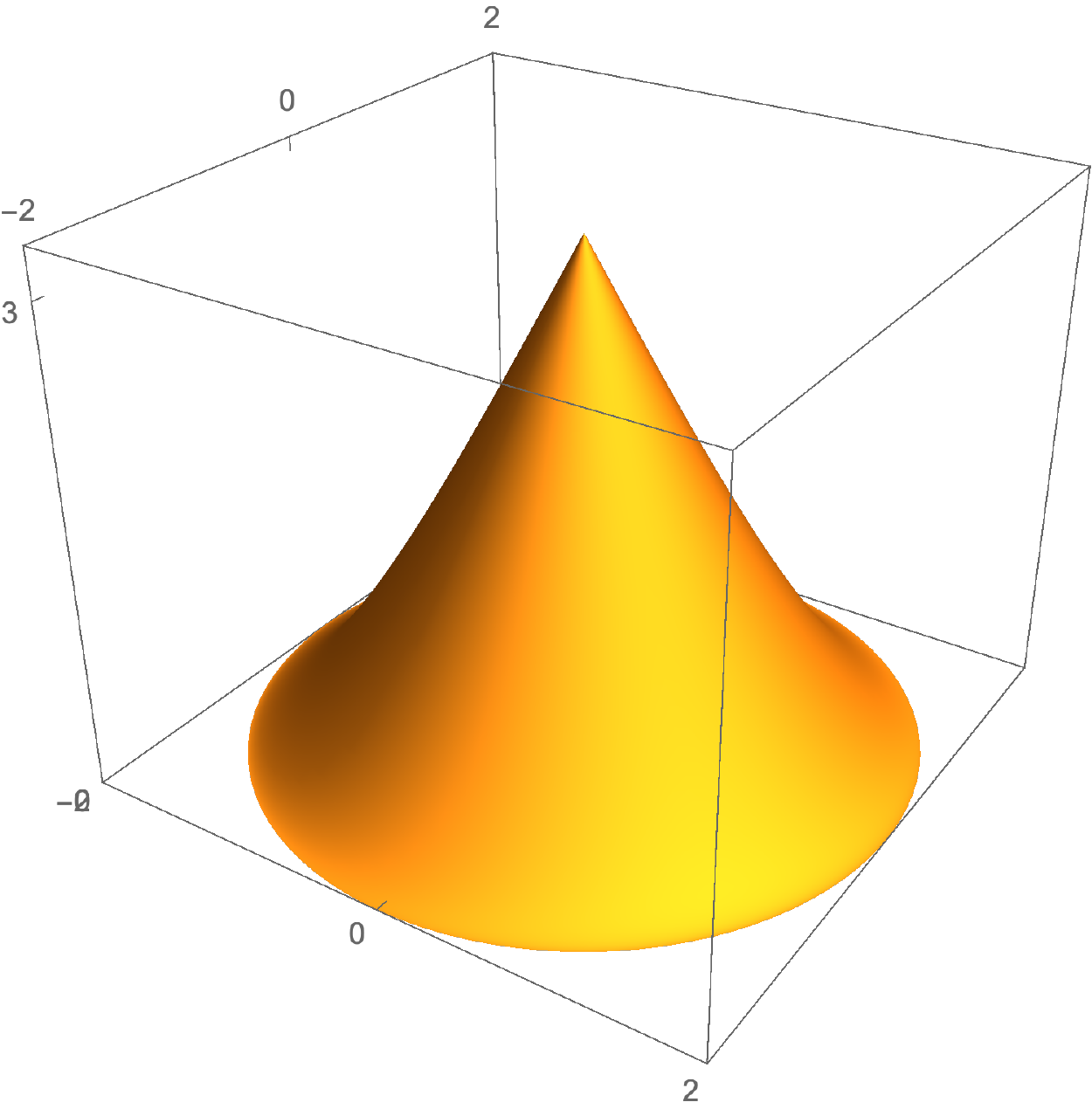}
\caption{An indefinite improper affine sphere
$f\ \left(0 < x < \pi/2,\, -\pi \leq y < \pi\right)$ over
the region enclosed by $2 \gamma$.}
\end{figure}
\begin{figure}[H]\label{fig:exam2-1}
\hfill
\hfill
\subfigure{\includegraphics[height=6cm,keepaspectratio]%
{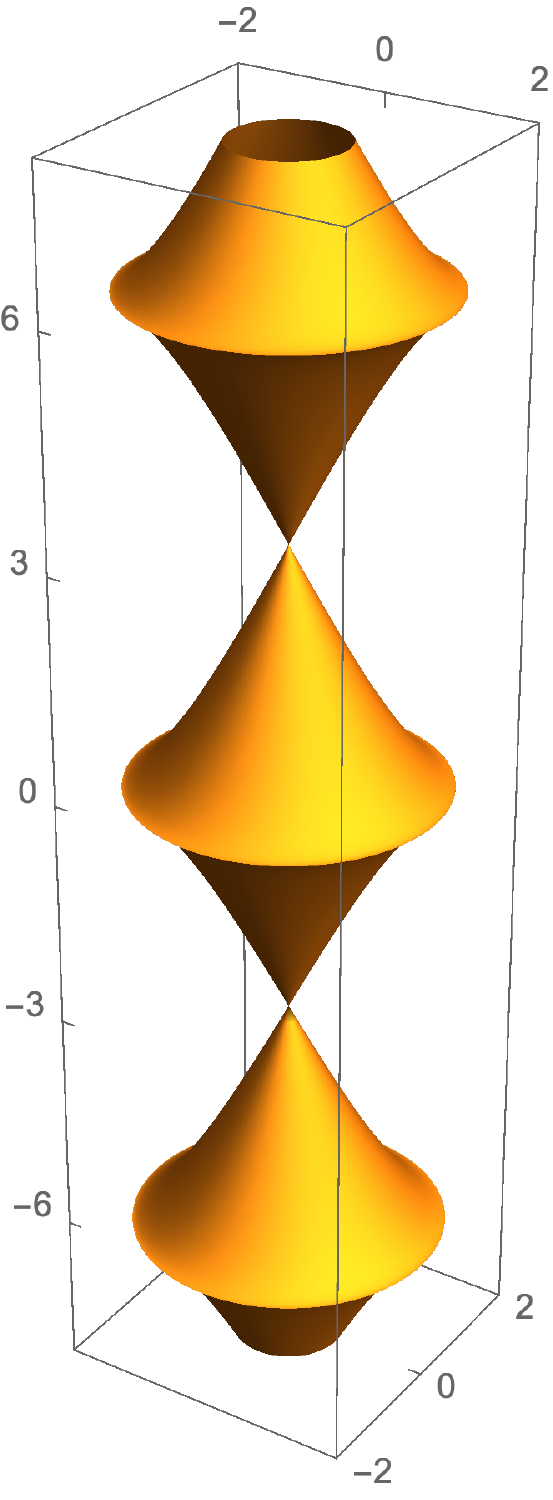}}
\hfill
\subfigure{\includegraphics[width=6cm,keepaspectratio]%
{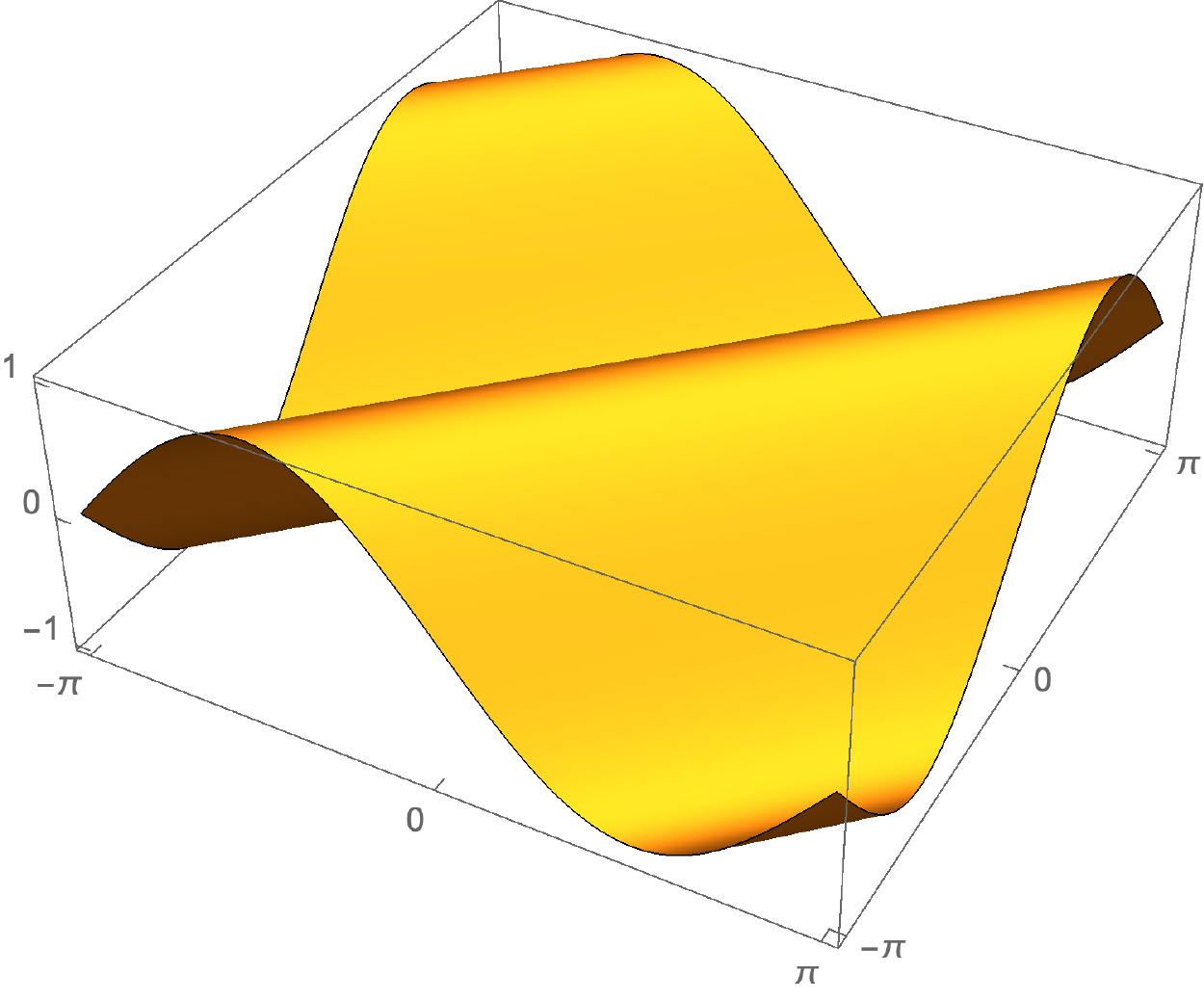}}
\hfill
\hfill
\caption{Left: an indefinite improper affine map $f$,
which is a series of surfaces in Figure \ref{fig:exam2-0},
joined along cuspidal edges and at cone points.
Right: the graph of $\omega$,
which gives the affine metric of $f$ apart from $S$.}
\end{figure}

\item
Second example is given by the square
\begin{equation*}
\gamma \left(u\right) =
\begin{bmatrix}
\left|\cos u\right| \cos u\\
\left|\sin u\right| \sin u
\end{bmatrix}.
\end{equation*}
We have that $\det \left[\gamma \left(u\right),
\gamma' \left(u\right)\right] = \left|\sin 2u\right|$
for $u \in \R \setminus \left(\pi/2\right) \Z$,
which follows from
\begin{equation*}
\gamma' \left(u\right) = 2
\begin{bmatrix}
- \left|\cos u\right| \sin u\\
\left|\sin u\right| \cos u
\end{bmatrix},\quad
\gamma'' \left(u\right) = 2 \cos 2u
\begin{bmatrix}
- \sign \left(\cos u\right)\\
\sign \left(\sin u\right)
\end{bmatrix},
\end{equation*}
where
\begin{equation*}
\sign x =
\begin{cases}
1 & \left(x>0\right)\\
0 & \left(x=0\right)\\
-1 & \left(x<0\right).
\end{cases}
\end{equation*}
It is convenient for the following discussion to interpret
$\gamma' \left(k\right) = 0$ and
$\gamma'' \left(k\right) \parallel \gamma \left(k\right)$
for all $k \in \left(\pi/2\right) \Z$.
It holds for all $u \in \R$ that
\begin{equation*}
% \int_0^u \det \left(\gamma \left(k\right),
% \gamma' \left(k\right)\right) dk
\int_0^u \left|\sin 2k\right| dk
= \left\lceil \frac{2}{\pi} u\right\rceil
- \frac{\sign \left(\sin 2u\right)}{2}
\left(\cos 2u + \sign \left(\sin 2u\right)\right),
\end{equation*}
where we denote by $\left\lceil u\right\rceil$
the ceiling of $u$,
that is,
the smallest integer greater than or equal to $u$.
Thus we have
for $u, v \in \R \setminus \left(\pi/2\right) \Z$
that
\begin{equation*}
f \left(u, v\right) =
\begin{bmatrix}
\left|\cos u\right| \cos u
+ \left|\cos v\right| \cos v\\
\left|\sin u\right| \sin u
+ \left|\sin v\right| \sin v\\
z \left(u, v\right)
\end{bmatrix},
\end{equation*}
where
\begin{equation*}
\begin{split}
z \left(u, v\right) =&
\left|\cos u \sin v\right| \cos u \sin v
- \left|\cos v \sin u\right| \cos v \sin u\\
&+ \left\lceil \left(2/\pi\right) u\right\rceil
- \left(1/2\right) \sign \left(\sin 2u\right)
\left(\cos 2u + \sign \left(\sin 2u\right)\right)\\
&- \left\lceil \left(2/\pi\right) v\right\rceil
+ \left(1/2\right) \sign \left(\sin 2v\right)
\left(\cos 2v + \sign \left(\sin 2v\right)\right).
\end{split}
\end{equation*}
Its data is $A = B = 0$ and
\begin{align*}
\omega
&=
4 \left(- \left|\cos u \sin v\right| \cos v \sin u
+ \left|\cos v \sin u\right| \cos u \sin v\right).
\end{align*}
The singular set $S$ is a checkerboard
\begin{equation*}
S = \left\{\left(u, v\right) \in \R^2 \;\left|\;
u \in \frac{\pi}{2} \Z\ \,\text{or}\ \,
v \in \frac{\pi}{2} \Z\ \,\text{or}\ \,
\left\lceil \frac{2}{\pi} u\right\rceil \equiv
\left\lceil \frac{2}{\pi} v\right\rceil\;
\left(\mathrm{mod}\ 2\right)\right.\right\}.
\end{equation*}
\begin{figure}[H]
\hfill
\hfill
\subfigure{\includegraphics[height=6cm,keepaspectratio]%
{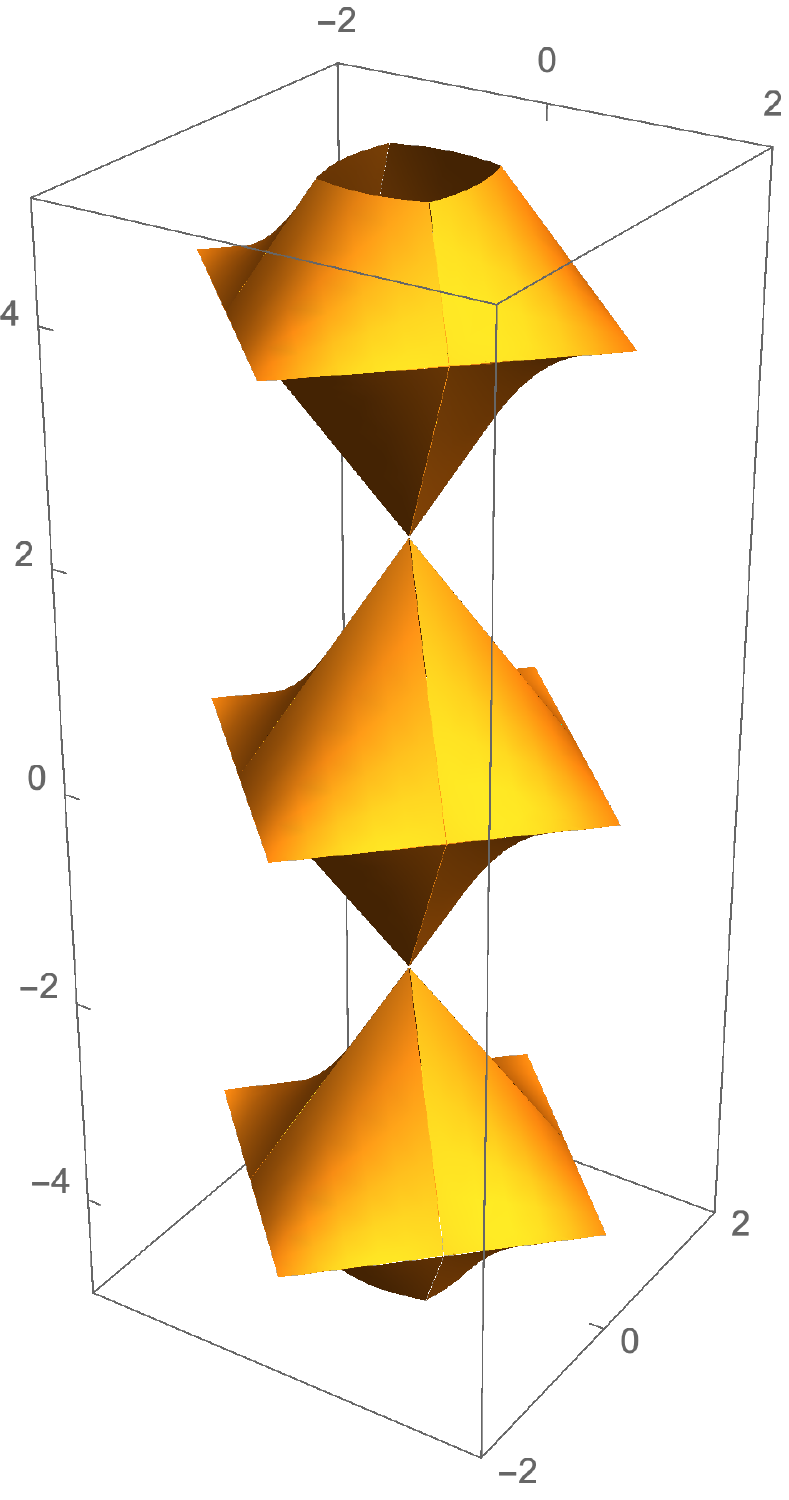}}
\hfill
\subfigure{\includegraphics[width=6cm,keepaspectratio]%
{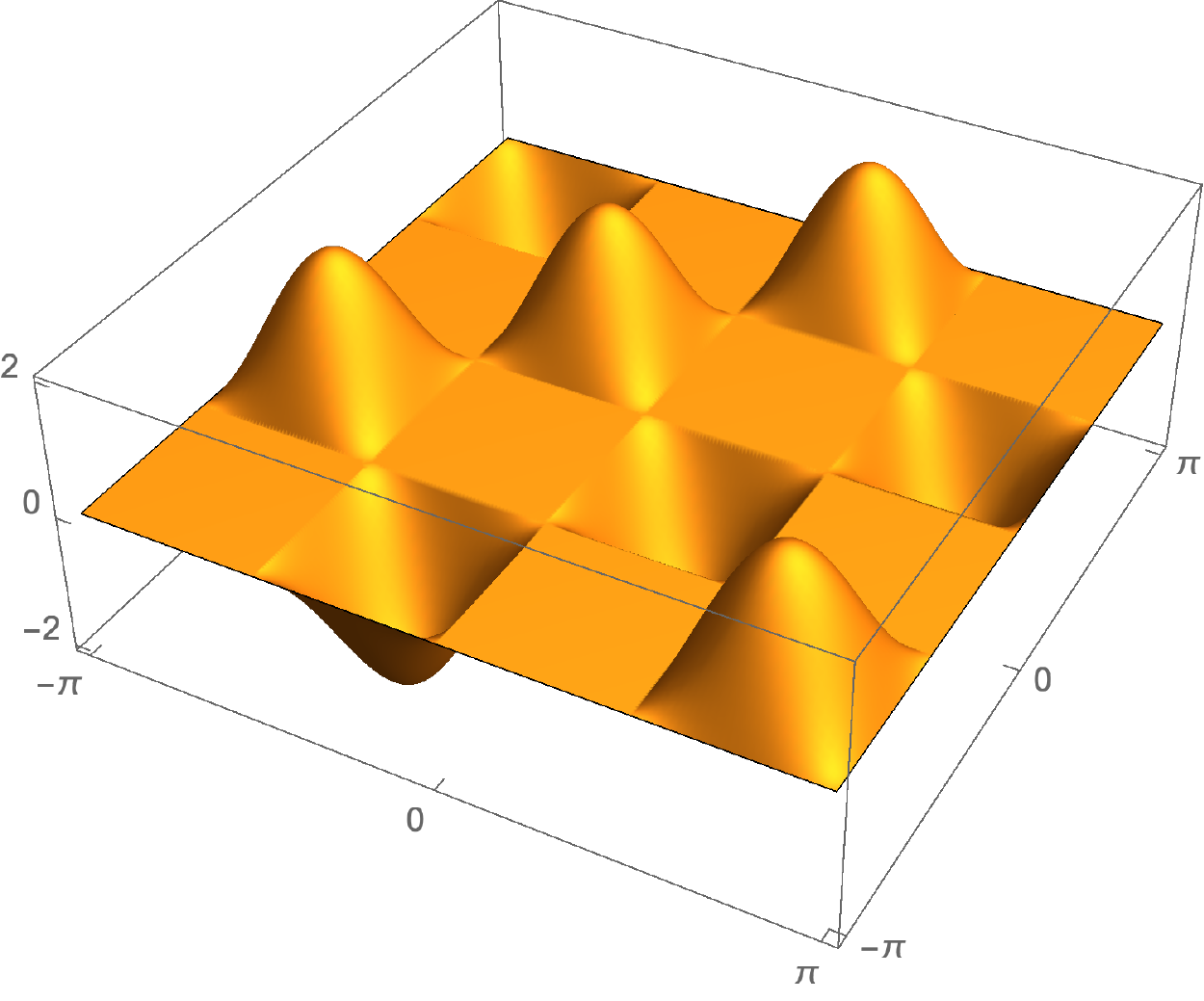}}
\hfill
\hfill
\caption{An indefinite improper affine map $f$ (left),
and its affine metric $\omega$ (right).}
\label{fig:exam2-2}
\end{figure}

\item
Last example is given by the curve
\begin{equation*}
\gamma \left(u\right) =
\cos u
\begin{bmatrix}
1/2 + \cos^2 u\\
2 \sin u
\end{bmatrix}.
\end{equation*}
We have
$\det \left[\gamma \left(u\right),
\gamma' \left(u\right)\right] = \left(3 + 2 \sin^2 u\right) \cos^3 u$
and hence
\begin{equation}\label{ex:integrate-det}
\int^u_0 \det \left[\gamma \left(k\right),
\gamma' \left(k\right)\right] dk
= \frac{5}{2} \sin u
+ \frac{5}{24} \sin 3u
- \frac{1}{40} \sin 5u.
\end{equation}
Therefore
\begin{align*}
f \left(u, v\right) &=
\begin{bmatrix}
\left(1 + \left(1/2\right) \cos 2u\right) \cos u
+ \left(1 + \left(1/2\right) \cos 2v\right) \cos v\\
\sin 2u + \sin 2v\\
z \left(u, v\right)
\end{bmatrix},
\end{align*}
where
\begin{equation*}
\begin{split}
z \left(u, v\right) =\;&
- \cos u \cos v \left(\sin u - \sin v\right)
\left(3 + 2 \sin u \sin v\right)\\
&+ \frac{5}{2} \left(\sin u - \sin v\right)
+ \frac{5}{24} \left(\sin 3u - \sin 3v\right)
- \frac{1}{40} \left(\sin 5u - \sin 5v\right).
\end{split}
\end{equation*}
Its data is
\begin{align*}
\omega &=
- \left(\sin u - \sin v\right)
\left(4 + 8 \sin u \sin v + 3 \cos 2u \cos 2v\right),\\
A &= \frac{1}{2} \left(19 - 8 \cos 2u + 3 \cos 4u\right) \cos u,\\
B &= - \frac{1}{2} \left(19 - 8 \cos 2v + 3 \cos 4v\right) \cos v.
\end{align*}
The singular set of $f$ is $S = S_1 \cup S_2$,
where
\begin{align*}
S_1 &= \left\{\left(u, v\right) \in \R^2 \;\big|\;
v \equiv u\ \left(\mathrm{mod}\ 2\pi\right),\
v \equiv - u + \pi\ \left(\mathrm{mod}\ 2\pi\right)\right\},\\
S_2 &= \left\{\left(u, v\right) \in \R^2 \;\big|\;
4 + 8 \sin u \sin v + 3 \cos 2u \cos 2v = 0\right\}.
\end{align*}
The sets $S_1$ and $S_2$ consist of lines and
circlelike curves, respectively.
The surface is compact and of genus $1$.
\end{enumerate}
\begin{figure}[H]
\hfill
\hfill
\subfigure{\includegraphics[height=6cm,keepaspectratio]%
{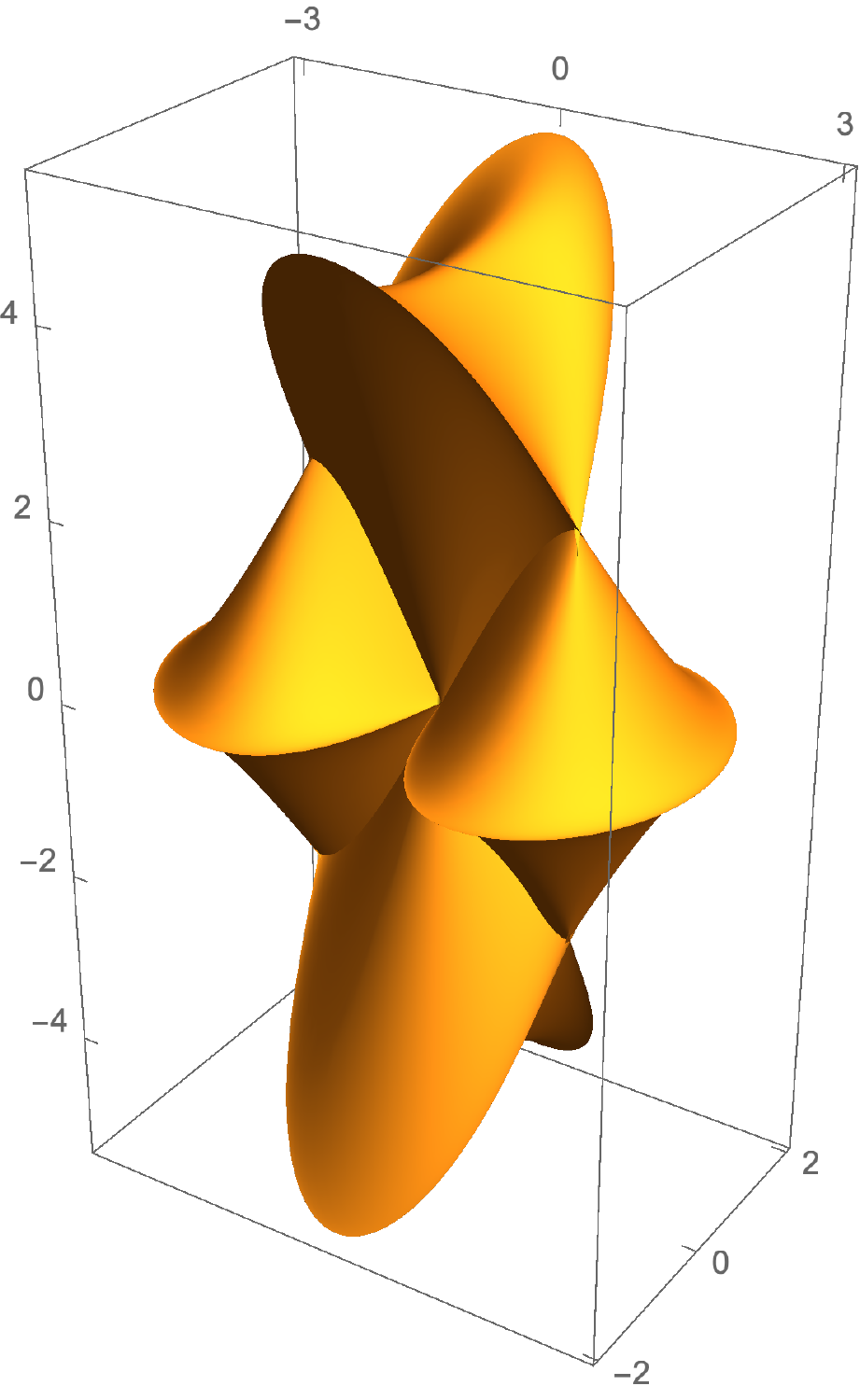}}
\hfill
\subfigure{\includegraphics[width=6cm,keepaspectratio]%
{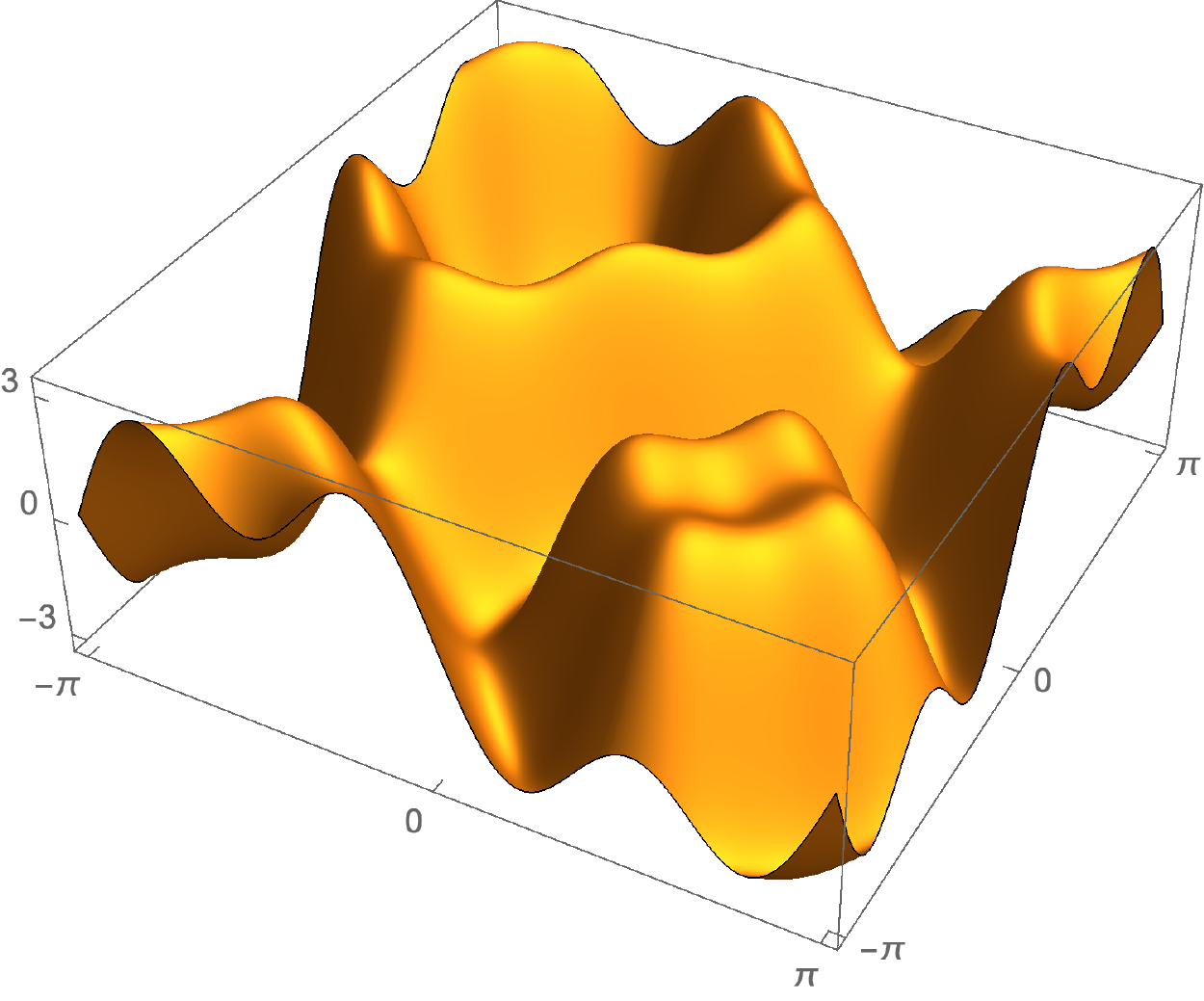}}
\hfill
\hfill
\caption{An indefinite improper affine map $f$ (left),
and its affine metric $\omega$ (right).}
\label{fig:exam2-3}
\end{figure}
\end{Example}

\section{Discrete indefinite affine spheres}\label{sc:discrete}

In the previous section
we derived Theorem \ref{thm:Rep} and Corollary \ref{cor:repformula}
which offer a Weierstrass type representation formula
of indefinite affine spheres.
Based on a technique of decompositions of the loop group,
we shall generalize this formula to discrete case,
and obtain a Weierstrass type representation formula for discrete
indefinite affine spheres.

\subsection{Definitions}

Let $f\colon \Z^2 \to \R^3,
\left(n, m\right) \mapsto f^m_n$ be a map.
We call $f$ a
\textit{discrete indefinite affine sphere} if it satisfies the following
two properties (\cite{MR1676596}, \cite{MR1724159}, \cite{MR1949349}):
\begin{enumerate}
\item\label{defn:d-asymptotic-coord}
Every five points $f^m_n$,
$f^m_{n \pm 1}$,
$f^{m \pm 1}_n$ lie on a plane.

\item
The line $\ell^m_n$ connecting two points
$f^{m+1}_{n+1} + f^m_n$ and $f^m_{n+1} + f^{m+1}_n$
satisfies either of the following two conditions:
\begin{enumerate}
\item\label{defn:d-proper-AF}
All the lines $\ell^m_n$ meet at one point in $\R^3$.

\item\label{defn:d-improper-AF}
All the lines $\ell^m_n$ are parallel to each other.
\end{enumerate}
\end{enumerate}
A discrete indefinite affine sphere $f$
is said to be \textit{proper} if it
satisfies the condition \eqref{defn:d-proper-AF},
or \textit{improper} if \eqref{defn:d-improper-AF}.

If $f$ is a discrete indefinite affine sphere,
the vector $f^{m+1}_{n+1} - f^m_{n+1} - f^{m+1}_n + f^m_n$ is
parallel to the \textit{discrete affine normal}
\begin{equation}\label{defn:d-affinenormal}
\xi^m_n = - H\,
\frac{f^m_{n+1} + f^{m+1}_n}{2}
+ \left(1 + H\right) \xi_0,
\end{equation}
where $\xi_0$ is a constant vector.
Here we set $H = - 1$ if $f$ is proper,
and $H = 0$ if $f$ is improper.
Without loss of generality
we can fix $\xi_0$ to be ${}^\mathrm{t} \left[0, 0, 1\right]$.
Taking into account a continuum limit,
we introduce positive numbers $\epsilon$ and $\delta$,
which play a role of lattice intervals.
In view of this it may be better to regard $f$ as a map
$f\colon \epsilon \Z \times \delta \Z \to \R^3$,
and hence entries of $f$ depend on $\epsilon$ and $\delta$.
We define
\begin{equation*}
\tilde{F}^m_n = \left[\frac{f^m_{n+1} - f^m_n}{\epsilon},\,
\frac{f^{m+1}_n - f^m_n}{\delta},\,
\xi^m_n\right].
\end{equation*}
We suppose that $\det \tilde{F}^m_n \neq 0$,
then there exist functions $\omega$, $A$, $B$ such that
\begin{align}
& \frac{f^m_{n+1} - 2 f^m_n + f^m_{n-1}}{\epsilon^2}
=\label{dGauss-11}
\left(\frac{\omega^m_n - \omega^m_{n-1}}{\epsilon \omega^m_n}
+ \frac{\delta H}{2} \omega^m_{n-1}\right)
\frac{f^m_{n+1} - f^m_n}{\epsilon}
+ \frac{A^m_n}{\omega^m_n}
\frac{f^{m+1}_n - f^m_n}{\delta},\\
& \frac{f^{m+1}_{n+1} - f^m_{n+1} - f^{m+1}_n + f^m_n}{\epsilon \delta}
=\label{dGauss-12}
\omega^m_n \xi^m_n,\\
& \frac{f^{m+1}_n - 2 f^m_n + f^{m-1}_n}{\delta^2}
=\label{dGauss-22}
\frac{B^m_n}{\omega^m_n}
\frac{f^m_{n+1} - f^m_n}{\epsilon}
+ \left(\frac{\omega^m_n - \omega^{m-1}_n}{\delta \omega^m_n}
+ \frac{\epsilon H}{2} \omega^{m-1}_n\right)
\frac{f^{m+1}_n - f^m_n}{\delta}.
\end{align}
Equations \eqref{dGauss-11} and \eqref{dGauss-22} are consequences
of the property \eqref{defn:d-asymptotic-coord}.
See \cite[p.~118]{MR1676596} or \cite[Proposition 3.4]{MR1949349}
for a proof.
Throughout the paper we further impose on $f$
the volume condition
\begin{equation*}
\det \tilde{F}^m_n
=
\frac{2 \omega^m_n}{2 - \epsilon \delta H \omega^m_n},
\end{equation*}
which can be regarded as a discrete analogue of
\eqref{volumecondition}.
We write
\begin{equation}\label{defn:g}
g^m_n = \frac{2}{2 -\epsilon \delta H \omega^m_n}
\end{equation}
to have expressions $\omega^m_n g^m_n = \det \tilde{F}^m_n$ and
\begin{align*}
A^m_n g^m_n &= \det \left[\frac{f^m_{n+1} - f^m_n}{\epsilon},\,
\frac{f^m_{n+1} - 2 f^m_n + f^m_{n - 1}}{\epsilon^2},\,
\xi^m_n\right],\\
B^m_n g^m_n &= \det \left[\frac{f^{m+1}_n - 2 f^m_n + f^{m-1}_n}{\delta^2},\,
\frac{f^{m+1}_n - f^m_n}{\delta},\,
\xi^m_n\right].
\end{align*}
From the compatibility condition among
\eqref{dGauss-11}--\eqref{dGauss-22},
it follows that $\omega$, $A$, $B$ satisfy the system
\begin{gather}
\omega^{m+1}_{n+1} \omega^m_n
- \frac{\omega^m_{n + 1} \omega^{m + 1}_n}{g^{m + 1}_{n + 1} g^m_n}
+ \epsilon \delta A^m_{n+1} B^{m+1}_{n+1} = 0,\label{eq:DGC-omega}\\
g^{m + 1}_{n + 1} A^{m+1}_{n+1} =
g^m_n A^m_{n+1}, \quad
g^{m + 1}_{n + 1} B^{m+1}_{n+1} =
g^m_n B^{m+1}_n.\label{eq:DGC-AB}
\end{gather}
Indeed,
the matrix $\tilde{F}^m_n$ varies according to the system
\begin{equation}\label{dGW:Ftilde}
\tilde{F}^m_{n+1} = \tilde{F}^m_n \tilde{U}^m_n,\quad
\tilde{F}^{m+1}_n = \tilde{F}^m_n \tilde{V}^m_n,
\end{equation}
where coefficient matrices are computed as
\begin{align}
\tilde{U}^m_n
&=\label{dlax:tildeU}
\begin{bmatrix}
\frac{\omega^m_{n+1}}{\omega^m_n} g^m_{n+1} & 0 &
- \epsilon \frac{H}{2}
\left(1 + \frac{\omega^m_{n+1}}{\omega^m_n} g^m_{n+1}\right)\\
\epsilon \frac{A^m_{n+1}}{\omega^m_n} g^m_{n+1} & 1 &
- \epsilon^2 \frac{H}{2} \frac{A^m_{n+1}}{\omega^m_n} g^m_{n+1}\\
\epsilon^2 A^m_{n+1} g^m_{n+1} & \epsilon \omega^m_n &
\frac{1}{g^m_n} - \epsilon^3 \frac{H}{2} A^m_{n+1} g^m_{n+1}
\end{bmatrix},\\
\tilde{V}^m_n
&=\label{dlax:tildeV}
\begin{bmatrix}
1 & \delta \frac{B^{m+1}_n}{\omega^m_n} g^{m+1}_n &
- \delta^2 \frac{H}{2} \frac{B^{m+1}_n}{\omega^m_n} g^{m+1}_n\\
0 & \frac{\omega^{m+1}_n}{\omega^m_n} g^{m+1}_n &
- \delta \frac{H}{2} \left(1 +
\frac{\omega^{m+1}_n}{\omega^m_n} g^{m+1}_n\right)\\
\delta \omega^m_n &
\delta^2 B^{m+1}_n g^{m+1}_n &
\frac{1}{g^m_n} - \delta^3 \frac{H}{2} B^{m+1}_n g^{m+1}_n
\end{bmatrix}.
\end{align}
The compatibility condition
$\tilde{U}^m_n \tilde{V}^m_{n+1}
= \tilde{V}^m_n \tilde{U}^{m+1}_n$ is
\eqref{eq:DGC-omega}--\eqref{eq:DGC-AB}.
Therefore the system \eqref{dGW:Ftilde}--\eqref{dlax:tildeV},
or equivalently \eqref{dGauss-11}--\eqref{dGauss-22},
has a solution
if and only if \eqref{eq:DGC-omega}--\eqref{eq:DGC-AB} hold.
The system \eqref{eq:DGC-omega}--\eqref{eq:DGC-AB} is
a discrete analogue of the system
\eqref{tzitzeica-liouville}--\eqref{tzitzeica-liouville:aux},
and hence called the \textit{discrete Tzitzeica equation} if $H = -1$,
or the \textit{discrete Liouville equation} if $H = 0$.
Since this system \eqref{eq:DGC-omega}--\eqref{eq:DGC-AB} is
invariant under the transformation
\begin{equation*}
A^m_n \mapsto \lambda^3 A^m_n,\quad
B^m_n \mapsto \lambda^{- 3} B^m_n,\quad
\lambda \in \R^\times,
\end{equation*}
the discrete affine sphere $f$ has $1$-parameter family,
which we call the \textit{associated family} of $f$.
The associated family preserves $\omega$.

\subsection{Loop group description}

In order to derive a representation formula for
discrete indefinite affine spheres,
we use decomposition techniques of loop groups.
To begin with,
following \cite{MR1676596},
we describe the discrete indefinite affine spheres
in terms of the loop groups.
We set
\begin{equation}\label{def:d-extended}
F^m_n = \tilde{F}^m_n
\begin{bmatrix}
\lambda^{-1} & 0 & \epsilon H/2\\
0 & \lambda \left(\omega^m_n g^m_n\right)^{-1} & \delta H/2\\
0 & 0 & 1
\end{bmatrix}.
\end{equation}
Then the map $F$ is $\SL$-valued,
and satisfies the system
\begin{equation}\label{eq:discreteLax}
F^m_{n+1} = F^m_n U^m_n, \quad
F^{m+1}_n = F^m_n V^m_n,
\end{equation}
where $U$ and $V$ are computed as
\begin{align}
U^m_n &=\label{eq:DU}
\begin{bmatrix}
\frac{\omega^m_{n+1} g^m_{n+1}}{\omega^m_n}
- \frac{H}{2} A^m_{n+1} g^m_{n+1} \left(\epsilon\lambda\right)^3 &
- \frac{H}{2} \frac{\omega^m_n}{\omega^m_{n+1} g^m_{n+1}}
\left(\epsilon\lambda\right)^2 &
- H \epsilon \lambda\\
A^m_{n+1} g^m_{n+1} \epsilon \lambda &
\frac{\omega^m_n}{\omega^m_{n+1} g^m_{n+1}} &
0\\
A^m_{n+1} g^m_{n+1} \left(\epsilon \lambda\right)^2 &
\frac{\omega^m_n}{\omega^m_{n+1} g^m_{n+1}} \epsilon \lambda &
1
\end{bmatrix},\\
V^m_n &=\label{eq:DV}
\begin{bmatrix}
\frac{1}{g^m_n} &
\frac{B^{m+1}_n}{\omega^{m+1}_n \omega^m_n g^m_n} \delta \lambda^{-1}
& 0 \\
- \frac{H}{2} \left(\omega^m_n\right)^2 g^m_n
\left(\delta \lambda^{-1}\right)^2 &
g^m_n - \frac{H}{2} \frac{B^{m+1}_n \omega^m_n g^m_n}{\omega^{m+1}_n}
\left(\delta\lambda^{-1}\right)^3 &
-  H \omega^m_n g^m_n \delta\lambda^{-1}\\
\omega^m_n \delta \lambda^{-1} &
\frac{B^{m+1}_n}{\omega^{m+1}_n} \left(\delta \lambda^{-1}\right)^2 &
1
\end{bmatrix}.
\end{align}
The consistency of \eqref{eq:discreteLax},
that is, $U^m_n V^m_{n+1} = V^m_n U^{m+1}_n$,
is of course given by \eqref{eq:DGC-omega}--\eqref{eq:DGC-AB}.
By multiplying $F$ by some constant matrix from the left if necessary,
without loss of generality we can assume that
\begin{equation}\label{eq:DFid}
F^0_0 = \id
\end{equation}
at the base point $(n, m)=(0, 0)$.
The family of gauged frames $F$ defined by \eqref{def:d-extended} with
the initial condition \eqref{eq:DFid} will be called
the \textit{extended frame} of discrete affine sphere $f$.
The extended frame $F$ is obviously a $\LSL$-valued map.
Conversely,
if the matrices $\overline{U}^m_n$ and $\overline{V}^m_n$ have
similar entries as \eqref{eq:DU} and \eqref{eq:DV} respectively,
then they give the extended family of discrete indefinite affine spheres.
In fact we have the following proposition,
which has been shown for the discrete indefinite
proper affine spheres ($H=-1$) in
\cite[Theorem in Section $6$]{MR1676596}.
\begin{Proposition}\label{prop:extendtosurf}
Let $\overline{U}{}^m_n$ and $\overline{V}{}^m_n$ be matrices
which depend on a parameter $\lambda \in \R^\times$ as
\begin{equation}\label{eq:DUVarb}
\begin{split}
\overline{U}{}^m_n &=
\overline{U}{}^0_{n,m} + \lambda \overline{U}{}^1_{n,m}
+ \lambda^2 \overline{U}{}^2_{n,m}+ \lambda^3 \overline{U}{}^3_{n,m},\\
\overline{V}{}^m_n &=
\overline{V}{}^0_{n,m} + \lambda^{-1} \overline{V}{}^1_{n,m}
+ \lambda^{-2} \overline{V}{}^2_{n,m} + \lambda^{-3} \overline{V}{}^3_{n,m}.
\end{split}
\end{equation}
Here coefficient matrices $\overline{U}{}^i$ and
$\overline{V}{}^i$,
which are labeled by the index $i\, \left(0 \leq i \leq 3\right)$,
have the entries
\begin{gather*}
\overline{U}{}^0_{n,m}
= \diag \left(1/u^{22}_{n,m},\, u^{22}_{n,m},\, 1\right),\quad
\overline{U}{}^3_{n,m}
= \diag \big(- \left(H/2\right)
\left(u^{13}_{n,m}\right)^2 u^{21}_{n,m},\, 0,\, 0\big),\\
\overline{U}{}^1_{n,m} =
\begin{bmatrix}
0 & 0 & - H u^{13}_{n,m}\\
u^{21}_{n,m} & 0 & 0\\
0 & u^{22}_{n,m} u^{13}_{n,m} & 0
\end{bmatrix},\quad
\overline{U}{}^2_{n,m} =
\begin{bmatrix}
0 & -\frac{H}{2} u^{22}_{n,m} \left(u^{13}_{n,m}\right)^2 & 0\\
0 & 0 & 0\\
u^{13}_{n,m} u^{21}_{n,m} & 0 & 0
\end{bmatrix},
\end{gather*}
and
\begin{gather*}
\overline{V}{}^0_{n,m}
= \diag \left(v^{11}_{n,m},\, 1/v^{11}_{n,m},\, 1\right),\quad
\overline{V}{}^3_{n,m}
= \diag \big(0,\,
- \left(H/2\right) \left(v^{23}_{n,m}\right)^2 v^{12}_{n,m},\, 0\big),\\
\overline{V}{}^1_{n,m} =
\begin{bmatrix}
0 & v^{12}_{n,m} & 0\\
0 & 0 & - H v^{23}_{n,m}\\
v^{11}_{n,m} v^{23}_{n,m} & 0 & 0
\end{bmatrix},\quad
\overline{V}{}^2_{n,m} =
\begin{bmatrix}
0 & 0 & 0\\
-\frac{H}{2} v^{11}_{n,m} \left(v^{23}_{n,m}\right)^2 & 0 & 0\\
0 & v^{23}_{n,m} v^{12}_{n,m} & 0
\end{bmatrix},
\end{gather*}
with some functions $u^{21}$, $v^{12}$,
and nowhere vanishing functions $u^{13}$, $u^{22}$, $v^{11}$, $v^{23}$,
and a constant $H \in \left\{- 1, 0\right\}$.
If $\overline{U}$ and $\overline{V}$ satisfy
the relation $\overline{U}{}^m_n \overline{V}{}^m_{n+1}
= \overline{V}{}^m_n \overline{U}{}^{m+1}_n$ for all $\lambda$,
then there exist a map $\overline{F}$ and a gauge $D$ such that
\begin{enumerate}
\item
$\overline{F}$ satisfies the system
$\overline{F}{}^m_{n+1}
= \overline{F}{}^m_n \overline{U}{}^m_n$,
$\overline{F}{}^{m+1}_n
= \overline{F}{}^m_n \overline{V}{}^m_n$ and

\item
$\left(D^0_0\right)^{-1} \overline{F}{}^m_n D^m_n$ is the extended frame of
some discrete indefinite affine sphere.
\end{enumerate}
\end{Proposition}
\begin{proof}
Because of the relation
$\overline{U}{}^m_n \overline{V}{}^m_{n+1}
= \overline{V}{}^m_n \overline{U}{}^{m+1}_n$,
the existence of $\overline{F}$ is clear.
We fix a positive number $\epsilon$, and set
\begin{equation*}
F^m_n = \overline{F}{}^m_n D^m_n,\quad
D^m_n = \diag \left(u^{13}_{n,m}/\epsilon,\,
\epsilon/u^{13}_{n,m},\, 1\right).
\end{equation*}
Then $F$ satisfies
$F^m_{n+1} = F^m_n U^m_n$ and
$F^{m+1}_n = F^m_n V^m_n$,
where
\begin{align*}
U^m_n
&= \left(D^m_n\right)^{-1} \overline{U}{}^m_n D^m_{n+1}\\
&=
\begin{bmatrix}
\frac{u^{13}_{n+1,m}}{u^{13}_{n,m} u^{22}_{n,m}}
- \lambda^3 \frac{H}{2} u^{13}_{n,m} u^{21}_{n,m} u^{13}_{n+1,m} &
- \frac{H}{2} \lambda^2
\frac{\epsilon^2 u^{13}_{n,m} u^{22}_{n,m}}{u^{13}_{n+1,m}} &
- H \epsilon \lambda\\
\lambda \frac{u^{13}_{n,m} u^{21}_{n,m} u^{13}_{n+1,m}}{\epsilon^2} &
\frac{u^{13}_{n,m} u^{22}_{n,m}}{u^{13}_{n+1,m}} & 0\\
\lambda^2 \frac{u^{13}_{n,m} u^{21}_{n,m} u^{13}_{n+1,m}}{\epsilon} &
\lambda \frac{\epsilon u^{13}_{n,m} u^{22}_{n,m}}{u^{13}_{n+1,m}} & 1
\end{bmatrix},\\
V^m_n
&= \left(D^m_n\right)^{-1} \overline{V}{}^m_n D^{m+1}_n\\
&=
\begin{bmatrix}
\frac{v^{11}_{n,m} u^{13}_{n,m+1}}{u^{13}_{n,m}} &
\frac{\epsilon^2 v^{12}_{n,m}}{\lambda u^{13}_{n,m} u^{13}_{n,m+1}} & 0\\
- \frac{H}{2}
\frac{u^{13}_{n,m} u^{13}_{n,m+1} v^{11}_{n,m} \left(v^{23}_{n,m}\right)^2}%
{\lambda^2 \epsilon^2} &
\frac{u^{13}_{n,m}}{u^{13}_{n,m+1} v^{11}_{n,m}}
- \frac{H}{2} \frac{u^{13}_{n,m} v^{12}_{n,m}
\left(v^{23}_{n,m}\right)^2}{\lambda^3 u^{13}_{n,m+1}} &
- H \frac{u^{13}_{n,m} v^{23}_{n,m}}{\lambda \epsilon}\\
\frac{v^{11}_{n,m} v^{23}_{n,m} u^{13}_{n,m+1}}{\lambda \epsilon} &
\frac{\epsilon v^{12}_{n,m} v^{23}_{n,m}}{\lambda^2 u^{13}_{n,m+1}} & 1
\end{bmatrix}.
\end{align*}
Next we fix a positive number $\delta$ and introduce sequences $\omega$,
$A$, $B$ by
\begin{equation*}
\omega^m_n =
\frac{v^{11}_{n,m} v^{23}_{n,m} u^{13}_{n,m+1}}{\epsilon \delta},
\end{equation*}
and
\begin{equation*}
A^m_n =
\frac{u^{13}_{n-1,m} u^{21}_{n-1,m} u^{13}_{n,m}}{\epsilon^3 g^m_n},\quad
B^m_n =
\frac{\epsilon\, v^{12}_{n,m-1} v^{23}_{n,m-1}}{\delta^2 u^{13}_{n,m}}
\omega^m_n,
\end{equation*}
where $g$ is defined by \eqref{defn:g}.
Therefore the matrices $U^m_n$ and $V^m_n$ are written as
\begin{align*}
U^m_n
&=
\begin{bmatrix}
\frac{1}{r^m_n} -
\lambda^3 \epsilon^3
\frac{H}{2} A^m_{n+1} g^m_{n+1} &
- \frac{H}{2} \lambda^2 \epsilon^2 r^m_n &
- H \epsilon \lambda\\
\lambda \epsilon A^m_{n+1} g^m_{n+1} &
r^m_n & 0\\
\lambda^2 \epsilon^2 A^m_{n+1} g^m_{n+1} &
\lambda \epsilon r^m_n & 1
\end{bmatrix},\\
V^m_n
&=
\begin{bmatrix}
\frac{1}{h^m_n} &
\frac{\delta B^{m+1}_n}{\lambda \omega^m_n \omega^{m+1}_n h^m_n} & 0\\
- \frac{H}{2}
\frac{\delta^2 \left(\omega^m_n\right)^2 h^m_n}{\lambda^2} &
h^m_n
- \frac{H}{2} \frac{\delta^3 B^{m+1}_n \omega^m_n h^m_n}%
{\lambda^3 \omega^{m+1}_n} &
- H \frac{\delta h^m_n \omega^m_n}{\lambda}\\
\frac{\delta \omega^m_n}{\lambda} &
\frac{\delta^2 B^{m+1}_n}{\lambda^2 \omega^{m+1}_n} & 1
\end{bmatrix},
\end{align*}
where
\begin{equation*}
h^m_n = \frac{u^{13}_{n,m}}{u^{13}_{n,m+1} v^{11}_{n,m}},\quad
r^m_n = \frac{u^{13}_{n,m} u^{22}_{n,m}}{u^{13}_{n+1,m}}.
\end{equation*}
We note that the condition $\overline{U}{}^m_n \overline{V}{}^m_{n+1}
= \overline{V}{}^m_n \overline{U}{}^{m+1}_n$ is equivalent to
\begin{equation}\label{eq:d-compatibility-2}
U^m_n V^m_{n+1} = V^m_n U^{m+1}_n,
\end{equation}
which implies that
\begin{align*}
h^m_n = g^m_n,\quad
r^m_n = \frac{\omega^m_n}{\omega^m_{n+1} g^m_{n+1}}.
\end{align*}
In fact,
comparing the $(2,2)$- and $(3,2)$-entries of
the both sides of \eqref{eq:d-compatibility-2},
we have
\begin{align*}
h^m_{n+1} r^m_n + \frac{\epsilon \delta g^m_{n+1} A^m_{n+1} B^{m+1}_{n+1}}%
{h^m_{n+1} \omega^m_{n+1} \omega^{m+1}_{n+1}}
&= \frac{h^m_n r^{m+1}_n}{\left(g^m_n\right)^2},\\
h^m_{n+1} r^m_n + \frac{\epsilon \delta g^m_{n+1} A^m_{n+1} B^{m+1}_{n+1}}%
{h^m_{n+1} \omega^m_{n+1} \omega^{m+1}_{n+1}}
&= \frac{r^{m+1}_n}{g^m_n}.
\end{align*}
Then it immediately follows that $h$ should be $g$.
Further,
from $(1,2)$- and $(3,2)$-entries of \eqref{eq:d-compatibility-2},
we have
\begin{align*}
\frac{g^m_{n+1} B^{m+1}_{n+1}
\left(2 - \epsilon \delta H \left(1 + r^m_n\right) \omega^m_{n+1}\right)
\omega^{m+1}_n}%
{2 \omega^{m+1}_{n+1} B^{m+1}_n r^{m+1}_n}
&= \frac{2 r^m_n \omega^m_{n+1}}{g^m_n \omega^m_n
\left(2 - \epsilon \delta H \left(1 + r^m_n\right) \omega^m_{n+1}\right)},\\
\frac{g^m_{n+1} B^{m+1}_{n+1}
\left(2 - \epsilon \delta H \left(1 + r^m_n\right) \omega^m_{n+1}\right)
\omega^{m+1}_n}%
{2 \omega^{m+1}_{n+1} B^{m+1}_n r^{m+1}_n}
&= 1,
\end{align*}
which implies that $r^m_n \omega^m_{n+1} g^m_{n+1} = \omega^m_n$.
Thus $U^m_n$ and $V^m_n$ become exactly the same as
\eqref{eq:DU} and \eqref{eq:DV}.
We conclude that
$\left(D^0_0\right)^{-1} F^m_n$ is the extended frame of
some discrete indefinite affine sphere.
\end{proof}
We are now in position to state one of the main theorems of this paper,
which is a discrete analogue of Theorem \ref{thm:Rep}.
\begin{Theorem}\label{maintheorem}
Let $\epsilon$ and $\delta$ be positive numbers,
and $F\colon \Z^2 \to \LSL$ be the extended frame
of a discrete indefinite affine sphere
with the discrete affine normal \eqref{defn:d-affinenormal}.
By the Birkhoff decomposition,
we decompose $F^m_n$ near $\left(n, m\right) = \left(0, 0\right)$ as
\begin{equation}\label{decomposeFbyF+F-G-G+}
F^m_n = F^+_{n,m} F^-_{n,m} = G^-_{n,m} G^+_{n,m},
\end{equation}
where $F^+_{n,m} \in \LSLPN$,
$F^-_{n,m} \in \LSLN$,
$G^-_{n,m} \in \LSLNN$ and
$G^+_{n,m} \in \LSLP$.
Then $F^+$ and $G^-$ do not depend on $m$ and $n$, respectively,
that is, they satisfy that
\begin{equation*}
F^+_{n,m+1} = F^+_{n,m},\quad
G^-_{n+1,m} = G^-_{n,m}.
\end{equation*}
We write $F^+_n$ for $F^+_{n,m}$ and $G^-_m$ for $G^-_{n,m}$,
so that we have the ordinary difference equations
\begin{equation}\label{eq:DMC0}
F^+_{n+1} = F^+_n \xi^+_n,\quad
G^-_{m+1} = G^-_m \xi^-_m
\end{equation}
with
\begin{align}
\xi^+_n &=\label{eq:DMC}
\begin{bmatrix}
1 - \frac{H}{2}  (\alpha_{n+1})^2 \beta_{n+1} (\epsilon \lambda)^{3} &
- \frac{H}{2} (\alpha_{n+1})^2 (\epsilon \lambda)^2 &
- H \alpha_{n+1} \epsilon \lambda \\
\beta_{n+1} \epsilon \lambda  & 1 & 0  \\
\alpha_{n+1} \beta_{n+1} (\epsilon \lambda)^2 &
\alpha_{n+1}  \epsilon \lambda &  1
\end{bmatrix},\\
\xi^-_m &=\label{eq:DMC2}
\begin{bmatrix}
1  & \sigma_{m+1} \delta \lambda^{-1}& 0 \\
- \frac{H}{2} (\rho_{m+1})^2 (\delta \lambda^{-1})^2 &
1 - \frac{H}{2} \sigma_{m+1} (\rho_{m+1})^2 (\delta \lambda^{-1})^3 &
-H \rho_{m+1} \delta \lambda^{-1} \\
\rho_{m+1} \delta \lambda^{-1} &
\sigma_{m+1} \rho_{m+1} (\delta \lambda^{-1})^2 & 1
\end{bmatrix},
\end{align}
where functions $\alpha, \beta$ depend only on $n$,
and $\sigma, \rho$ only on $m$.
Moreover $\alpha_n \neq 0$ and $\rho_m \neq 0$ for all $n$ and $m$.

Conversely, let $\alpha_n$, $\beta_n$ be
functions depending only on $n$,
and $\sigma_m$, $\rho_m$ functions depending only on $m$.
Assume that  $\alpha_n$ and $\rho_m$ have no zeros.
Let $F^+_n$ and $G^-_m$ be solutions to the system
\eqref{eq:DMC0}--\eqref{eq:DMC2} with
the initial condition $F^+_0 = G^-_0 = \id$.
Define $V^+_{n,m} \in \LSLPN$ and $V^-_{n,m} \in \LSLN$
by the Birkhoff decomposition
for $\left(G^-_m\right)^{-1} F^+_n$ near $(n, m)=(0, 0)$ as
\begin{equation}\label{eq:DFG}
\left(G^-_m\right)^{-1} F^+_n =
V^+_{n,m} \left(V^-_{n,m}\right)^{-1},
\end{equation}
and write $\hat{F}^m_n =  F^+_n V^-_{n,m} = G^-_m V^+_{n,m}$.
Then there exists a diagonal matrix $D^m_n$
% $D^m_n =\diag \left(d^m_n, 1/d^m_n, 1\right)$
% with some non-vanishing function $d$
such that $\left(D^0_0\right)^{-1} \hat{F}^m_n D^m_n$ is
the extended frame of a discrete indefinite affine sphere $f^m_n$.
In particular,
in case of discrete indefinite proper affine spheres $\left(H = -1\right)$,
the third column of the extended frame
$\left(D^0_0\right)^{-1} \hat{F}^m_n D^m_n$ directly gives
the position vector of $f^m_n$.
\end{Theorem}
\begin{proof}
Let $F$ be an extended frame,
and define $F^+$ and $F^-$ by \eqref{decomposeFbyF+F-G-G+}.
Therefore we have $F^+_{n,m} = F_{n,m} \left(F^-_{n,m}\right)^{- 1}$
and so that
\begin{equation*}
\left(F^+_{n,m}\right)^{-1} F^+_{n,m+1}
= F^-_{n,m} \left(F^m_n\right)^{-1}
F^{m+1}_n \big(F^-_{n,m+1}\big)^{-1}
= F^-_{n,m} V^m_n \big(F^-_{n,m+1}\big)^{-1},
\end{equation*}
where $V^m_n$ is given by \eqref{eq:DV}.
The left-hand side takes values in $\LSLPN$
and the right-hand side takes values in $\LSLN$.
Thus $\left(F^+_{n,m}\right)^{-1} F^+_{n,m+1} = \id$.
Similarly $\left(G^-_{n,m}\right)^{-1} G^-_{n+1,m}$ is identity matrix.
Therefore $F^+$ and $G^-$ do not depend on $m$ and $n$, respectively.

Next, let us compute $\left(F^+_n\right)^{-1} F^+_{n+1}$ and
$\left(G^-_m\right)^{-1} G^-_{m+1}$. It is straightforward to see that
\begin{equation*}
\left(F^+_n\right)^{-1} F^+_{n+1}
= F^-_{n,m} \left(F^m_n\right)^{-1} F^m_{n+1} \big(F^-_{n+1,m}\big)^{-1}
= F^-_{n,m} U^m_n \big(F^-_{n+1,m}\big)^{-1},
\end{equation*}
where $U^m_n$ is given by \eqref{eq:DU}.
 Since $U^m_n$ has the form
$U^m_n = \sum_{k=0}^3 \lambda^k U^k_{n,m}$ and
$F^-_{n,m}$ takes values in $\LSLN$, we have
\begin{equation*}
\xi^+_n
= \left(F^+_n\right)^{-1} F^+_{n+1}
= X^0_n + \lambda X^1_n + \lambda^2 X^2_n + \lambda^3 X^3_n.
\end{equation*}
With the expansions
\begin{align*}
F^-_{n,m} &= I^0_{n,m} + \lambda^{-1} I^1_{n,m} + \lambda^{-2} I^2_{n,m}
+ \lambda^{-3} I^3_{n,m} + \cdots,\\
\big(F^-_{n,m}\big)^{-1} &=
J^0_{n,m} + \lambda^{-1} J^1_{n,m} + \lambda^{-2} J^2_{n,m}
+ \lambda^{-3} J^3_{n,m} + \cdots,
\end{align*}
it is easy to see that $X^0_n$, $X^1_n$, $X^2_n$, $X^3_n$ are computed as
\begin{align*}
X^0_n =\;&
I^0_{n,m} U^3_{n,m} J^3_{n+1,m}
+ I^1_{n,m} U^3_{n,m} J^2_{n+1,m}
+ I^2_{n,m} U^3_{n,m} J^1_{n+1,m}
+ I^3_{n,m} U^3_{n,m} J^0_{n+1,m}\\
&+ I^0_{n,m} U^2_{n,m} J^2_{n+1,m}
+ I^1_{n,m} U^2_{n,m} J^1_{n+1,m}
+ I^2_{n,m} U^2_{n,m} J^0_{n+1,m}\\
&+ I^0_{n,m} U^1_{n,m} J^1_{n+1,m}
+ I^1_{n,m} U^1_{n,m} J^0_{n+1,m}
+ I^0_{n,m} U^0_{n,m} J^0_{n+1,m},\\
X^1_n =\;&
I^0_{n,m} U^3_{n,m} J^2_{n+1,m}
+ I^1_{n,m} U^3_{n,m} J^1_{n+1,m}
+ I^2_{n,m} U^3_{n,m} J^0_{n+1,m}\\
&+ I^0_{n,m} U^2_{n,m} J^1_{n+1,m}
+ I^1_{n,m} U^2_{n,m} J^0_{n+1,m}
+ I^0_{n,m} U^1_{n,m} J^0_{n+1,m},\\
X^2_n =\;& I^0_{n,m} U^3_{n,m} J^1_{n+1,m}
+ I^1_{n,m} U^3_{n,m} J^0_{n+1,m}
+ I^0_{n,m} U^2_{n,m} J^0_{n+1,m},\\
X^3_n =\;& I^0_{n,m} U^3_{n,m} J^0_{n+1,m}.
\end{align*}
From Proposition \ref{prop:twisted-Q},
every one of $\big\{I^0_{n,m}, J^0_{n,m}, I^3_{n,m}, J^3_{n,m}\big\}$,
$\big\{I^1_{n,m}, J^1_{n,m}\big\}$,
$\big\{I^2_{n,m}, J^2_{n,m}\big\}$ has the following form
\begin{equation*}
\diag \left(\ast, \ast, \ast\right),\quad
\begin{bmatrix}
0&*&0\\
0&0&*\\
*&0&0
\end{bmatrix},\quad
\begin{bmatrix}
0&0&*\\
*&0&0\\
0&*&0
\end{bmatrix},
\end{equation*}
respectively.
Since $\xi^+_n$ takes values in $\LSLPN$, thus $X^0_n = \id$.
Noticing that $U^1_{n,m}$, $U^2_{n,m}$, $U^3_{n,m}$ are given
by \eqref{eq:DU},
it readily follows that the coefficients $X^k_n$ have the form
\begin{gather*}
%X^0 = \diag \left(y^{11}_n, y^{22}_n, y^{33}_n\right),\quad
X^3_n = \diag \left(x^{11}_n, 0, 0\right), \quad
X^1_n =
\begin{bmatrix}
0 & 0 & x^{13}_n\\
x^{21}_n & 0 & 0\\
0 & x^{32}_n & 0
\end{bmatrix},\quad
X^2_n =
\begin{bmatrix}
0 & x^{12}_n & 0\\
0 & 0 & 0\\
x^{31}_n & 0 & 0
\end{bmatrix},
\end{gather*}
where $x^{ij}_n$ are some functions in $n$.
Thus we have
\begin{equation*}
\xi^+_n =
\begin{bmatrix}
1 + \lambda^3 x^{11}_n & \lambda^2 x^{12}_n & \lambda x^{13}_n\\
\lambda x^{21}_n & 1 & 0\\
\lambda^2 x^{31} & \lambda x^{32}_n & 1
\end{bmatrix}.
\end{equation*}
We now consider the twisted condition \eqref{twisted-T},
namely ${}^\mathrm{t} \xi \left(-\lambda\right) T \xi \left(\lambda\right) = T$.
It is easy to see that the twisted condition is equivalent to the system
\begin{equation*}
x^{13}_n = - H x^{32}_n,\quad
x^{31}_n = x^{21}_n x^{32}_n, \quad
x^{12}_n = - \frac{H}{2} \left(x^{32}_n\right)^2, \quad
x^{11}_n = -\frac{H}{2} \left(x^{32}_n\right)^2 x^{21}_n.
\end{equation*}
Thus we have
\begin{equation*}
\xi^+_n =
\begin{bmatrix}
1 - \lambda^3 \frac{H}{2} \left(x^{32}_n\right)^2 x^{21}_n &
- \lambda^2 \frac{H}{2} \left(x^{32}_n\right)^2 &
- \lambda H x^{32}_n\\
\lambda x^{21}_n & 1 & 0\\
\lambda^2 x^{21}_n x^{32}_n & \lambda x^{32}_n & 1
\end{bmatrix}
\end{equation*}
and the expression \eqref{eq:DMC}
on rewriting $x^{32}_n = \epsilon \alpha_{n + 1}$ and
$x^{21}_n = \epsilon \beta_{n + 1}$.
We write $P \left(i, j\right)$ for the $\left(i, j\right)$-entry
of a matrix $P$,
and show that $\alpha_{n+1}$ has no zeros as follows.
We compute $x^{32}_n$, $x^{12}_n$ and $X^0_n(2, 2)$
so that we have
\begin{align}
x^{32}_n &=\label{x^{32}}
\epsilon \left(-2 + \epsilon H I^1_{n, m}(3, 1)\right) C^m_n,\\
x^{12}_n &=\label{x^{12}}
\epsilon^2 H I^0_{n, m}(1, 1)\, C^m_n,\\
X^0_n(2, 2) &=\label{X^0(2,2)}
\left(- \frac{2}{I^0_{n, m}(1, 1)} + \epsilon^2 H I^2_{n, m}(2, 1)
- 2 \epsilon I^2_{n, m}(2, 3)\right) C^m_n,
\end{align}
where we set
\begin{equation*}
a^m_n = \frac{\omega^m_{n+1} g^m_{n+1}}{\omega^m_n} \neq0,\quad
b^m_n = A^m_{n+1} g^m_{n+1},\quad
C^m_n = - \frac{1 + a^m_n b^m_n J^0_{n+1, m}(1, 1) J^1_{n+1,m}(1, 2)}%
{2 a^m_n J^0_{n+1,m} (1,1)}.
\end{equation*}
Because of the condition $X^0_n(2,2) = 1$,
expressions \eqref{x^{32}} and \eqref{X^0(2,2)} imply that
$x^{32}_n$ has no zeros if $H = 0$.
When $H = -1$,
expressions \eqref{x^{12}} and \eqref{X^0(2,2)} imply $x^{32}_n \neq 0$
because $x^{12}_n = - (H/2) \left(x^{32}_n\right)^2$.
Similarly,
$\left(G^-_m\right)^{-1} G^-_{m+1}$ can be computed as in \eqref{eq:DMC2}
with a nowhere vanishing function $\rho_m$.

Conversely,
let $\left(F^+_n, G^-_m\right)$ be a pair of solutions of \eqref{eq:DMC0}
such that $F^+_0 = G^-_0 = \id$.
We write
\begin{equation*}
\xi^+_n = \sum_{j=0}^3 \lambda^j X^j_n,\quad
\xi^-_m = \sum_{j=0}^3 \lambda^{-j} Y^j_m,
\end{equation*}
where coefficient matrices $X^j$ and $Y^j$ are defined by
\eqref{eq:DMC} and \eqref{eq:DMC2}.
Consider the Birkhoff decomposition of $\left(G^-_m\right)^{-1} F^+_n$
near $\left(n, m\right) = \left(0, 0\right)$
as $\left(G^-_m\right)^{-1} F^+_n
= V^+_{n,m} \left(V^-_{n,m}\right)^{-1}$ and define $\hat{F}^m_n
= F^+_n V^-_{n,m}
= G^-_m V^+_{n,m}$.
We express
\begin{gather*}
V^-_{n,m} = \sum_{j=0}^\infty \lambda^{-j} K^j_{n,m},\quad
V^+_{n,m} = \sum_{j=0}^\infty \lambda^j M^j_{n,m}.
\end{gather*}
Their inverses are
\begin{gather*}
\left(V^-_{n,m}\right)^{-1} = \sum_{j=0}^\infty \lambda^{-j} L^j_{n,m},\quad
\left(V^+_{n,m}\right)^{-1} = \sum_{j=0}^\infty \lambda^j N^j_{n,m},
\end{gather*}
with $K^0_{n,m} L^0_{n,m} = \id$ and $M^0_{n,m} = N^0_{n,m} = \id$,
and for all $j \geq 0$ it holds that
\begin{equation*}
\sum_{k=0}^{j+1} K^k_{n,m} L^{j+1-k}_{n,m} = 0,\quad
\sum_{k=0}^{j+1} M^k_{n,m} N^{j+1-k}_{n,m} = 0.
\end{equation*}
Namely the matrices $L^{j+1}$ and $N^{j+1}$ are computed as
\begin{align*}
L^1_{n,m} &=
- L^0_{n,m} K^1_{n,m} L^0_{n,m},\\
L^2_{n,m} &= - L^0_{n,m}
\left(K^2_{n,m} - K^1_{n,m} L^0_{n,m} K^1_{n,m}\right) L^0_{n,m},\\
L^3_{n,m} &=
- L^0_{n,m}
\left(K^3_{n,m}
- K^1_{n,m} L^0_{n,m} K^2_{n,m}
- K^2_{n,m} L^0_{n,m} K^1_{n,m}\right.\\
&\qquad\qquad
\left.
+ K^1_{n,m} L^0_{n,m} K^1_{n,m} L^0_{n,m} K^1_{n,m}\right) L^0_{n,m},\\
N^1_{n,m} &= - M^1_{n,m},\\
N^2_{n,m} &= - M^2_{n,m} + \left(M^1_{n,m}\right)^2,\\
N^3_{n,m} &= - M^3_{n,m} + M^1_{n,m} M^2_{n,m}
+ M^2_{n,m} M^1_{n,m} - \left(M^1_{n,m}\right)^3,
\end{align*}
and so forth.
Further,
from the twisted condition \eqref{twisted-T},
it holds that
\begin{equation*}
\left(-1\right)^j {}^\mathrm{t} L^j_{n,m} T = T K^j_{n,m},\quad
\left(-1\right)^j {}^\mathrm{t} N^j_{n,m} T = T M^j_{n,m}
\end{equation*}
for all $j \geq 0$.
In particular,
setting $j=0$,
we have that
\begin{equation*}
K^0_{n,m} = \diag \left(k^m_n, 1/k^m_n , 1\right),\quad
L^0_{n,m} = %\diag \left(1/k^m_n, k^m_n ,1/l^m_n\right),
\left(K^0_{n,m}\right)^{-1},
\end{equation*}
where $k$ is some sequence which has no zeros.
For higher $j$, we have that
\begin{align*}
K^1_{n,m} &=
\begin{bmatrix}
0 & \kappa^{12}_{n,m} & 0\\
0 & 0 & - \frac{H}{k^m_n} \kappa^{31}_{n,m}\\
\kappa^{31}_{n,m} & 0 & 0
\end{bmatrix},\\
K^2_{n,m} &=
\begin{bmatrix}
0 & 0 & H \left(k^m_n \kappa^{32}_{n,m}
- \kappa^{12}_{n,m} \kappa^{31}_{n,m}\right)\\
- \frac{H}{2 k^m_n} \left(\kappa^{31}_{n,m}\right)^2 & 0 & 0\\
0 & \kappa^{32}_{n,m} & 0
\end{bmatrix},\\
K^3_{n,m} &=
\diag \left(\left(k^m_n\right)^2 \kappa^{22}_{n,m}
+ \frac{H}{2} \kappa^{31}_{n,m} \left(2 k^m_n \kappa^{32}_{n,m}
- \kappa^{12}_{n,m} \kappa^{31}_{n,m}\right),\
\kappa^{22}_{n,m},\
\kappa^{33}_{n,m}\right),\\
M^1_{n,m} &=
% \begin{bmatrix}
% 0 & 0 & - H M^1_{n,m}(3,2)\\
% M^1_{n,m}(2,1) & 0 & 0\\
% 0 & M^1_{n,m}(3,2) & 0
% \end{bmatrix},\\
\begin{bmatrix}
0 & 0 & - H \mu^{32}_{n,m}\\
\mu^{21}_{n,m} & 0 & 0\\
0 & \mu^{32}_{n,m} & 0
\end{bmatrix},\\
M^2_{n,m} &=
% \begin{bmatrix}
% 0 & - \frac{H}{2} \left(M^1_{n,m}(3,2)\right)^2 & 0\\
% 0 & 0 & H \left(M^2_{n,m}(3,1) - M^1_{n,m}(2,1) M^1_{n,m}(3,2)\right)\\
% M^2_{n,m}(3,1) & 0 & 0
% \end{bmatrix},\\
\begin{bmatrix}
0 & - \frac{H}{2} \left(\mu^{32}_{n,m}\right)^2 & 0\\
0 & 0 & H \left(\mu^{31}_{n,m} - \mu^{21}_{n,m} \mu^{32}_{n,m}\right)\\
\mu^{31}_{n,m} & 0 & 0
\end{bmatrix},\\
M^3_{n,m} &=
% \diag \left(M^3_{n,m}(1,1),
% M^3_{n,m}(1,1)
% + \frac{H}{2} M^1_{n,m}(3,2) \left(2 M^2_{n,m}(3,1)
% - M^1_{n,m}(2,1) M^1_{n,m}(3,2)\right),
% M^3_{n,m}(3,3)\right).
\diag \left(\mu^{11}_{n,m},\
\mu^{11}_{n,m}
+ \frac{H}{2} \mu^{32}_{n,m} \left(2 \mu^{31}_{n,m}
- \mu^{21}_{n,m} \mu^{32}_{n,m}\right),\
\mu^{33}_{n,m}\right),
\end{align*}
and so on.
Here $\kappa^{ij}$ and $\mu^{ij}$ are some sequences in $n$, $m$.
Now we are ready to compute the Maurer-Cartan form of $\hat{F}$.
As for $\big(\hat{F}^m_n\big)^{-1} \hat{F}^m_{n+1}$,
we have
\begin{align}
\big(\hat{F}^m_n\big)^{-1} \hat{F}^m_{n+1}
&=\label{eq:dMC-Fhat-Uhat-1}
\left(V^+_{n,m}\right)^{-1} V^+_{n+1,m},\\
\big(\hat{F}^m_n\big)^{-1} \hat{F}^m_{n+1}
&=\label{eq:dMC-Fhat-Uhat-2}
\left(V^-_{n,m}\right)^{-1} \xi^+_n V^-_{n+1,m}.
\end{align}
Comparing these two expressions,
it readily follows that there exist matrices such that
\begin{align*}
\big(\hat{F}^m_n\big)^{-1} \hat{F}^m_{n+1}
= \id + \lambda \hat{U}^1_{n,m}
+ \lambda^2 \hat{U}^2_{n,m} + \lambda^3 \hat{U}^3_{n,m}.
\end{align*}
Similarly,
from
\begin{align}
\big(\hat{F}^m_n\big)^{-1} \hat{F}^{m+1}_n
&=\label{eq:dMC-Fhat-Vhat-1}
\left(V^-_{n,m}\right)^{-1} V^-_{n,m+1},\\
\big(\hat{F}^m_n\big)^{-1} \hat{F}^{m+1}_n
&=\label{eq:dMC-Fhat-Vhat-2}
\left(V^+_{n,m}\right)^{-1} \xi^-_m V^+_{n,m+1},
\end{align}
we have that
\begin{equation*}
\big(\hat{F}^m_n\big)^{-1} \hat{F}^{m+1}_n
= \hat{V}^0_{n,m} + \lambda^{-1} \hat{V}^1_{n,m}
+ \lambda^{-2} \hat{V}^2_{n,m} + \lambda^{-3} \hat{V}^3_{n,m}.
\end{equation*}
From \eqref{eq:dMC-Fhat-Uhat-2} and \eqref{eq:dMC-Fhat-Vhat-2},
it readily follows that the coefficient matrices are of the form
\begin{gather}
\hat{U}^1_{n,m}
=\label{eq:dMC-Fhat-Uhat-2:123}
\begin{bmatrix}
0 & 0 & \ast\\
\ast & 0 & 0\\
0 & \ast & 0
\end{bmatrix},\quad
\hat{U}^2_{n,m}
=
\begin{bmatrix}
0 & \ast & 0\\
0 & 0 & 0\\
\ast & 0 & 0
\end{bmatrix},\quad
\hat{U}^3_{n,m}
=
\diag \left(\ast, 0, 0\right),\\
\hat{V}^1_{n,m}
=\label{eq:dMC-Fhat-Vhat-2:123}
\begin{bmatrix}
0 & \ast & 0\\
0 & 0 & \ast\\
\ast & 0 & 0
\end{bmatrix},\quad
\hat{V}^2_{n,m}
=
\begin{bmatrix}
0 & 0 & 0\\
\ast & 0 & 0\\
0 & \ast & 0
\end{bmatrix},\quad
\hat{V}^3_{n,m}
=
\diag \left(0, \ast, 0\right),
\end{gather}
and $\hat{V}^0_{n,m}$ is diagonal.
On the other hand,
from \eqref{eq:dMC-Fhat-Uhat-1},
we have
\begin{align}
\hat{U}^1_{n,m}
&=\nonumber
M^1_{n+1,m} + N^1_{n,m}\\
&=\nonumber
M^1_{n+1,m} - M^1_{n,m}\\
&=\nonumber
\begin{bmatrix}
0 & 0 & - H u^{13}_{n,m}\\
u^{21}_{n,m} & 0 & 0\\
0 & u^{13}_{n,m} & 0
\end{bmatrix},\\
\hat{U}^2_{n,m}
&=\nonumber
M^2_{n+1,m} + N^1_{n,m} M^1_{n+1,m} + N^2_{n,m}\\
&=\nonumber
M^2_{n+1,m} - M^2_{n,m} - M^1_{n,m} \hat{U}^1_{n,m}\\
&=\label{eq:dMC-Fhat-Uhat-1:2}
\begin{bmatrix}
0 & - \frac{H}{2} \left(u^{13}_{n,m}\right)^2 &  0\\
0 & 0 & H u^{23}_{n,m}\\
u^{31}_{n,m} & 0 & 0
\end{bmatrix},\\
\hat{U}^3_{n,m}
&=\nonumber
M^3_{n+1,m} + N^1_{n,m} M^2_{n+1,m}
+ N^2_{n,m} M^1_{n+1,m} + N^3_{n,m}\\
&=\nonumber
M^3_{n+1,m} - M^3_{n,m} - M^1_{n,m} \hat{U}^2_{n,m}
- M^2_{n,m} \hat{U}^1_{n,m}\\
&=\label{eq:dMC-Fhat-Uhat-1:3}
\diag \left(u^{11}_{n,m}, u^{22}_{n,m}, \ast\right),
\end{align}
where
\begin{align*}
u^{13}_{n,m} &= \mu^{32}_{n+1,m} - \mu^{32}_{n,m},\\
u^{21}_{n,m} &= \mu^{21}_{n+1,m} - \mu^{21}_{n,m},\\
u^{23}_{n,m} &= \mu^{31}_{n+1,m} - \mu^{31}_{n,m}
- \mu^{32}_{n+1,m} u^{21}_{n,m},\\
u^{31}_{n,m} &= \mu^{31}_{n+1,m} - \mu^{31}_{n,m}
- \mu^{32}_{n,m} u^{21}_{n,m},\\
u^{11}_{n,m} &= \mu^{11}_{n+1,m} - \mu^{11}_{n,m}
+ \left(H/2\right) \mu^{32}_{n,m}
\left(\mu^{32}_{n+1,m} + u^{13}_{n,m}\right) u^{21}_{n,m},\\
u^{22}_{n,m} &= \mu^{11}_{n+1,m} - \mu^{11}_{n,m}
+ \left(H/2\right) \mu^{32}_{n+1,m}
\left(\mu^{31}_{n+1,m} - \mu^{31}_{n,m} + u^{23}_{n,m}\right).
\end{align*}
Similarly,
from \eqref{eq:dMC-Fhat-Vhat-1},
we have
\begin{align*}
\hat{V}^0_{n,m} &= L^0_{n,m} K^0_{n,m+1}
= \diag
\left(\frac{k^{m+1}_n}{k^m_n},
\frac{k^m_n}{k^{m+1}_n},
1\right),
\end{align*}
and hence we have that
\begin{align}
\hat{V}^1_{n,m}
&=\nonumber
L^0_{n,m} K^1_{n,m+1} + L^1_{n,m} K^0_{n,m+1}\\
&=\nonumber
L^0_{n,m} \left(K^1_{n,m+1} - K^1_{n,m} \hat{V}^0_{n,m}\right)\\
&=\nonumber
\begin{bmatrix}
0 & v^{12}_{n,m} & 0\\
0 & 0 & - H v^{23}_{n,m}\\
\frac{k^{m+1}_n}{k^m_n} v^{23}_{n,m} & 0 & 0
\end{bmatrix},\\
\hat{V}^2_{n,m}
&=\nonumber
L^0_{n,m} K^2_{n,m+1} + L^1_{n,m} K^1_{n,m+1} + L^2_{n,m} K^0_{n,m+1}\\
&=\nonumber
L^0_{n,m} \left(K^2_{n,m+1} - K^2_{n,m} \hat{V}^0_{n,m}
- K^1_{n,m} \hat{V}^1_{n,m}\right)\\
&=\label{eq:dMC-Fhat-Vhat-1:2}
\begin{bmatrix}
0 & 0 & H v^{13}_{n,m}\\
- \frac{H}{2} \frac{k^{m+1}_n}{k^m_n}
\left(v^{23}_{n,m}\right)^2 & 0 & 0\\
0 & v^{32}_{n,m} & 0
\end{bmatrix},\\
\hat{V}^3_{n,m}
&=\nonumber
L^0_{n,m} K^3_{n,m+1} + L^1_{n,m} K^2_{n,m+1}
+ L^2_{n,m} K^1_{n,m+1} + L^3_{n,m} K^0_{n,m+1}\\
&=\nonumber
L^0_{n,m} \left(K^3_{n,m+1} - K^3_{n,m} \hat{V}^0_{n,m}
- K^2_{n,m} \hat{V}^1_{n,m}
- K^1_{n,m} \hat{V}^2_{n,m}\right)\\
&=\label{eq:dMC-Fhat-Vhat-1:3}
\diag \left(v^{11}_{n,m}, v^{22}_{n,m}, \ast\right),
\end{align}
where
\begin{align*}
v^{12}_{n,m} &=
\frac{\kappa^{12}_{n,m+1}}{k^m_n}
- \frac{\kappa^{12}_{n,m}}{k^{m+1}_n},\\
v^{23}_{n,m} &= \frac{k^m_n}{k^{m+1}_n} \kappa^{31}_{n,m+1}
- \kappa^{31}_{n,m},\\
% v^{31}_{n,m} &= \frac{k^{m+1}_n}{k^m_n}
% \frac{1}{l^m_n}
% \left(\frac{k^m_n}{k^{m+1}_n} \kappa^{31}_{n,m+1}
% - \kappa^{31}_{n,m}\right),\\
v^{13}_{n,m} &=
\frac{k^{m+1}_n}{k^m_n}
\left(\kappa^{32}_{n,m+1}
- \frac{k^m_n}{k^{m+1}_n} \kappa^{32}_{n,m}\right)
- \kappa^{31}_{n,m+1}
v^{12}_{n,m},\\
% v^{21}_{n,m} &= \frac{k^{m+1}_n}{k^m_n}
% \left(\frac{k^m_n}{k^{m+1}_n} \kappa^{31}_{n,m+1}
% - \kappa^{31}_{n,m}\right)^2,\\
v^{32}_{n,m} &=
\kappa^{32}_{n,m+1}
- \frac{k^m_n}{k^{m+1}_n} \kappa^{32}_{n,m}
- \kappa^{31}_{n,m}
 v^{12}_{n,m},\\
v^{11}_{n,m} &=
\frac{k^{m+1}_n}{k^m_n}
\left(
k^{m+1}_n \kappa^{22}_{n,m+1} - k^m_n \kappa^{22}_{n,m}
+ \frac{H}{2} \kappa^{31}_{n,m+1}
\left(
\kappa^{32}_{n,m+1} - \frac{k^m_n}{k^{m+1}_n} \kappa^{32}_{n,m}
+ \frac{k^m_n}{k^{m+1}_n}
v^{13}_{n,m}\right)\right),\\
v^{22}_{n,m} &=
\frac{k^m_n}{k^{m+1}_n}
\left(k^{m+1}_n \kappa^{22}_{n,m+1} - k^m_n \kappa^{22}_{n,m}\right)
+ \frac{H}{2} \kappa^{31}_{n,m}
\left(\kappa^{32}_{n,m+1}
- \frac{k^m_n}{k^{m+1}_n} \kappa^{32}_{n,m}
+ v^{32}_{n,m}\right).
\end{align*}
For $H \neq 0$,
comparing \eqref{eq:dMC-Fhat-Uhat-1:2} and \eqref{eq:dMC-Fhat-Uhat-2:123},
we have $u^{23}_{n,m} = 0$,
which implies $u^{31}_{n,m} = u^{13}_{n,m} u^{21}_{n,m}$.
For $H=0$,
by using \eqref{eq:dMC-Fhat-Uhat-2},
we have $u^{31}_{n,m} = u^{13}_{n,m} u^{21}_{n,m}$.
Comparing \eqref{eq:dMC-Fhat-Uhat-1:3} and \eqref{eq:dMC-Fhat-Uhat-2:123},
we have $u^{22}_{n,m} = 0$,
which implies
\begin{align*}
u^{11}_{n,m}
&= \frac{H}{2}
\mu^{32}_{n,m}
\left(\mu^{32}_{n+1,m} + u^{13}_{n,m}\right) u^{21}_{n,m}
- \frac{H}{2} \mu^{32}_{n+1,m}
\left(\mu^{31}_{n+1,m} - \mu^{31}_{n,m}\right)\\
&= \frac{H}{2}
\left(\mu^{32}_{n+1,m} - u^{13}_{n,m}\right)
\left(\mu^{32}_{n+1,m} + u^{13}_{n,m}\right) u^{21}_{n,m}
- \frac{H}{2} \left(\mu^{32}_{n+1,m}\right)^2 u^{21}_{n,m}\\
&= - \frac{H}{2} \left(u^{13}_{n,m}\right)^2 u^{21}_{n,m}.
\end{align*}
For $H \neq 0$,
comparing \eqref{eq:dMC-Fhat-Vhat-1:2} and \eqref{eq:dMC-Fhat-Vhat-2:123},
we have $v^{13}_{n,m} = 0$,
which implies $v^{32}_{n,m} = v^{23}_{n,m} v^{12}_{n,m}$.
For $H=0$,
by using \eqref{eq:dMC-Fhat-Vhat-2},
we have $v^{32}_{n,m} = v^{23}_{n,m} v^{12}_{n,m}$.
Comparing \eqref{eq:dMC-Fhat-Vhat-1:3} and \eqref{eq:dMC-Fhat-Vhat-2:123},
we have $v^{11}_{n,m} = 0$,
which implies
\begin{align*}
v^{22}_{n,m}
&= - \frac{H}{2}
\frac{k^m_n}{k^{m+1}_n}
\kappa^{31}_{n,m+1}
\left(
\kappa^{32}_{n,m+1} - \frac{k^m_n}{k^{m+1}_n} \kappa^{32}_{n,m}
\right)
+ \frac{H}{2} \kappa^{31}_{n,m}
\left(\kappa^{32}_{n,m+1}
- \frac{k^m_n}{k^{m+1}_n} \kappa^{32}_{n,m}
+ v^{32}_{n,m}\right)\\
&= - \frac{H}{2}
\left(\frac{k^m_n}{k^{m+1}_n} \kappa^{31}_{n,m+1}
- \kappa^{31}_{n,m}\right)
\left(
\kappa^{32}_{n,m+1} - \frac{k^m_n}{k^{m+1}_n} \kappa^{32}_{n,m}
\right)
+ \frac{H}{2} \kappa^{31}_{n,m}  v^{32}_{n,m}\\
&= - \frac{H}{2}
 v^{23}_{n,m}
\left(
 v^{32}_{n,m} + \kappa^{31}_{n,m} v^{12}_{n,m}
\right)
+ \frac{H}{2} \kappa^{31}_{n,m} v^{32}_{n,m}\\
&= - \frac{H}{2} \left(v^{23}_{n,m}\right)^2 v^{12}_{n,m}.
\end{align*}
Finally,
again from \eqref{eq:dMC-Fhat-Uhat-2} and \eqref{eq:dMC-Fhat-Vhat-2},
we have
\begin{align*}
\left(u^{13}_{n,m}\right)^2 &=
\frac{\epsilon^2 \left(\alpha_{n+1}\right)^2}{k^m_n}
\left(\frac{1}{k^m_{n+1}} + \epsilon \beta_{n+1} \kappa^{12}_{n+1,m}\right),\\
\left(v^{23}_{n,m}\right)^2 &=
\frac{k^m_n}{k^{m+1}_n}
\delta^2 \left(\rho_{m+1}\right)^2
\left(1
+ \delta \sigma_{m+1}
\mu^{21}_{n,m+1}
\right).
\end{align*}
On the other hand,
a straightforward computation shows that $(2, 2)$-entry
of the constant coefficient with respect to $\lambda$
for $(\hat F^m_n)^{-1} \hat F^m_{n+1}$
and $(1, 1)$-entry of the constant coefficient with respect to $\lambda$
for $(\hat F^m_n)^{-1} \hat F^{m+1}_{n}$ can be respectively computed as
\begin{gather*}
1 = \left(\frac{1}{2}
+ \frac{\epsilon \alpha H \kappa^{31}_{n, m}}{4 k^m_n}\right)
\left(\frac{1}{k^m_{n+1}}
+ \epsilon \beta_{n+1} \kappa_{n+1, m}^{12}\right),\\
0 \neq \frac{k^{m+1}_n}{k^{m}_n}
= \left(1 + \frac{\delta \rho H}{2} \mu^{32}_{n,m}\right)^2
\left(1+ \delta \sigma_{m+1} \mu^{21}_{n,m+1}\right).
\end{gather*}
Thus $1/k^m_{n+1} + \epsilon \beta_{n+1} \kappa^{12}_{n+1,m}$
and $1 + \delta \sigma_{m+1} \mu^{21}_{n,m+1}$ have no zeros near $(0, 0)$,
which implies that $u^{13}_{n, m}$ and $v^{23}_{n, m}$
never vanish near $(n, m)= (0, 0)$.

Thus,
coefficient matrices $\hat{U}^j_{n,m}$, $\hat{V}^j_{n,m}$ satisfy
the assumption of Proposition \ref{prop:extendtosurf}.
Therefore by Proposition \ref{prop:extendtosurf}
there exists a diagonal gauge $D^m_n$ such that
$\left(D^0_0\right)^{-1} \hat{F}^m_n D^m_n$ is the extended frame
for some discrete indefinite affine sphere.
\end{proof}
\begin{Remark}
In case of discrete proper affine spheres $\left(H = -1\right)$,
the functions $A^m_n$ and $B^m_n$ do not have simple expressions among
the functions $\alpha_n, \beta_n, \rho_m$ and $\sigma_m$.
On the other hand,
in case of discrete improper affine spheres $\left(H = 0\right)$,
they are simply represented as \eqref{explicit:d-omegaAB} below.
\end{Remark}

\subsection{discrete indefinite improper affine spheres}

For the rest of this paper,
we shall be concerned only with discrete indefinite
improper affine spheres $\left(H = 0\right)$.
Equations in \eqref{eq:DGC-AB} are simplified as
\begin{equation*}
A^{m+1}_{n+1} = A^m_{n+1},\quad
B^{m+1}_{n+1} = B^{m+1}_n,
\end{equation*}
which indicate that $A$ and $B$ depend only on $n$ and $m$, respectively.
Hence we write $A_n$ for $A^m_n$ and $B_m$ for $B^m_n$.
Thus the discrete Liouville equation \eqref{eq:DGC-omega} is written as
\begin{equation}\label{eq:dLiouville}
\omega^{m+1}_{n+1} \omega^m_n - \omega^m_{n+1} \omega^{m+1}_n
+ \epsilon \delta A_{n+1} B_{m+1} = 0.
\end{equation}
We now introduce a notation of summation of a sequence $x$ as
\begin{equation*}
{\sum}_k^n\, x_k
=
\begin{cases}
\sum_{k=1}^n x_k & \left(n \geq 1\right)\\
0 & \left(n=0\right)\\
- \sum_{k=n+1}^0 x_k & \left(n \leq -1\right).
\end{cases}
\end{equation*}
It holds for all integers $n$ that
\begin{gather*}
{\textstyle\sum}_k^n x_k - {\textstyle\sum}_k^{n-1} x_k = x_n,\\
{\textstyle\sum}_k^{n-1} x_k - {\textstyle\sum}_k^n x_{k-1} = - x_0,
\end{gather*}
and hence
${\sum}_k^n \left(x_k - x_{k-1}\right) = x_n - x_0$.
In particular we have a formula of summation by parts
\begin{equation}\label{eq:summertionbyparts}
{\textstyle\sum}_k^n x_k \left(y_k - y_{k-1}\right)
= - x_0 y_0
+ x_n y_n - {\textstyle\sum}_k^n \left(x_k - x_{k-1}\right) y_{k-1},
\end{equation}
which holds for all integers $n$.
Then the ordinary difference equations \eqref{eq:DMC0} can be
explicitly computed as follows.
\begin{Lemma}
Let $H = 0$.
Then the pair of solutions $\left(F^+_n,\, G^-_m\right)$
to the system \eqref{eq:DMC0}--\eqref{eq:DMC2}
with the initial condition $F^+_0 = G^-_0 = \id$ is explicitly given by
\begin{equation}\label{eq:DESol}
F^+_n =
\begin{bmatrix}
1 & 0 & 0\\
b_n \lambda & 1 & 0\\
c_n \lambda^2 & a_n \lambda & 1
\end{bmatrix},\quad
G^-_m =
\begin{bmatrix}
1 & s_m \lambda^{-1} & 0\\
0 & 1 & 0\\
r_m \lambda^{-1} & t_m \lambda^{-2} & 1
\end{bmatrix},
\end{equation}
where $a$, $b$, $c$, $r$, $s$, $t$ are defined as
\begin{gather}
a_n = \epsilon {\sum}_k^n \alpha_{k},\quad
b_n = \epsilon {\sum}_k^n \beta_{k},\quad
c_n = \epsilon {\sum}_k^{n} a_{k} \beta_{k},\label{def:d-abc}\\
r_m = \delta {\sum}_k^m \rho_{k},\quad
s_m = \delta {\sum}_k^m \sigma_{k},\quad
t_m = \delta {\sum}_k^{m} r_{k} \sigma_{k}.\label{def:d-rst}
\end{gather}
Moreover,
assume $1 - b_n s_m \neq 0$,
and define $V^+_{n,m} \in \LSLPN$ and $V^-_{n,m} \in \LSLN$
by the Birkhoff decomposition of $\left(G^-_m\right)^{-1} F^+_n$ as
\begin{equation}\label{eq:DdFG}
\left(G^-_m\right)^{-1} F^+_n
= V^+_{n,m} \left(V^-_{n,m}\right)^{-1}.
\end{equation}
Then $V^\pm$ are explicitly given as
\begin{align}
V^+_{n,m}
&=\label{eq:d-VPM:P}
\begin{bmatrix}
1 & 0 & 0\\
b_n \left(1 - b_n s_m\right)^{-1} \lambda & 1 & 0\\
c_n \left(1 - b_n s_m\right)^{-1} \lambda^2 &
\left(a_n - s_m
\left(a_n b_n - c_n\right)\right) \lambda & 1
\end{bmatrix},\\
V^-_{n,m}
&=\label{eq:d-VPM:M}
\begin{bmatrix}
\left(1 - b_n s_m\right)^{-1} & s_m \lambda^{-1} & 0\\
0 & 1 - b_n s_m & 0\\
\left(r_m + b_n t_m
\left(1 - b_n s_m\right)^{-1}\right) \lambda^{-1} &
t_m \lambda^{-2} & 1
\end{bmatrix}.
\end{align}
\end{Lemma}
\begin{proof}
It is easy to check that matrices \eqref{eq:DESol} satisfy the system
\eqref{eq:DMC0}--\eqref{eq:DMC2} with $H = 0$ and $F^+_0 = G^-_0 = \id$.
The decomposition \eqref{eq:DdFG} is also easily checked.
\end{proof}
\begin{Theorem}
Let $\epsilon$, $\delta$ be positive numbers,
and $\alpha$, $\beta$, $\rho$, $\sigma$ be functions in one variable,
and $F^+$, $G^-$, $V^+$, $V^-$ be the loops given by
\eqref{eq:DESol}, \eqref{eq:d-VPM:P}, \eqref{eq:d-VPM:M}.
Define $a$, $b$, $c$, $r$, $s$, $t$ by \eqref{def:d-abc}, \eqref{def:d-rst},
and $\hat F$ by $\hat{F}^m_n = F^+_n V^-_{n,m} = G^-_m V^+_{n,m}$.
We assume that sequences $\alpha$, $\rho$ and $1 - b s$ have no zeros.
Then there exists a diagonal matrix $D^m_n$ such that
$\left(D^0_0\right)^{-1} \hat{F}^m_n D^m_n$ is the extended frame
of some discrete indefinite improper affine sphere $f$,
whose data solving the discrete Liouville equation \eqref{eq:dLiouville}
are given as
\begin{equation}\label{explicit:d-omegaAB}
\omega^m_n =\left(1 - b_n s_m\right) \alpha_{n+1} \rho_{m+1},\quad
A_n = \alpha_{n+1} \alpha_n \beta_n,\quad
B_m = \rho_{m+1} \rho_m \sigma_m.
\end{equation}
Moreover,
the associated family of $f$ is given by the representation formula
\begin{equation}\label{formula:representation-d-improper}
f^m_n =
\begin{bmatrix}
\lambda a_n + \lambda^{-2} \left(r_m s_m - t_m\right)\\
\lambda^2 \left(a_n b_n - c_n\right) + \lambda^{-1} r_m\\
a_n r_m
- \left(a_n b_n - c_n\right) \left(r_m s_m - t_m\right)
+ \lambda^3 \epsilon {\sum}_k^n \alpha_k c_{k-1}
+ \lambda^{-3} \delta {\sum}_k^m \rho_k t_{k-1}
\end{bmatrix},
\end{equation}
where $\lambda \in \R^\times$.
All discrete indefinite improper affine spheres are locally
constructed in this way.
\end{Theorem}
\begin{proof}
First a straightforward computation shows that Maurer-Cartan form of
$\hat{F}$ is computed as
\begin{align*}
\big(\hat{F}^m_n\big)^{-1} \hat{F}^m_{n+1}
&= \left(V^-_{n,m}\right)^{-1} \xi^+_n V^-_{n+1,m}\\
&=
\begin{bmatrix}
1 & 0 & 0\\
\frac{\beta_{n+1}}{\left(1 - b_n s_m\right)
\left(1 - b_{n+1} s_m\right)} \epsilon \lambda & 1 & 0\\
\frac{\alpha_{n+1} \beta_{n+1}}{1 - b_{n+1} s_m}
\left(\epsilon \lambda\right)^2 &
\alpha_{n+1} \left(1 - b_n s_m\right) \epsilon \lambda & 1
\end{bmatrix},\\
\big(\hat{F}^m_n\big)^{-1} \hat{F}^{m+1}_n
&= \left(V^+_{n,m}\right)^{-1} \xi^-_m V^+_{n,m+1}\\
&=
\begin{bmatrix}
\frac{1 - b_n s_m}{1 - b_n s_{m+1}} &
\sigma_{m+1} \delta \lambda^{-1} & 0\\
0 & \frac{1 - b_n s_{m+1}}{1 - b_n s_m} & 0\\
\frac{\rho_{m+1} \left(1 - b_n s_m\right)}%
{1 - b_n s_{m+1}} \delta \lambda^{-1} &
\rho_{m+1} \sigma_{m+1} \left(\delta \lambda^{-1}\right)^2 & 1
\end{bmatrix}.
\end{align*}
Next we take a diagonal gauge
$D^m_n = \diag \left(d^m_n, 1/d^m_n, 1 \right)$
so that the Maurer-Cartan form of $F^m_n = \hat{F}^m_n D^m_n$ is
\begin{align*}
\left(F^m_n\right)^{-1} F^m_{n+1}
&=
\begin{bmatrix}
\frac{d^m_{n+1}}{d^m_n} & 0 & 0\\
\frac{\beta_{n+1} d^m_{n+1} d^m_n}{\left(1 - b_n s_m\right)
\left(1 - b_{n+1} s_m\right)} \epsilon \lambda &
\frac{d^m_n}{d^m_{n+1}} & 0\\
\frac{\alpha_{n+1} \beta_{n+1} d^m_{n+1}}{1 - b_{n+1} s_m}
\left(\epsilon \lambda\right)^2 &
\frac{\alpha_{n+1} \left(1 - b_n s_m\right)}{d^m_{n+1}}
\epsilon \lambda & 1
\end{bmatrix},\\
\left(F^m_n\right)^{-1} F^{m+1}_n
&=
\begin{bmatrix}
\frac{d^{m+1}_n}{d^m_n}
\frac{1 - b_n s_m}{1 - b_n s_{m+1}} &
\frac{\sigma_{m+1}}{d^{m+1}_n d^m_n} \delta \lambda^{-1} & 0\\
0 & \frac{d^m_n}{d^{m+1}_n}
\frac{1 - b_n s_{m+1}}{1 - b_n s_m} & 0\\
\frac{\rho_{m+1} \left(1 - b_n s_m\right) d^{m+1}_n}%
{1 - b_n s_{m+1}} \delta \lambda^{-1} &
\frac{\rho_{m+1} \sigma_{m+1}}{d^{m+1}_n}
\left(\delta \lambda^{-1}\right)^2 & 1
\end{bmatrix}.
\end{align*}
This should be compared with \eqref{eq:DU}--\eqref{eq:DV} with $H=0$,
thus we have \eqref{explicit:d-omegaAB} and
\begin{equation*}
d^m_n = \left(1 - b_n s_m\right) \alpha_{n+1}.
\end{equation*}
To obtain the formula \eqref{formula:representation-d-improper},
we consider another diagonal gauge as introduced in
\eqref{def:d-extended}.
Namely,
setting $\tilde{F}^m_n = F^m_n \tilde{D}^m_n$ where
$\tilde{D}^m_n = \diag \left(\lambda,\, \lambda^{-1} \omega^m_n,\, 1\right)$,
we have
\begin{align}
\big(\tilde{F}^m_n\big)^{-1} \tilde{F}^m_{n+1}
&=\label{explicitformula-movingframe:U}
\begin{bmatrix}
\frac{\alpha_{n+2}}{\alpha_{n+1}}
\frac{1 - b_{n+1} s_m}{1- b_n s_m} & 0 & 0\\
\frac{\alpha_{n+2} \beta_{n+1} \epsilon}%
{\rho_{m+1} \left(1 - b_n s_m\right)} \lambda^3 & 1 & 0\\
\alpha_{n+2} \alpha_{n+1} \beta_{n+1} \epsilon^2 \lambda^3 &
\alpha_{n+1} \rho_{m+1} \left(1 - b_n s_m\right)
\epsilon & 1
\end{bmatrix}
= \tilde{U}^m_n,\\
\big(\tilde{F}^m_n\big)^{-1} \tilde{F}^{m+1}_n
&=\label{explicitformula-movingframe:V}
\begin{bmatrix}
1 & \frac{\sigma_{m+1} \rho_{m+2}}{\alpha_{n+1}
\left(1 - b_n s_m\right)}
\delta \lambda^{-3} & 0\\
0 & \frac{\rho_{m+2} \left(1 - b_n s_{m+1}\right)}%
{\rho_{m+1} \left(1 - b_n s_m\right)} & 0\\
\alpha_{n+1} \rho_{m+1}
\left(1 - b_n s_m\right) \delta &
\rho_{m+2} \rho_{m+1} \sigma_{m+1} \delta^2 \lambda^{-3} & 1
\end{bmatrix}
= \tilde{V}^m_n.
\end{align}
Of course this pair of matrices $\big(\tilde{U}^m_n,
\tilde{V}^m_n\big)$ accords with \eqref{dlax:tildeU} and
\eqref{dlax:tildeV}.
The frame $\tilde{F}$ can be computed explicitly as
\begin{equation*}
\tilde{F}^m_n
= G^-_m V^+_{n,m} D^m_n \tilde{D}^m_n
= \left[v^m_n, w^m_n, \xi_0\right],
\end{equation*}
where
\begin{gather*}
v^m_n =
\alpha_{n+1}
\begin{bmatrix}
\lambda\\
b_n \lambda^2\\
r_m - b_n \left(r_m s_m - t_m\right)
+ c_n \lambda^3
\end{bmatrix},\\
w^m_n =
\rho_{m+1}
\begin{bmatrix}
s_m \lambda^{-2}\\
\lambda^{-1}\\
a_n
- s_m \left(a_n b_n - c_n\right)
+ t_m \lambda^{-3}
\end{bmatrix}.
\end{gather*}
By definition of the moving frame,
we have
\begin{equation*}
v^m_n = \frac{f^m_{n+1} - f^m_n}{\epsilon},\quad
w^m_n = \frac{f^{m+1}_n - f^m_n}{\delta}.
\end{equation*}
Hence, noticing \eqref{eq:summertionbyparts}, $a_0 b_0 = 0$ and
the relations
\begin{align*}
a_n b_n - c_n
&= a_n b_n - \epsilon \textstyle{\sum}_k^n a_k \beta_k\\
&= a_n b_n - \textstyle{\sum}_k^n a_k \left(b_k - b_{k-1}\right)\\
&= \textstyle{\sum}_k^n \left(a_k - a_{k-1}\right) b_{k-1}\\
&= \epsilon \textstyle{\sum}_k^n \alpha_k b_{k-1}
\end{align*}
and $r_m s_m - t_m = \delta {\sum}_k^m \rho_k s_{k-1}$,
it follows for all integers $n$ and $m$ that
\begin{align*}
& f^m_n - f^0_0\\
&= \textstyle\sum_{k}^n \big(f^0_k - f^0_{k-1}\big)
+ \textstyle\sum_{l}^m \big(f^l_n - f^{l-1}_n\big)\\
&= \epsilon \textstyle\sum_{k}^n v^0_{k-1}
+ \delta \textstyle\sum_{l}^m w^{l-1}_n\\
&=
\epsilon \textstyle\sum_{k}^n
\alpha_k
\begin{bmatrix}
\lambda\\
b_{k-1} \lambda^2\\
c_{k-1} \lambda^3
\end{bmatrix}
+ \delta \textstyle\sum_{l}^m
\rho_l
\begin{bmatrix}
s_{l-1} \lambda^{-2}\\
\lambda^{-1}\\
a_n
- s_{l-1} \left(a_n b_n - c_n\right)
+ t_{l-1} \lambda^{-3}
\end{bmatrix}\\
&=
\begin{bmatrix}
\lambda a_n + \lambda^{-2} \left(r_m s_m - t_m\right)\\
\lambda^2 \left(a_n b_n - c_n\right) + \lambda^{-1} r_m\\
a_n r_m
- \left(a_n b_n - c_n\right) \left(r_m s_m - t_m\right)
+ \lambda^3 \epsilon {\sum}_k^n \alpha_k c_{k-1}
+ \lambda^{-3} \delta {\sum}_l^m \rho_l t_{l-1}
\end{bmatrix}.
\end{align*}
Thus
we have the representation formula \eqref{formula:representation-d-improper}
up to equiaffine transformations.
\end{proof}
\begin{Remark}
For given functions $A_n$, $B_m$ and
positive constants $\epsilon$, $\delta$,
it is known that a general solution to
the discrete Liouville equation \eqref{eq:dLiouville} can be expressed as
\begin{equation*}
\omega^m_n = \frac{\epsilon \left(p_0 +
{\sum}_k^n A_k \phi_k \phi_{k-1}\right)
+ \delta \left(q_0 + {\sum}_l^m B_l \psi_l \psi_{l-1}\right)}{\phi_n \psi_m},
\end{equation*}
where $\phi$ and $\psi$ are arbitrary functions with no zeros
in one variable,
and $p_0$ and $q_0$ are arbitrary constants.
See, for example, \cite{MR2749061}.
On the other hand,
our formula \eqref{explicit:d-omegaAB} tells that
a general solution also has an expression
\begin{equation*}
\omega^m_n
= \alpha_{n+1} \rho_{m+1} \left(1 - \epsilon \delta\,
{\sum}_k^n \frac{A_k}{\alpha_{k+1} \alpha_k}\,
{\sum}_l^m \frac{B_l}{\rho_{l+1} \rho_l}\right),
\end{equation*}
where $\alpha$ and $\rho$ are arbitrary functions with no zeros
in one variable.
As the most simple solutions,
by setting $\phi = \psi = \alpha = \rho = 1$ and $p_0 = q_0 = 0$,
these formulas give
an additive one $\omega^m_n = \epsilon {\sum}_k^n A_k
+ \delta {\sum}_l^m B_l$ and
a multiplicative one
$\omega^m_n = 1 - \epsilon \delta {\sum}_k^n A_k {\sum}_l^m B_l$,
respectively.
\end{Remark}
We conclude this paper with one of our main results,
which offers a representation formula using two discrete plane curves
for a discrete indefinite improper affine sphere.
\begin{Corollary}[Representation formula]
Fix positive numbers $\epsilon$ and $\delta$.
% For sequences $a_n$, $p_n$ which possibly depend on $\epsilon$,
% and $q_m$, $r_m$ which possibly depend on $\delta$,
For maps $a, p\colon \epsilon \Z \to \R$ and
$q, r\colon \delta \Z \to \R$,
set
\begin{equation*}
\gamma^1_n =
\begin{bmatrix}
a_n\\
p_n
\end{bmatrix},\quad
\gamma^2_m =
\begin{bmatrix}
q_m\\
r_m
\end{bmatrix}.
\end{equation*}
% \footnote{\red{We assume that $a_{n+1} \neq a_n$ and
% $r_{m+1} \neq r_m$ for all $n$, $m$.}
% This assumption is unnecessary.}
Then the map
\begin{gather}
f^m_n =\label{formula:d-improper-AF-from-d-planecurves}
\begin{bmatrix}
\gamma^1_n + \gamma^2_m\\
z^m_n
\end{bmatrix},
\end{gather}
where
\begin{gather}
z^m_n =\label{formula:d-improper-AF-from-d-planecurves:z}
\det \left[\gamma^1_n, \gamma^2_m\right]
+ {\sum}_k^n \det \left[\gamma^1_{k-1}, \gamma^1_k\right]
- {\sum}_k^m \det \left[\gamma^2_{k-1}, \gamma^2_k\right]
\end{gather}
is a discrete indefinite improper affine sphere
with the affine normal ${}^\mathrm{t} \left[0, 0, 1\right]$.
Its data solving the discrete Liouville equation \eqref{eq:dLiouville} is
\begin{align*}
\omega^m_n &= \det
\left[\frac{\gamma^1_{n+1} - \gamma^1_n}{\epsilon},
\frac{\gamma^2_{m+1} - \gamma^2_m}{\delta}\right],\\
A_n &= \det
\left[\frac{\gamma^1_{n+1} - \gamma^1_n}{\epsilon},
- \frac{\gamma^1_n - \gamma^1_{n - 1}}{\epsilon^2}\right],\\
B_m &= \det
\left[- \frac{\gamma^2_m - \gamma^2_{m - 1}}{\delta^2},
\frac{\gamma^2_{m+1} - \gamma^2_m}{\delta}\right].
\end{align*}
Moreover the associated family of $f$ is given by the transformation
\begin{equation*}
\gamma^1 \mapsto
\begin{bmatrix}
\lambda & 0\\
0 & \lambda^2
\end{bmatrix} \gamma^1,\quad
\gamma^2 \mapsto
\begin{bmatrix}
\lambda^{-2} & 0\\
0 & \lambda^{-1}
\end{bmatrix} \gamma^2
\end{equation*}
where $\lambda \in \R^\times$.
Conversely all discrete
indefinite improper affine spheres can be
constructed in this way.
\end{Corollary}
\begin{proof}
First, introducing functions $p_n=a_nb_n-c_n$ and $q_m= r_ms_m -t_m$,
we rephrase \eqref{formula:representation-d-improper} as
\begin{equation}\label{formula:representation-d-improper-2}
f^m_n =
\begin{bmatrix}
\lambda a_n + \lambda^{-2} q_m\\
\lambda^2 p_n + \lambda^{-1} r_m\\
 a_n r_m - p_n q_m + \lambda^3 {\sum}_k^n
\left(a_{k-1} p_k - a_k p_{k-1}\right)
- \lambda^{-3} {\sum}_k^m
\left(r_k q_{k-1}- r_{k-1} q_k\right)
\end{bmatrix},
\end{equation}
where we use the identities
\begin{align*}
\epsilon \alpha_k c_{k-1}
&= \left(a_k-a_{k-1}\right) c_{k-1} \\
&= a_k c_{k-1} + a_{k-1} \left(\epsilon a_k \beta_k - c_k\right)\\
&= a_k c_{k-1} + a_k a_{k-1}\left(b_k - b_{k-1}\right) - a_{k-1} c_k\\
&= a_{k-1} p_k - a_k p_{k-1},
\end{align*}
and $\delta \rho_k t_{k-1} = r_{k-1} q_{k} - r_k q_{k-1}$.
Note that $a_0= p_0= q_0 = r_0 =0$.
We then consider an equiaffine transformation of $f_n^m$ as
\begin{equation*}
\tilde f_n^m=
\begin{bmatrix}
1 & 0 & 0 \\
0 & 1 & 0 \\
\lambda^{-1} {\tilde r}_0 - \lambda^2 {\tilde p}_0 &
\lambda {\tilde a}_0 - \lambda^{-2} {\tilde q}_0 &
1
\end{bmatrix}
f_n^m +
\begin{bmatrix}
\lambda {\tilde a}_0 + \lambda^{-2} {\tilde q}_0 \\
\lambda^2 {\tilde p}_0 + \lambda^{-1} {\tilde r}_0 \\
{\tilde a}_0 {\tilde r}_0 - {\tilde p}_0 {\tilde q}_0
\end{bmatrix},
\end{equation*}
where ${\tilde a}_0$,
${\tilde r}_0$,
${\tilde p}_0$ and ${\tilde q}_0$ are some constants.
A straightforward computation shows that
\begin{equation*}
\tilde f^m_n =
\begin{bmatrix}
\lambda \tilde a_n + \lambda^{-2} \tilde q_m\\
\lambda^2 \tilde p_n + \lambda^{-1} \tilde r_m\\
 \tilde a_n \tilde r_m - \tilde p_n \tilde q_m + \lambda^3 {\sum}_k^n
\left(\tilde a_{k-1} \tilde p_k - \tilde a_k \tilde p_{k-1}\right)
- \lambda^{-3} {\sum}_k^m
\left(\tilde r_k \tilde q_{k-1} - \tilde r_{k-1} \tilde q_k\right)
\end{bmatrix},
\end{equation*}
where ${\tilde a}_n = a_n+ {\tilde a}_0$,
${\tilde p}_n = p_n + {\tilde p}_0$,
${\tilde q}_n = q_n + {\tilde q}_0$
and ${\tilde r} = r + {\tilde r}_0$.
Thus we obtain \eqref{formula:d-improper-AF-from-d-planecurves}
on writing
\begin{equation*}
\gamma^1_n =
\begin{bmatrix}
{\tilde a}_n \\
{\tilde p}_n
\end{bmatrix},\quad
\gamma^2_m =
\begin{bmatrix}
{\tilde q}_m\\
{\tilde r}_m
\end{bmatrix}.
\end{equation*}
Since $\gamma^1_n$ and $\gamma^2_m$ are arbitrary,
\eqref{formula:d-improper-AF-from-d-planecurves}
gives the all improper indefinite affine spheres.
\end{proof}
\begin{Remark}
We recall that the height function $z$
defined by \eqref{formula:improper-AF-from-planecurves:z}
satisfies $\pd{u} \pd{v} z \left(u, v\right)
= \det \left[\gamma_1' \left(u\right),
\gamma_2' \left(v\right)\right]$.
As a discrete analogue of this,
the sequence $z$ defined by
\eqref{formula:d-improper-AF-from-d-planecurves:z}
satisfies a difference equation
\begin{equation}\label{z:recurrent}
z^{m+1}_{n+1} - z^m_{n+1} - z^{m+1}_n + z^m_n
= \det \left[\gamma^1_{n+1} - \gamma^1_n,
\gamma^2_{m+1} - \gamma^2_m\right].
\end{equation}
In particular,
if $\epsilon = \delta$ and $\gamma^1_i = \gamma^2_i$ for all $i \in \Z$,
then $z$ satisfies $z^i_i = z^{i \pm 1}_i = 0$ for all $i \in \Z$,
and then every $z^{i \pm 2}_i$ can be fast computed
by using the recurrent relation \eqref{z:recurrent}.
Iterating this,
we can obtain all the values $z^m_n$ numerically.
Refer to Example \ref{example:d-circle,d-square,d-cpt}
for specific examples,
where we explicitly calculate $z$
as a function in $\left(n, m\right)$.
\end{Remark}
We fix positive numbers $\epsilon$ and $\delta$ arbitrarily
to illustrate examples of
discrete indefinite improper affine spheres
by taking several discrete curves.
\begin{Example}
Let $P_n$ and $R_m$ be arbitrary sequences which possibly depend on
$\epsilon$ and $\delta$ respectively.
We denote by $\Delta P$ and $\Delta R$ the forward differences of them,
that is,
\begin{equation*}
\Delta P_n = \frac{P_{n+1}-P_n}{\epsilon},\quad
\Delta R_m = \frac{R_{m+1}-R_m}{\delta}.
\end{equation*}
We substitute discrete curves
\begin{equation*}
\gamma^1_n =
\begin{bmatrix}
\epsilon n\\
\Delta P_n
\end{bmatrix},\quad
\gamma^2_m =
\begin{bmatrix}
\Delta R_m\\
\delta m
\end{bmatrix}
\end{equation*}
into the representation formula
\eqref{formula:d-improper-AF-from-d-planecurves},
and have a discrete indefinite improper affine sphere
\begin{align*}
f^m_n &=
\begin{bmatrix}
\epsilon n + \Delta R_m\\
\delta m + \Delta P_n\\
\left(\epsilon n + \Delta R_m\right)
\left(\delta m + \Delta P_n\right)
- 2 \left(\Delta P_n\right) \left(\Delta R_m\right)
- P_{n+1} - P_n
- R_{m+1} - R_m
\end{bmatrix}.
\end{align*}
Its data is
\begin{gather*}
\omega^m_n =
1 - \epsilon^{-2} \delta^{-2}
\left(P_{n+2} - 2 P_{n+1} + P_n\right)
\left(R_{m+2} - 2 R_{m+1} + R_m\right),\\
A_n = \epsilon^{-3}
\left(P_{n+2} - 3 P_{n+1} + 3 P_n - P_{n-1}\right),\\
B_m = \delta^{-3}
\left(R_{m+2} - 3 R_{m+1} + 3 R_m - R_{m-1}\right).
\end{gather*}
\end{Example}
\begin{Example}\label{example:d-circle,d-square,d-cpt}
We illustrate discrete counterparts to those surfaces
given in Example \ref{example:circle,square,cpt}.
Fix positive numbers $q_1$, $q_2$ arbitrarily and
introduce positive numbers
\begin{equation*}
\theta_1 = \frac{2}{\epsilon}
\arctan \left(\frac{\epsilon}{2} q_1\right),\quad
\theta_2 = \frac{2}{\delta}
\arctan \left(\frac{\delta}{2} q_2\right).
\end{equation*}
\begin{enumerate}
\item
First example is given by the discrete curves
\begin{gather*}
\gamma^1_n =
\begin{bmatrix}
\cos \left(\theta_1 \epsilon n\right)\\
\sin \left(\theta_1 \epsilon n\right)
\end{bmatrix},\quad
\gamma^2_m =
\begin{bmatrix}
\cos \left(\theta_2 \delta m\right)\\
\sin \left(\theta_2 \delta m\right)
\end{bmatrix}.
\end{gather*}
Substituting these into the representation formula,
we have
\begin{gather*}
f^m_n =
\begin{bmatrix}
\cos \left(\theta_1 \epsilon n\right)
+ \cos \left(\theta_2 \delta m\right)\\
\sin \left(\theta_1 \epsilon n\right)
+ \sin \left(\theta_2 \delta m\right)\\
- \sin \left(\theta_1 \epsilon n - \theta_2 \delta m\right)
+ n \sin \left(\theta_1 \epsilon\right)
- m \sin \left(\theta_2 \delta\right)
\end{bmatrix}.
\end{gather*}
The data solving the discrete Liouville equation
\eqref{eq:dLiouville} are given by constants
\begin{gather*}
A_n = \frac{8}{\epsilon^3}
\sin^3 \left(\frac{\epsilon}{2} \theta_1\right)
\cos \left(\frac{\epsilon}{2} \theta_1\right)
=
\frac{16 q_1^3}{\left(4 + \epsilon^2 q_1^2\right)^2},\\
B_m = - \frac{8}{\delta^3}
\sin^3 \left(\frac{\delta}{2} \theta_2\right)
\cos \left(\frac{\delta}{2} \theta_2\right)
=
- \frac{16 q_2^3}{\left(4 + \delta^2 q_2^2\right)^2},
\end{gather*}
and a sequence
\begin{equation*}
\begin{split}
\omega^m_n =\;&
\frac{4 q_1 q_2}
{\sqrt{4 + \epsilon^2 q_1^2} \sqrt{4 + \delta^2 q_2^2}}
\sin \left(\frac{\epsilon}{2} \theta_1 \left(2n+1\right)
- \frac{\delta}{2} \theta_2 \left(2m+1\right)\right).
\end{split}
\end{equation*}
Therefore $f$ is singular if $n$, $m$ satisfy
\begin{equation*}
\epsilon \theta_1 \left(2n+1\right)
\equiv \delta \theta_2 \left(2m+1\right)
\pmod{2 \pi}.
\end{equation*}
\begin{figure}[H]
\hfill
\hfill
\subfigure{\includegraphics[width=3cm,keepaspectratio]%
{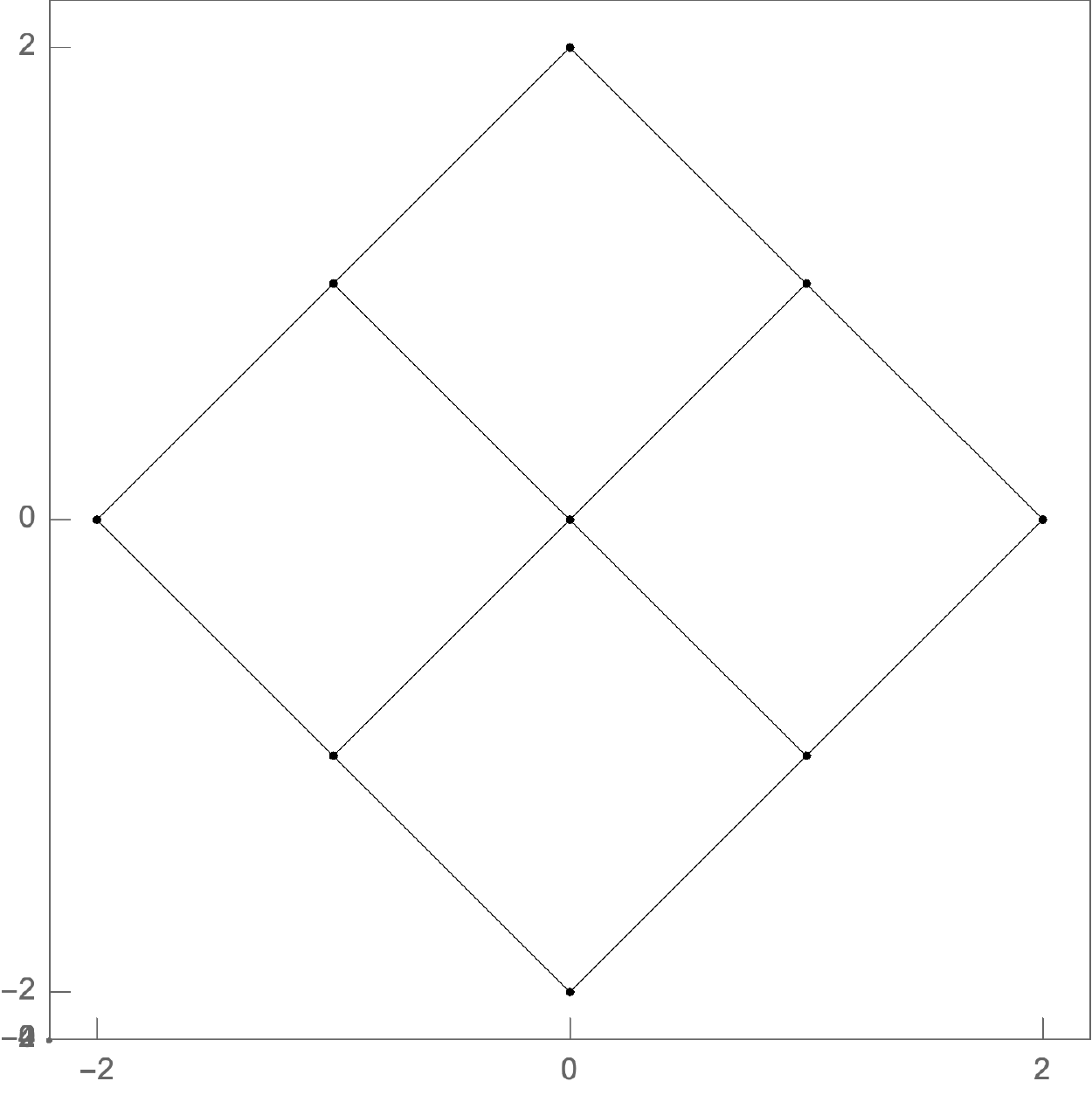}}
\hfill
\subfigure{\includegraphics[width=3cm,keepaspectratio]%
{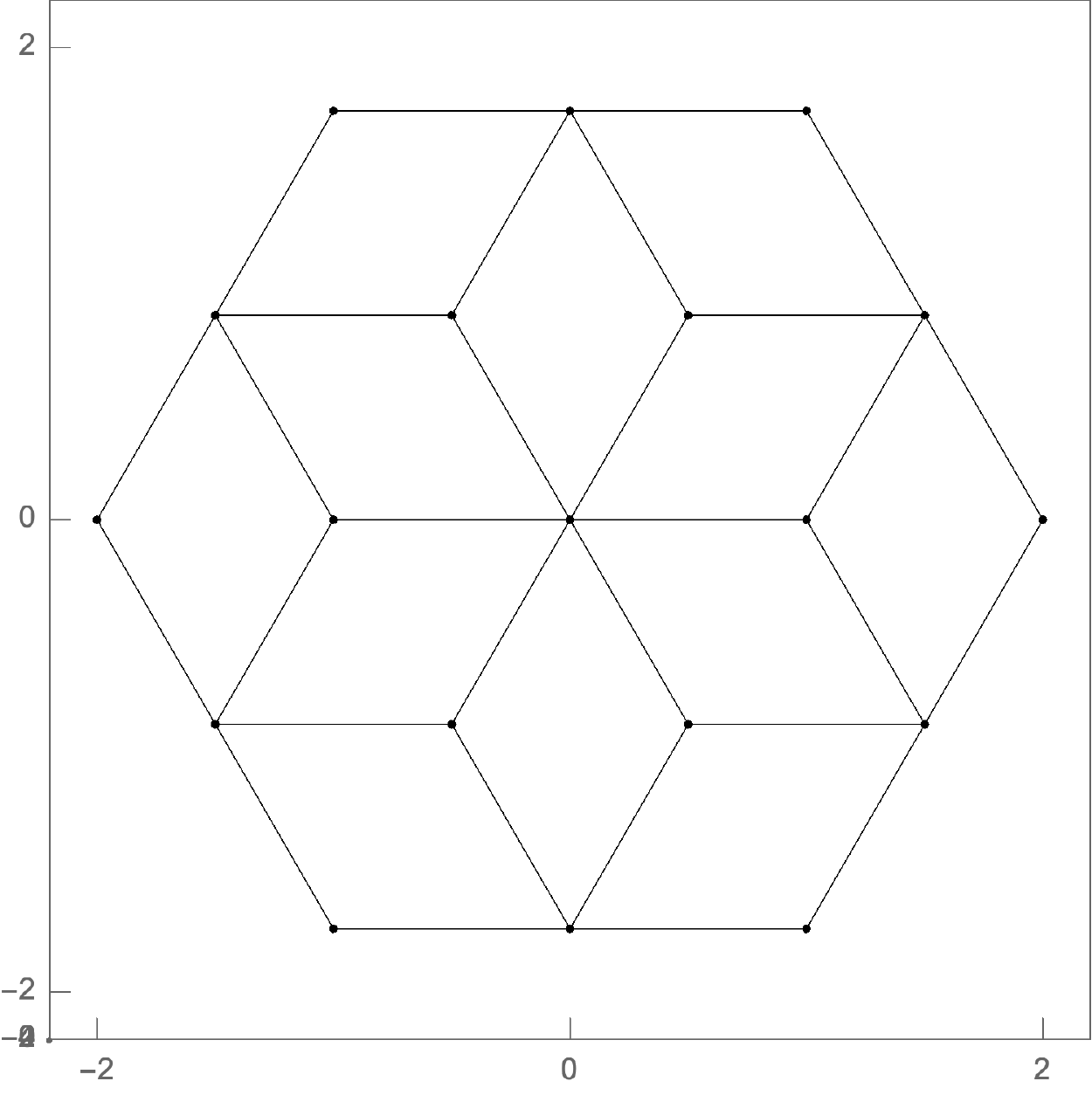}}
\hfill
\subfigure{\includegraphics[width=3cm,keepaspectratio]%
{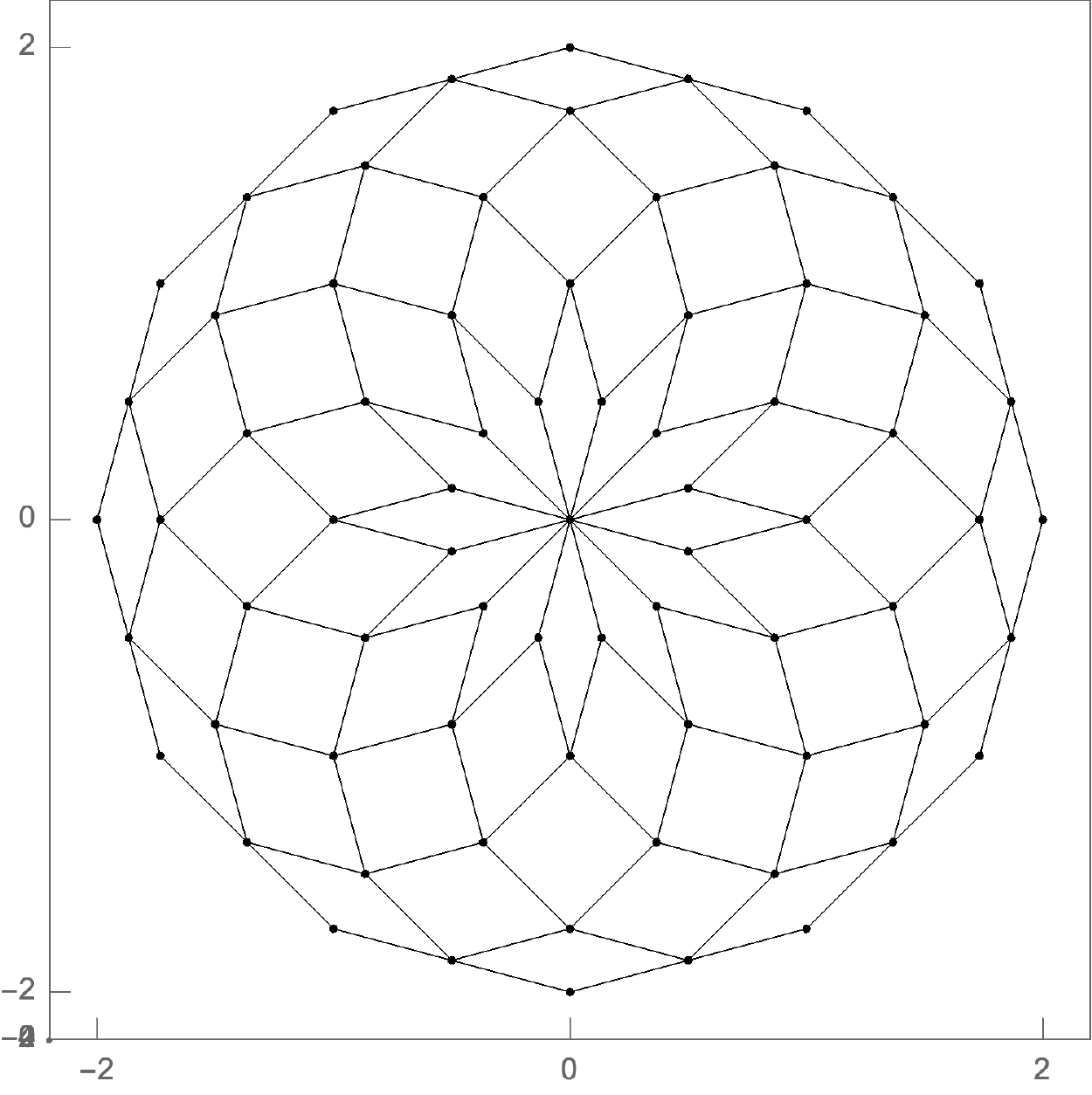}}
\hfill
\hfill
%%%% Don't delete the next line

%%%% Don't delete the previous line
\hfill
\hfill
\subfigure[$q = 2\tan \left(\pi/4\right)$]%
{\includegraphics[height=6cm,keepaspectratio]%
{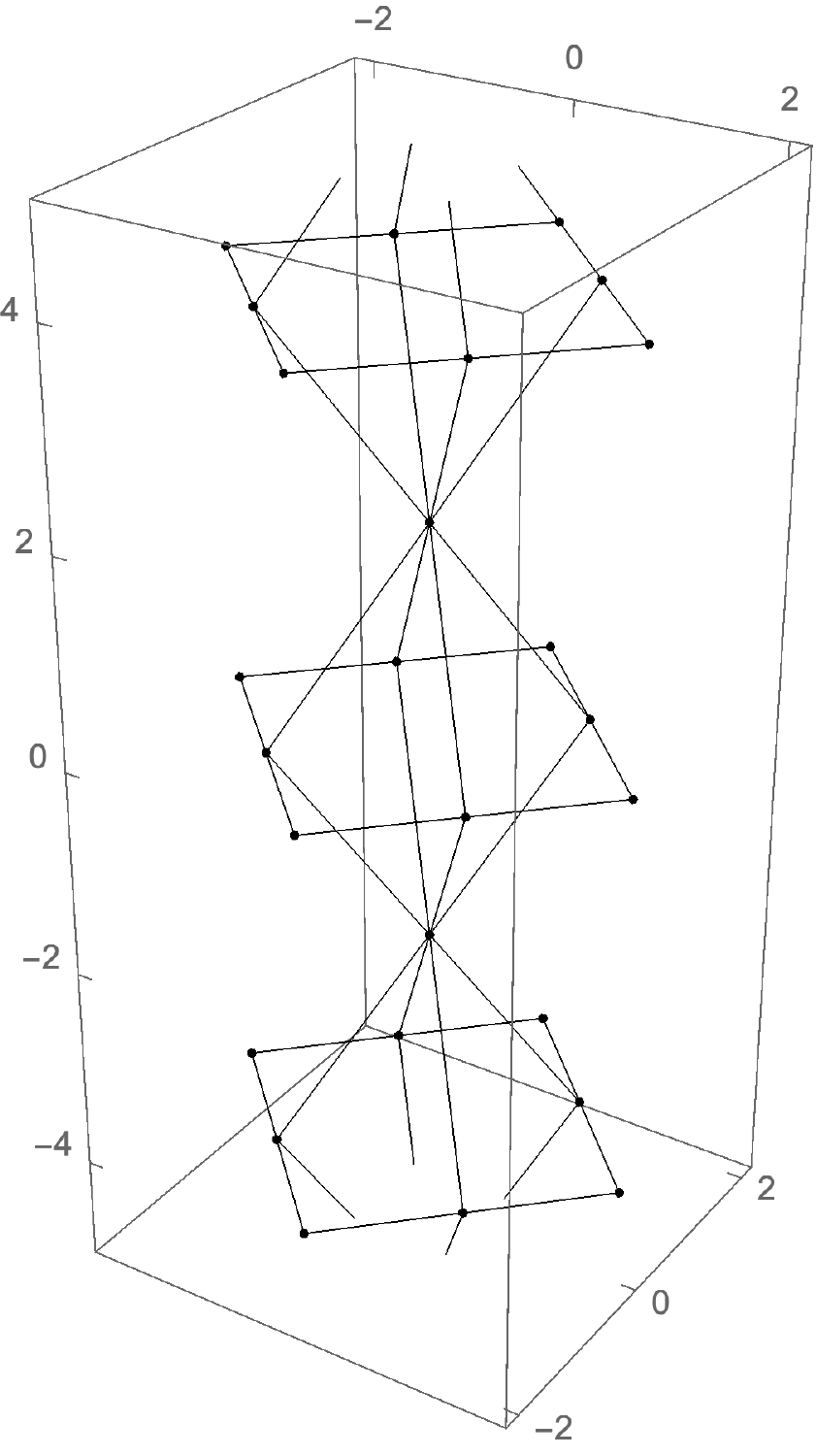}}
\hfill
\subfigure[$q = 2 \tan \left(\pi/6\right)$]%
{\includegraphics[height=6cm,keepaspectratio]%
{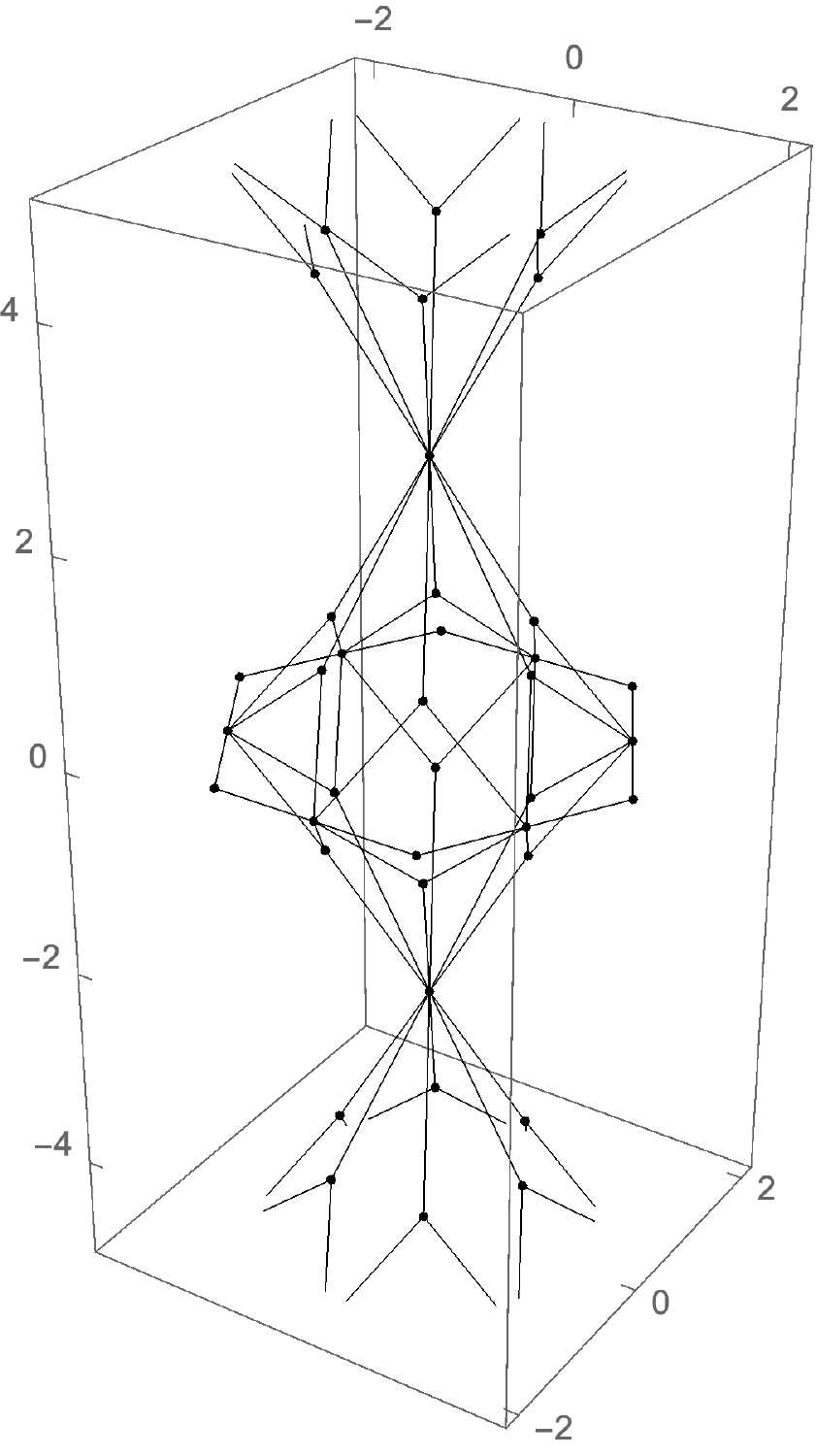}}
\hfill
\subfigure[$q = 2 \tan \left(\pi/12\right)$]%
{\includegraphics[height=6cm,keepaspectratio]%
{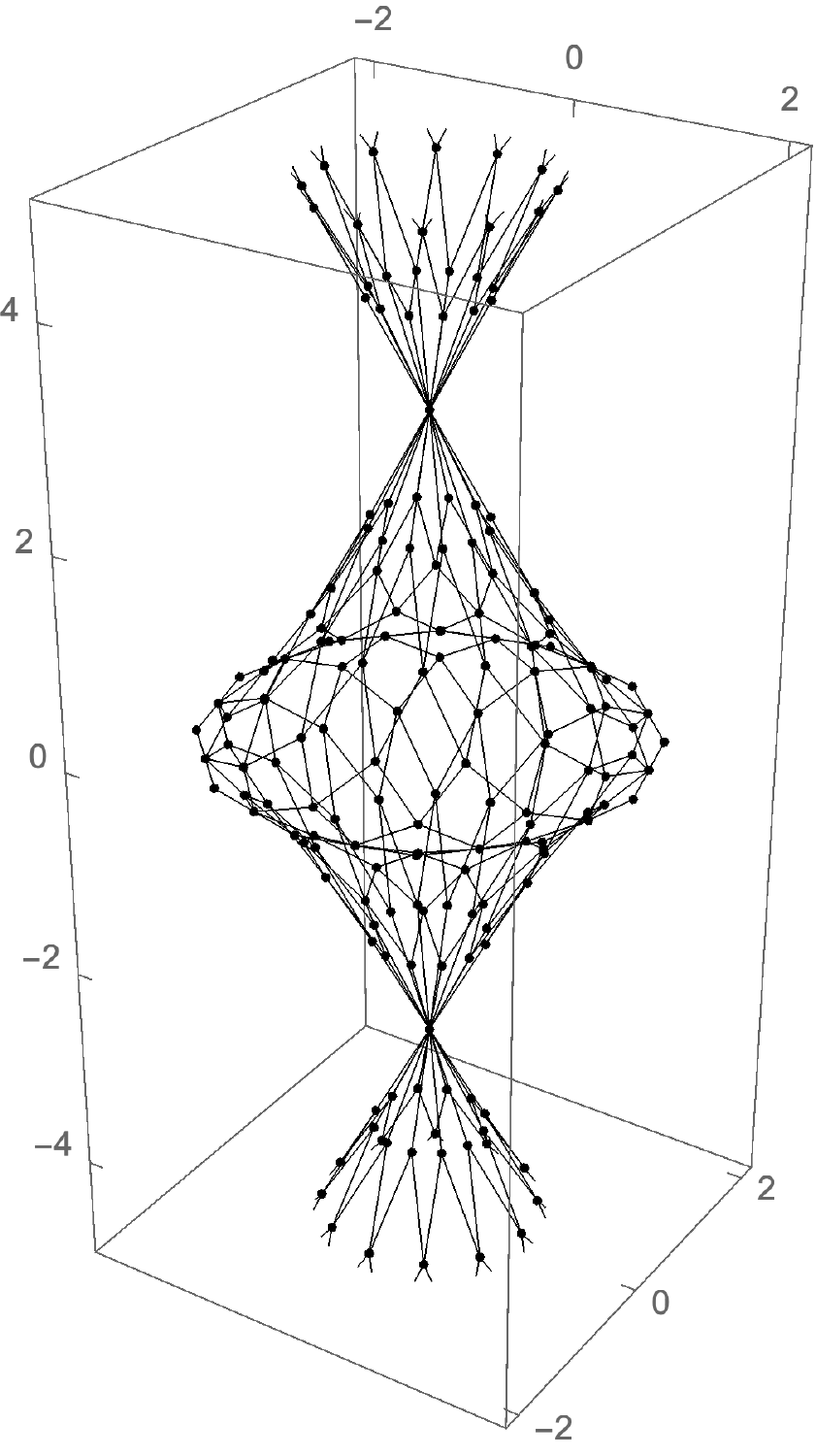}}
\hfill
\hfill
\caption{Discrete indefinite improper affine spheres $f^m_n$
with $\epsilon = \delta = 1$ and $q = q_1 = q_2$,
which exhibit cone points.
The figures in upper line show views from the top.}
\label{fig:d-exam2-1}
\end{figure}

\item
Second example is given by
\begin{gather*}
\gamma^1_n =
\begin{bmatrix}
\left|\cos \left(\theta_1 \epsilon n\right)\right|
\cos \left(\theta_1 \epsilon n\right)\\
\left|\sin \left(\theta_1 \epsilon n\right)\right|
\sin \left(\theta_1 \epsilon n\right)
\end{bmatrix},\quad
\gamma^2_m =
\begin{bmatrix}
\left|\cos \left(\theta_2 \delta m\right)\right|
\cos \left(\theta_2 \delta m\right)\\
\left|\sin \left(\theta_2 \delta m\right)\right|
\sin \left(\theta_2 \delta m\right)
\end{bmatrix}.
\end{gather*}
We assume that both the images of $\gamma^1$ and $\gamma^2$
contain four points
\begin{equation*}
\pm
\begin{bmatrix}
1\\
0
\end{bmatrix},\quad
\pm
\begin{bmatrix}
0\\
1
\end{bmatrix},
\end{equation*}
which we can always achieve by choosing $q_1$, $q_2$ appropriately.
By virtue of this assumption,
we are able to assume that $0 < \theta_1 \epsilon \leq \pi/2$ and
$\cos \left(\theta_1 \epsilon \left(n-1\right) \right)
\cos \left(\theta_1 \epsilon n\right) \geq 0$ and
$\sin \left(\theta_1 \epsilon \left(n-1\right) \right)
\sin \left(\theta_1 \epsilon n\right) \geq 0$.
% \begin{align*}
% \cos \left(\theta_1 \epsilon \left(n-1\right) \right)
% \cos \left(\theta_1 \epsilon n\right) \geq 0,\\
% \cos \left(\theta_2 \delta \left(m-1\right) \right)
% \cos \left(\theta_2 \delta m\right) \geq 0,\\
% \sin \left(\theta_1 \epsilon \left(n-1\right) \right)
% \sin \left(\theta_1 \epsilon n\right) \geq 0,\\
% \sin \left(\theta_2 \delta \left(m-1\right) \right)
% \sin \left(\theta_2 \delta m\right) \geq 0.
% \end{align*}
Therefore the differences of $\gamma^1$ can be written into
simple forms as
\begin{gather*}
\gamma^1_n - \gamma^1_{n-1} =
2 \sin \left(\theta_1 \epsilon\right)
\begin{bmatrix}
- \left|\cos \frac{\theta_1 \epsilon \left(2n-1\right)}{2}\right|
\sin \frac{\theta_1 \epsilon \left(2n-1\right)}{2}\\
\left|\sin \frac{\theta_1 \epsilon \left(2n-1\right)}{2}\right|
\cos \frac{\theta_1 \epsilon \left(2n-1\right)}{2}
\end{bmatrix}.
% \gamma^1_{n+1} - 2 \gamma^1_n + \gamma^1_{n-1} =
% 2 \sin \left(\theta_1 \epsilon\right)
% \begin{bmatrix}
% - \left|\cos \frac{\theta_1 \epsilon \left(2n+1\right)}{2}\right|
% \sin \frac{\theta_1 \epsilon \left(2n+1\right)}{2}
% + \left|\cos \frac{\theta_1 \epsilon \left(2n-1\right)}{2}\right|
% \sin \frac{\theta_1 \epsilon \left(2n-1\right)}{2}\\
% \left|\sin \frac{\theta_1 \epsilon \left(2n+1\right)}{2}\right|
% \cos \frac{\theta_1 \epsilon \left(2n+1\right)}{2}
% - \left|\sin \frac{\theta_1 \epsilon \left(2n-1\right)}{2}\right|
% \cos \frac{\theta_1 \epsilon \left(2n-1\right)}{2}
% \end{bmatrix}
\end{gather*}
Hence we have that
\begin{align*}
\det \left[\gamma^1_{n-1}, \gamma^1_n\right]
= \det \left[\gamma^1_n, \gamma^1_n - \gamma^1_{n-1}\right]
= \sin \left(\theta_1 \epsilon\right)
\left|\sin \left(\theta_1 \epsilon \left(2n-1\right)\right)\right|.
\end{align*}
Further, on choosing the parameter $q_1$ as
\begin{equation*}
q_1 = \frac{2}{\epsilon} \tan \frac{\pi}{4 N_1}
% q_2 = \frac{2}{\delta} \tan \frac{\pi}{4 N_2}
\end{equation*}
with a positive integer $N_1$,
which is equivalent to setting
$\theta_1 \epsilon = \pi/\left(2 N_1\right)$,
it holds for all $n \in \Z$ that
\begin{align*}
{\sum}_k^n
\det \left(\gamma^1_{k-1}, \gamma^1_k\right)
&=
{\sum}_k^n
\sin \frac{\pi}{2 N_1}
\left|\sin \frac{\left(2k-1\right) \pi}{2 N_1}\right|\\
&=
\left\lfloor \frac{n}{N_1}\right\rfloor
+ \frac{1}{2}
\left(1 - \left(-1\right)^{\left\lfloor n/{N_1}\right\rfloor}
\cos \frac{n \pi}{N_1}\right).
\end{align*}
Here $\left\lfloor u\right\rfloor$ is the floor of $u$,
that is,
the greatest integer less than or equal to $u$.
We apply the same discussion as above to $\gamma^2$,
and set $q_2 = \left(2/\delta\right)
\tan \left(\pi/\left(4 N_2\right)\right)$ to have
\begin{gather*}
f^m_n =
\begin{bmatrix}
\left|\cos \frac{n \pi}{2 N_1}\right|
\cos \frac{n \pi}{2 N_1}
+ \left|\cos \frac{m \pi}{2 N_2}\right|
\cos \frac{m \pi}{2 N_2}\\
\left|\sin \frac{n \pi}{2 N_1}\right|
\sin \frac{n \pi}{2 N_1}
+ \left|\sin \frac{m \pi}{2 N_2}\right|
\sin \frac{m \pi}{2 N_2}\\
z^m_n
\end{bmatrix},\\
\begin{split}
z^m_n =\;&
\left|\cos \textstyle\frac{n \pi}{2 N_1}
\sin \textstyle\frac{m \pi}{2 N_2}\right|
\cos \textstyle\frac{n \pi}{2 N_1}
\sin \textstyle\frac{m \pi}{2 N_2}
- \left|\sin \textstyle\frac{n \pi}{2 N_1}
\cos \textstyle\frac{m \pi}{2 N_2}\right|
\sin \textstyle\frac{n \pi}{2 N_1}
\cos \textstyle\frac{m \pi}{2 N_2}\\
&+
\left\lfloor \textstyle\frac{n}{N_1}\right\rfloor
- \left\lfloor \textstyle\frac{m}{N_2}\right\rfloor
- \textstyle\frac{1}{2}
\left(\left(-1\right)^{\left\lfloor n/{N_1}\right\rfloor}
\cos \textstyle\frac{n \pi}{N_1}
- \left(-1\right)^{\left\lfloor m/{N_2}\right\rfloor}
\cos \textstyle\frac{m \pi}{N_2}\right).
\end{split}
\end{gather*}
Its data is
\begin{gather*}
A_n =
\begin{cases}
\frac{2}{\epsilon^3}
\sin^4 \frac{\pi}{2 N_1} & n \in N_1 \Z\\
0 & n \not\in N_1 \Z,
\end{cases}\\
B_m =
\begin{cases}
- \frac{2}{\delta^3}
\sin^4 \frac{\pi}{2 N_2} & m \in N_2 \Z\\
0 & m \not\in N_2 \Z,
\end{cases}\\
\begin{split}
\omega^m_n =\;&
\textstyle\frac{4}{\epsilon \delta}
\sin \textstyle\frac{\pi}{2 N_1}
\sin \textstyle\frac{\pi}{2 N_2}
\left(\left|\cos \textstyle\frac{\left(2n+1\right) \pi}{4 N_1}
\sin \textstyle\frac{\left(2m+1\right) \pi}{4 N_2}\right|
\sin \textstyle\frac{\left(2n+1\right) \pi}{4 N_1}
\cos \textstyle\frac{\left(2m+1\right) \pi}{4 N_2}\right.\\
&\left. -
\cos \textstyle\frac{\left(2n+1\right) \pi}{4 N_1}
\sin \textstyle\frac{\left(2m+1\right) \pi}{4 N_2}
\left|\sin \textstyle\frac{\left(2n+1\right) \pi}{4 N_1}
\cos \textstyle\frac{\left(2m+1\right) \pi}{4 N_2}\right|\right)\\
=\;&
\begin{cases}
0 & \left(n, m\right) \in S\\
\textstyle\frac{2}{\epsilon \delta}
\sin \textstyle\frac{\pi}{2 N_1}
\sin \textstyle\frac{\pi}{2 N_2}
\sin \textstyle\frac{\left(2n+1\right) \pi}{2 N_1}
\sin \textstyle\frac{\left(2m+1\right) \pi}{2 N_2} &
\left(n, m\right) \not\in S.
\end{cases}
\end{split}
\end{gather*}
The singular set $S$ consists of integer points
in a checkerboard,
that is
\begin{equation*}
S = \left\{\left(n, m\right) \in \Z^2 \;\big|\;
\left\lfloor n/N_1\right\rfloor \equiv
\left\lfloor m/N_2\right\rfloor\;
\left(\mathrm{mod}\ 2\right)\right\}.
\end{equation*}
\begin{figure}[H]
\hfill
\hfill
\subfigure{\includegraphics[width=3cm,keepaspectratio]%
{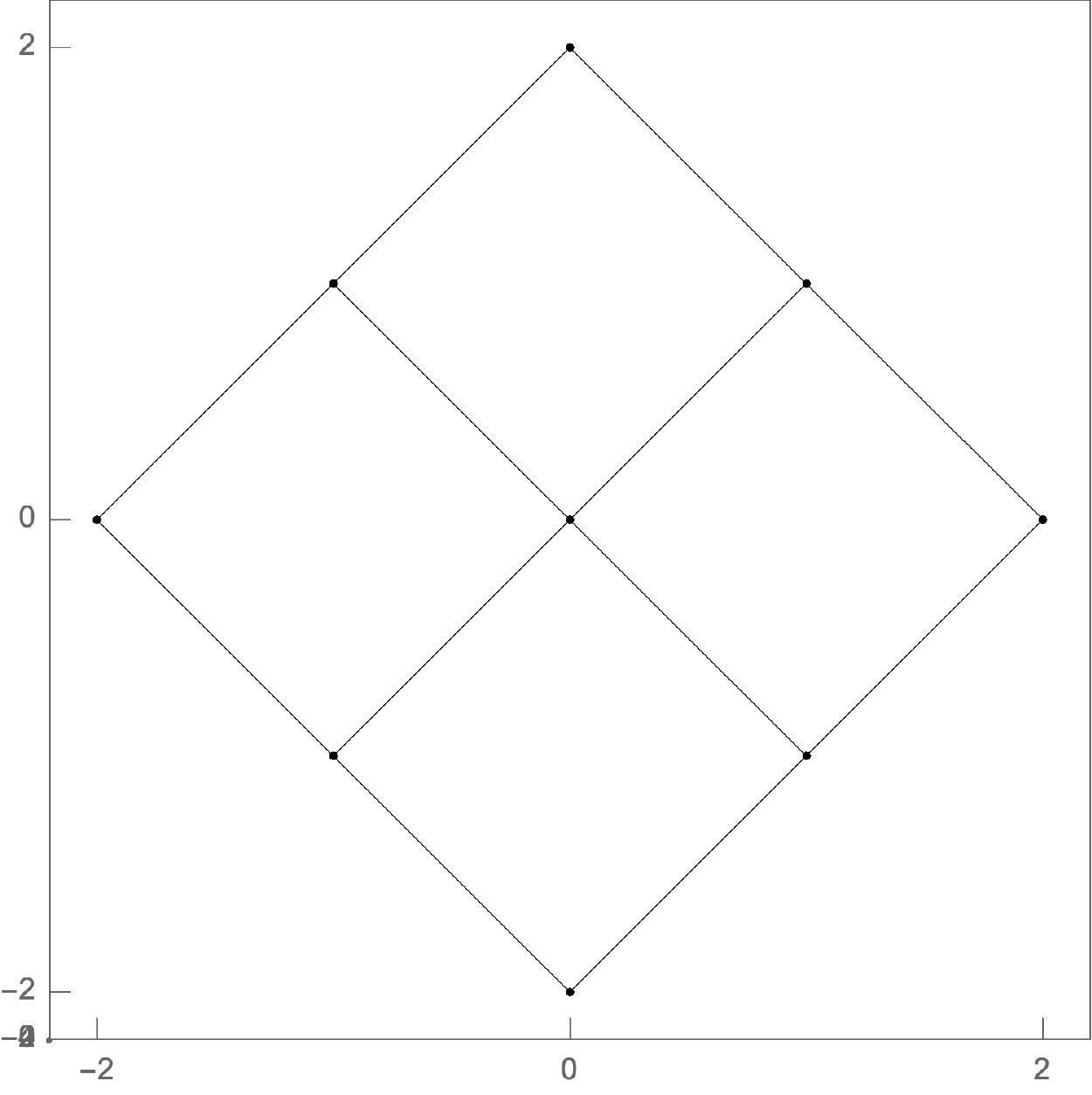}}
\hfill
\subfigure{\includegraphics[width=3cm,keepaspectratio]%
{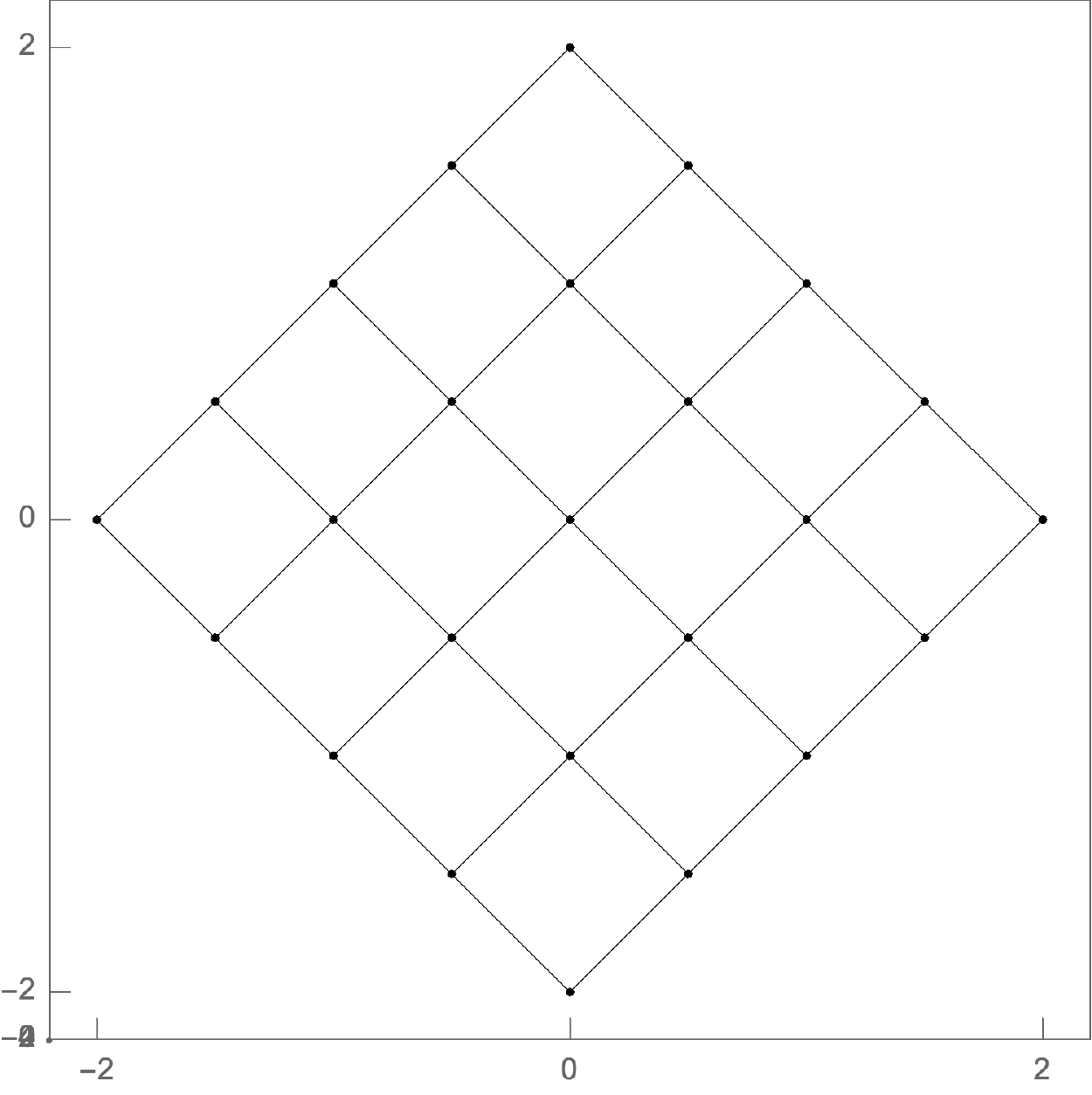}}
\hfill
\subfigure{\includegraphics[width=3cm,keepaspectratio]%
{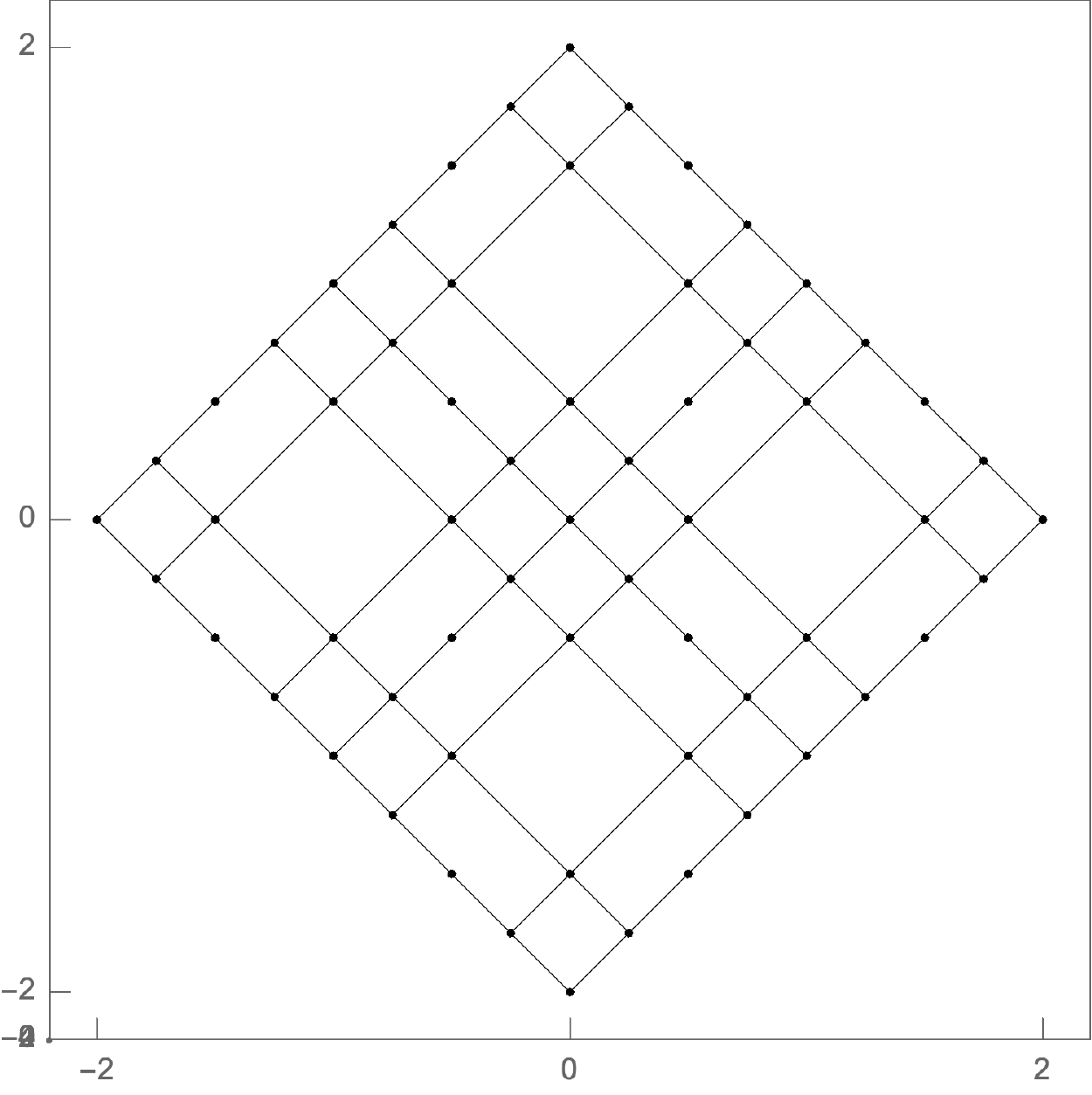}}
\hfill
\hfill
%%%% Don't delete the next line

%%%% Don't delete the previous line
\hfill
\hfill
\subfigure[$N=1$]{\includegraphics[height=6cm,keepaspectratio]%
{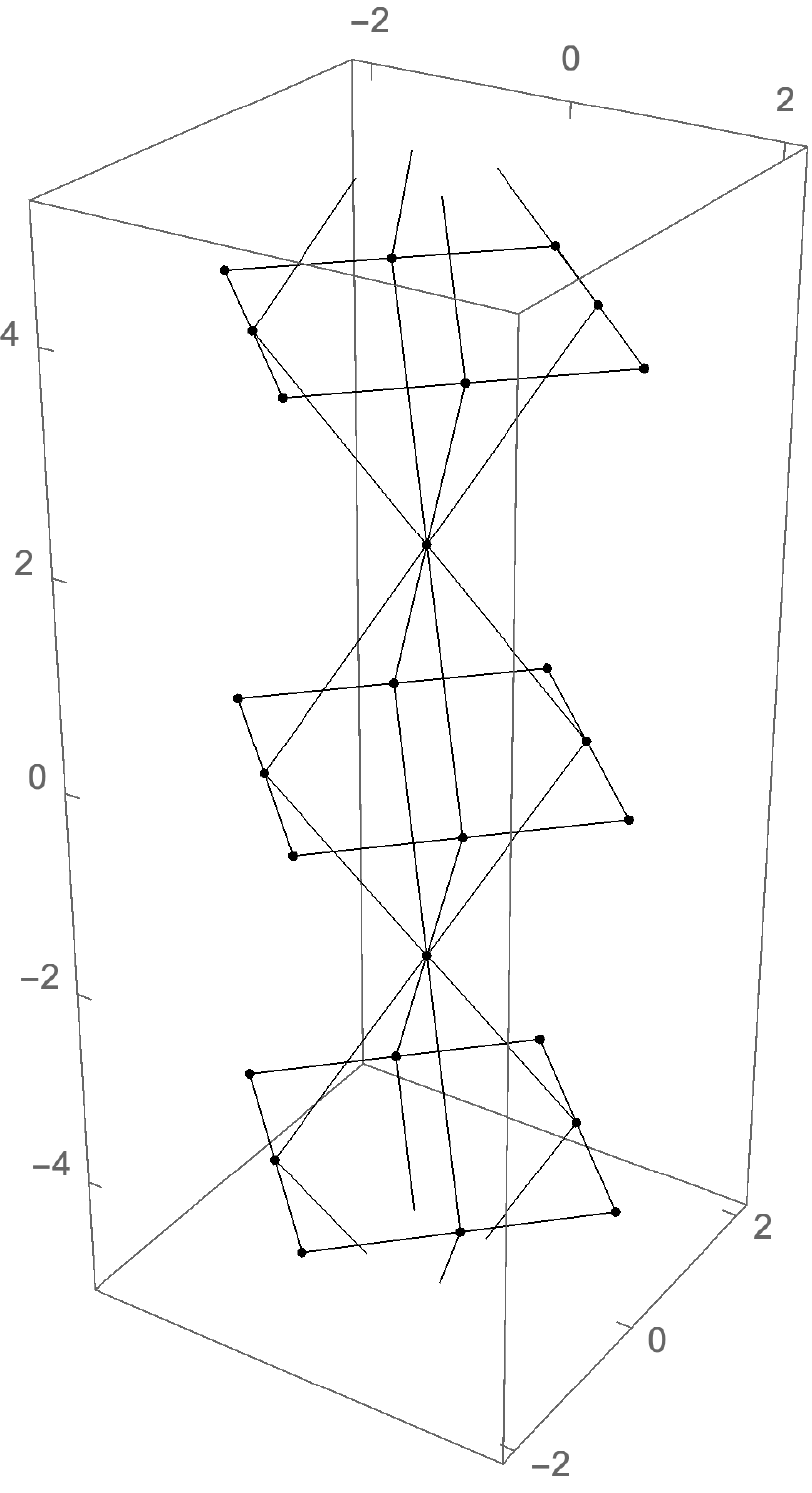}}
\hfill
\subfigure[$N=2$]{\includegraphics[height=6cm,keepaspectratio]%
{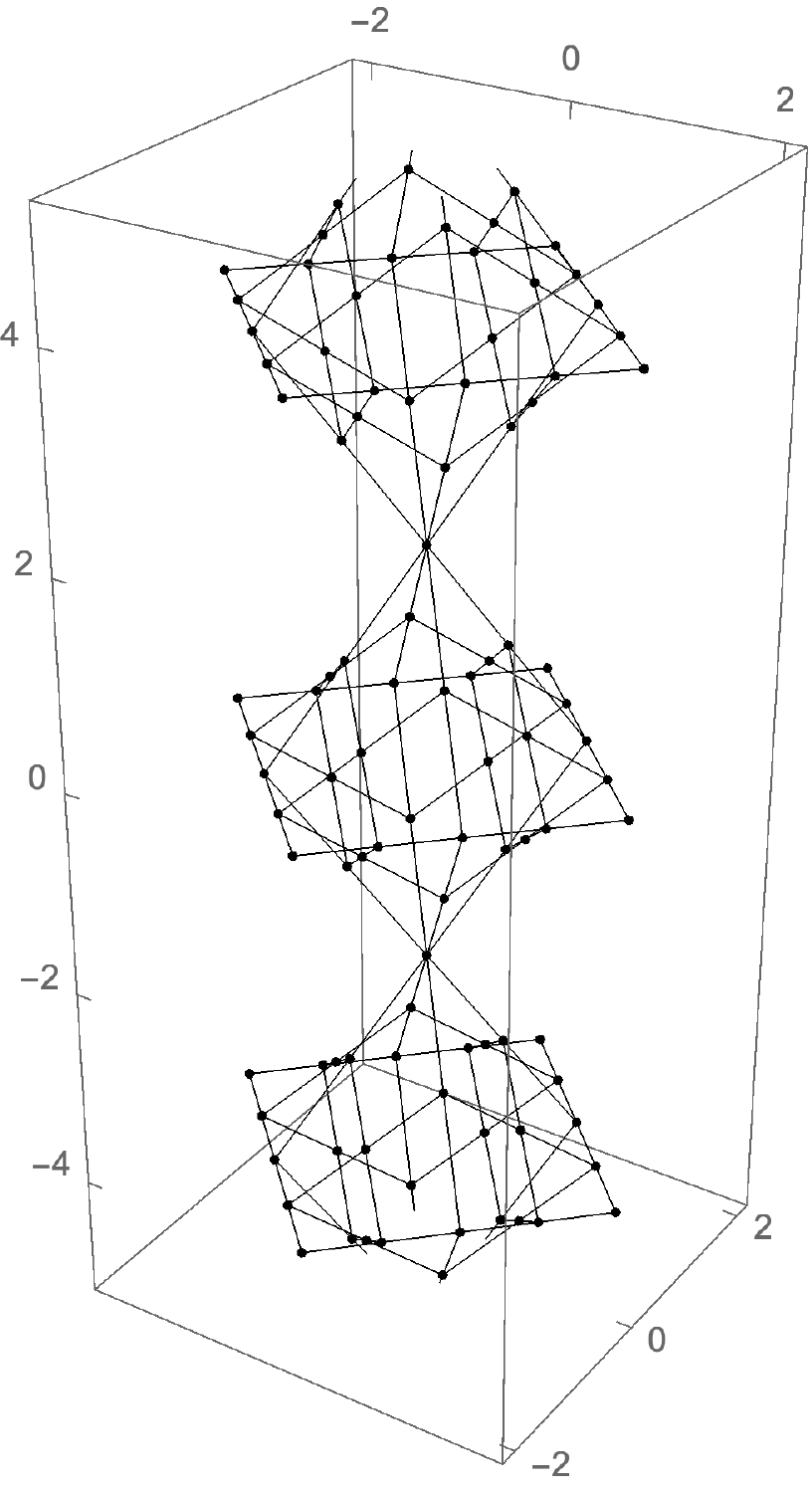}}
\hfill
\subfigure[$N=3$]{\includegraphics[height=6cm,keepaspectratio]%
{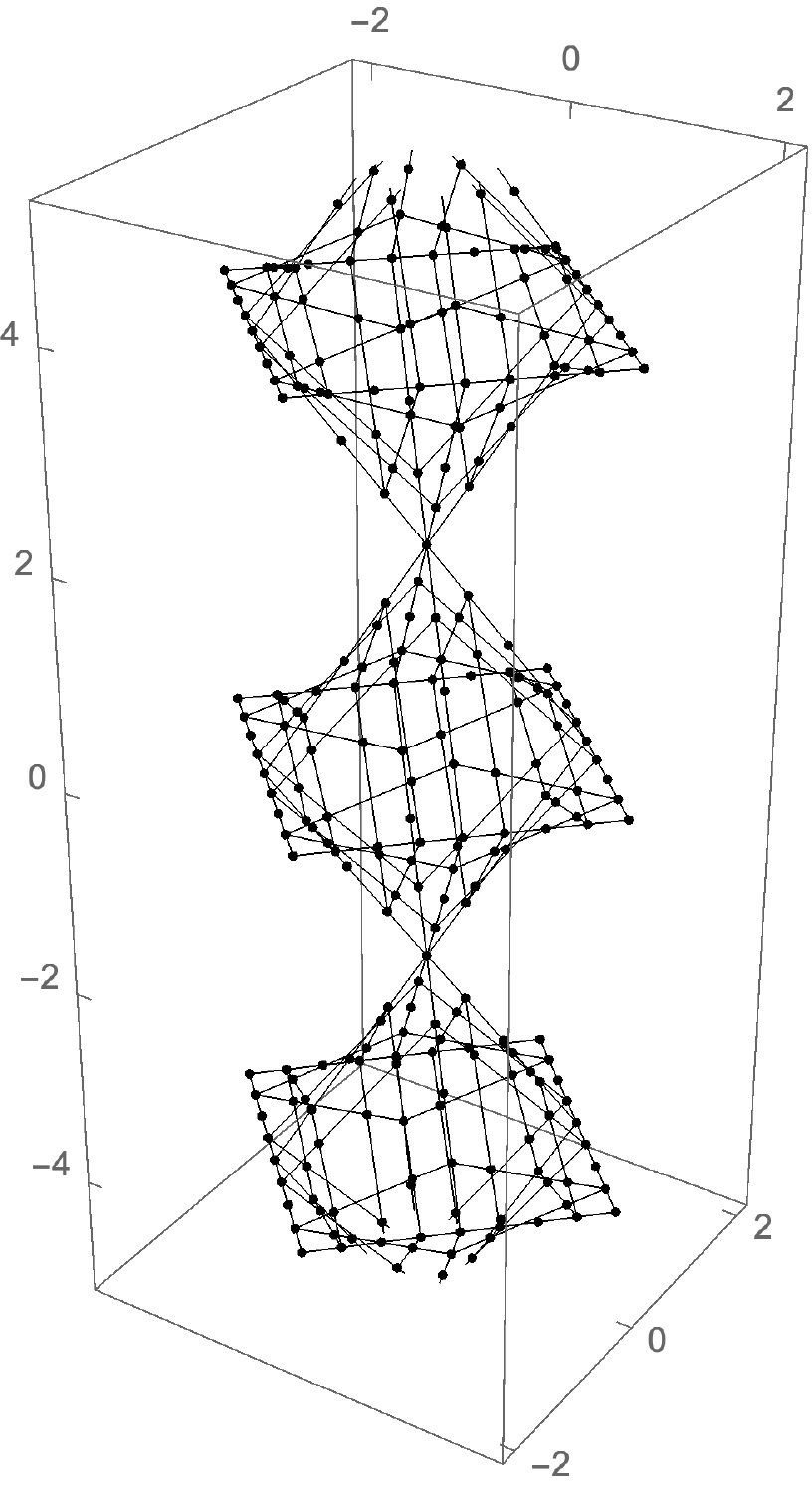}}
\hfill
\hfill
\caption{Discrete indefinite improper affine spheres $f^m_n$
with $\epsilon = \delta = 1$ and $N = N_1 = N_2$.
The figures in upper line show views from the top.}
\label{fig:d-exam2-2}
\end{figure}

\item
Third example is given by
\begin{equation*}
\gamma^1_n =
\cos \left(\theta_1 \epsilon n\right)
\begin{bmatrix}
1/2 + \cos^2 \left(\theta_1 \epsilon n\right)\\
2 \sin \left(\theta_1 \epsilon n\right)
\end{bmatrix},\quad
\gamma^2_m =
\cos \left(\theta_2 \delta m\right)
\begin{bmatrix}
1/2 + \cos^2 \left(\theta_2 \delta m\right)\\
2 \sin \left(\theta_2 \delta m\right)
\end{bmatrix}.
\end{equation*}
We have
\begin{equation*}
\begin{split}
\det \left[\gamma^1_{n-1}, \gamma^1_n\right]
=\;&
2 \left(3 + 2
\sin \left(\theta_1 \epsilon n\right)
\sin \left(\theta_1 \epsilon \left(n-1\right)\right)\right)
\sin \left(\theta_1 \epsilon/2\right)\\
&\cdot \cos \left(\theta_1 \epsilon n\right)
\cos \left(\theta_1 \epsilon \left(n-1/2\right)\right)
\cos \left(\theta_1 \epsilon \left(n-1\right)\right),
\end{split}
\end{equation*}
and hence
\begin{equation*}
\begin{split}
{\sum}_k^n \det \left[\gamma^1_{k-1}, \gamma^1_k\right]
=\;&
% \frac{3 + 6 \cos \left(\theta_1 \epsilon\right)
% + \cos \left(2 \theta_1 \epsilon\right)}{4}
% \sin \left(\theta_1 \epsilon n\right)\\
% &+ \frac{5}{8 \left(1 + 2 \cos \left(\theta_1 \epsilon\right)\right)}
% \sin \left(3 \theta_1 \epsilon n\right)\\
% &- \frac{1}{8 \left(1 + 2 \cos \left(\theta_1 \epsilon\right)
% + 2 \cos \left(2 \theta_1 \epsilon\right)\right)}
% \sin \left(5 \theta_1 \epsilon n\right).
c_1 \left(\theta_1 \epsilon\right)
\sin \left(\theta_1 \epsilon n\right)
+ c_2 \left(\theta_1 \epsilon\right)
\sin \left(3 \theta_1 \epsilon n\right)\\
&+ c_3 \left(\theta_1 \epsilon\right)
\sin \left(5 \theta_1 \epsilon n\right).
\end{split}
\end{equation*}
Here the coefficients are given as
\begin{align*}
c_1 \left(t\right) &=
\frac{3 + 6 \cos t + \cos 2t}{4},\\
c_2 \left(t\right) &=
\frac{5}{8 \left(1 + 2 \cos t\right)},\\
c_3 \left(t\right) &=
- \frac{1}{8 \left(1 + 2 \cos t + 2 \cos 2t\right)}.
\end{align*}
This should be compared with \eqref{ex:integrate-det}.
We have
\begin{equation*}
f^m_n =
\begin{bmatrix}
\left(1 + \left(1/2\right) \cos \left(2 \theta_1 \epsilon n\right)\right)
\cos \left(\theta_1 \epsilon n\right)
+ \left(1 + \left(1/2\right) \cos \left(2 \theta_2 \delta m\right)\right)
\cos \left(\theta_2 \delta m\right)\\
\sin \left(2 \theta_1 \epsilon n\right)
+ \sin \left(2 \theta_2 \delta m\right)\\
z^m_n
\end{bmatrix},
\end{equation*}
where
\begin{equation*}
\begin{split}
z^m_n =\;&
- \cos \left(\theta_1 \epsilon n\right)
\cos \left(\theta_2 \delta m\right)
\left(\sin \left(\theta_1 \epsilon n\right)
- \sin \left(\theta_2 \delta m\right)\right)\\
&\cdot
\left(3 + 2 \sin \left(\theta_1 \epsilon n\right)
\sin \left(\theta_2 \delta m\right)\right)\\
&+ c_1 \left(\theta_1 \epsilon\right)
\sin \left(\theta_1 \epsilon n\right)
- c_1 \left(\theta_2 \delta\right)
\sin \left(\theta_2 \delta m\right)\\
&+ c_2 \left(\theta_1 \epsilon\right)
\sin \left(3 \theta_1 \epsilon n\right)
- c_2 \left(\theta_2 \delta\right)
\sin \left(3 \theta_2 \delta m\right)\\
&+ c_3 \left(\theta_1 \epsilon\right)
\sin \left(5 \theta_1 \epsilon n\right)
- c_3 \left(\theta_2 \delta\right)
\sin \left(5 \theta_2 \delta m\right).
\end{split}
\end{equation*}
Its data is given as
\begin{align*}
A_n =\;&
4 \epsilon^{-3}
\sin^3 \left(\theta_1 \epsilon/2\right)
\cos \left(\theta_1 \epsilon/2\right)
\big(a_1 \left(\theta_1 \epsilon/2\right)\\
&+ a_2 \left(\theta_1 \epsilon/2\right)
\cos \left(2 \theta_1 \epsilon n\right)\\
&+
a_3 \left(\theta_1 \epsilon/2\right)
\cos \left(4 \theta_1 \epsilon n\right)\big)
\cos \left(\theta_1 \epsilon n\right),\\
B_m =\;&
- 4 \delta^{-3}
\sin^3 \left(\theta_2 \delta/2\right)
\cos \left(\theta_2 \delta/2\right)
\big(a_1 \left(\theta_2 \delta/2\right)\\
&+ a_2 \left(\theta_2 \delta/2\right)
\cos \left(2 \theta_2 \delta m\right)\\
&+
a_3 \left(\theta_2 \delta/2\right)
\cos \left(4 \theta_2 \delta m\right)\big)
\cos \left(\theta_2 \delta m\right),\\
\omega^m_n =\;&
- 4 \epsilon^{-1} \delta^{-1}
\sin \left(\theta_1 \epsilon/2\right)
\sin \left(\theta_2 \delta/2\right)
\big(4 D^1_{n,m}\\
&+ 8 D^2_{n,m}
\sin \left(\theta_1 \epsilon \left(n + 1/2\right)\right)
\sin \left(\theta_2 \delta \left(m + 1/2\right)\right)\\
&+
3 D^3_{n,m}
\cos \left(\theta_1 \epsilon \left(2n+1\right)\right)
\cos \left(\theta_2 \delta \left(2m+1\right)\right)\big),
\end{align*}
where
\begin{align*}
a_1 \left(t\right) &=
7 + 9 \cos 2t + 2 \cos 4t + \cos 6t,\\
a_2 \left(t\right) &= - 6 - 2 \cos 2t,\\
a_3 \left(t\right) &= 1 + 2 \cos 2t
\end{align*}
and
\begin{align*}
D^1_{n,m} &=
\textstyle
\frac{3 + \cos \left(\theta_1 \epsilon\right)}{4}
\cos \frac{\theta_2 \delta}{2}
\sin \frac{\theta_1 \epsilon \left(2n+1\right)}{2}
- \frac{3 + \cos \left(\theta_2 \delta\right)}{4}
\cos \frac{\theta_1 \epsilon}{2}
\sin \frac{\theta_2 \delta \left(2m+1\right)}{2},\\
D^2_{n,m} &=
\textstyle
\frac{3 + \cos \left(\theta_2 \delta\right)}{4}
\cos \frac{\theta_1 \epsilon}{2}
\sin \frac{\theta_1 \epsilon \left(2n+1\right)}{2}
- \frac{3 + \cos \left(\theta_1 \epsilon\right)}{4}
\cos \frac{\theta_2 \delta}{2}
\sin \frac{\theta_2 \delta \left(2m+1\right)}{2},\\
D^3_{n,m} &=
\textstyle
\frac{1 + 2 \cos \left(\theta_1 \epsilon\right)}{3}
\cos \frac{\theta_2 \delta}{2}
\sin \frac{\theta_1 \epsilon \left(2n+1\right)}{2}
- \frac{1 + 2 \cos \left(\theta_2 \delta\right)}{3}
\cos \frac{\theta_1 \epsilon}{2}
\sin \frac{\theta_2 \delta \left(2m+1\right)}{2}.
\end{align*}
Especially if we choose the parameters $q_1$, $q_2$ so as to be
\begin{equation*}
\theta_1 \epsilon = \theta_2 \delta = \frac{\pi}{N}
\end{equation*}
with a positive integer $N$,
then $\omega$ is factorized as
\begin{align*}
\omega^m_n =\;&
- \textstyle\frac{4}{\epsilon \delta}
\sin^2 \textstyle\frac{\pi}{2N}
\cos \textstyle\frac{\pi}{2N}
\left(\sin \textstyle\frac{\pi \left(2n+1\right)}{2N}
- \sin \textstyle\frac{\pi \left(2m+1\right)}{2N}\right) W^m_n,\\
W^m_n =\;&
\left(3 + \cos \textstyle\frac{\pi}{N}\right)
\left(1 + 2
\sin \textstyle\frac{\pi \left(2n+1\right)}{2N}
\sin \textstyle\frac{\pi \left(2m+1\right)}{2N}\right)\\
&+
\left(1 + 2 \cos \textstyle\frac{\pi}{N}\right)
\cos \textstyle\frac{\pi \left(2n+1\right)}{N}
\cos \textstyle\frac{\pi \left(2m+1\right)}{N}.
\end{align*}
The singular set $S = S_1 \cup S_2$ is given by
\begin{gather*}
S_1 = \left\{\left(n, m\right) \in \Z^2 \;\big|\;
m \equiv n\ \left(\mathrm{mod}\ 2N\right),\
m \equiv - n + N-1\ \left(\mathrm{mod}\ 2N\right)\right\},\\
S_2 = \left\{\left(n, m\right) \in \Z^2 \;\big|\;
W^m_n = 0\right\}.
\end{gather*}
\begin{figure}[H]
\hfill
\hfill
\subfigure[$N=4$]{\includegraphics[height=6cm,keepaspectratio]%
{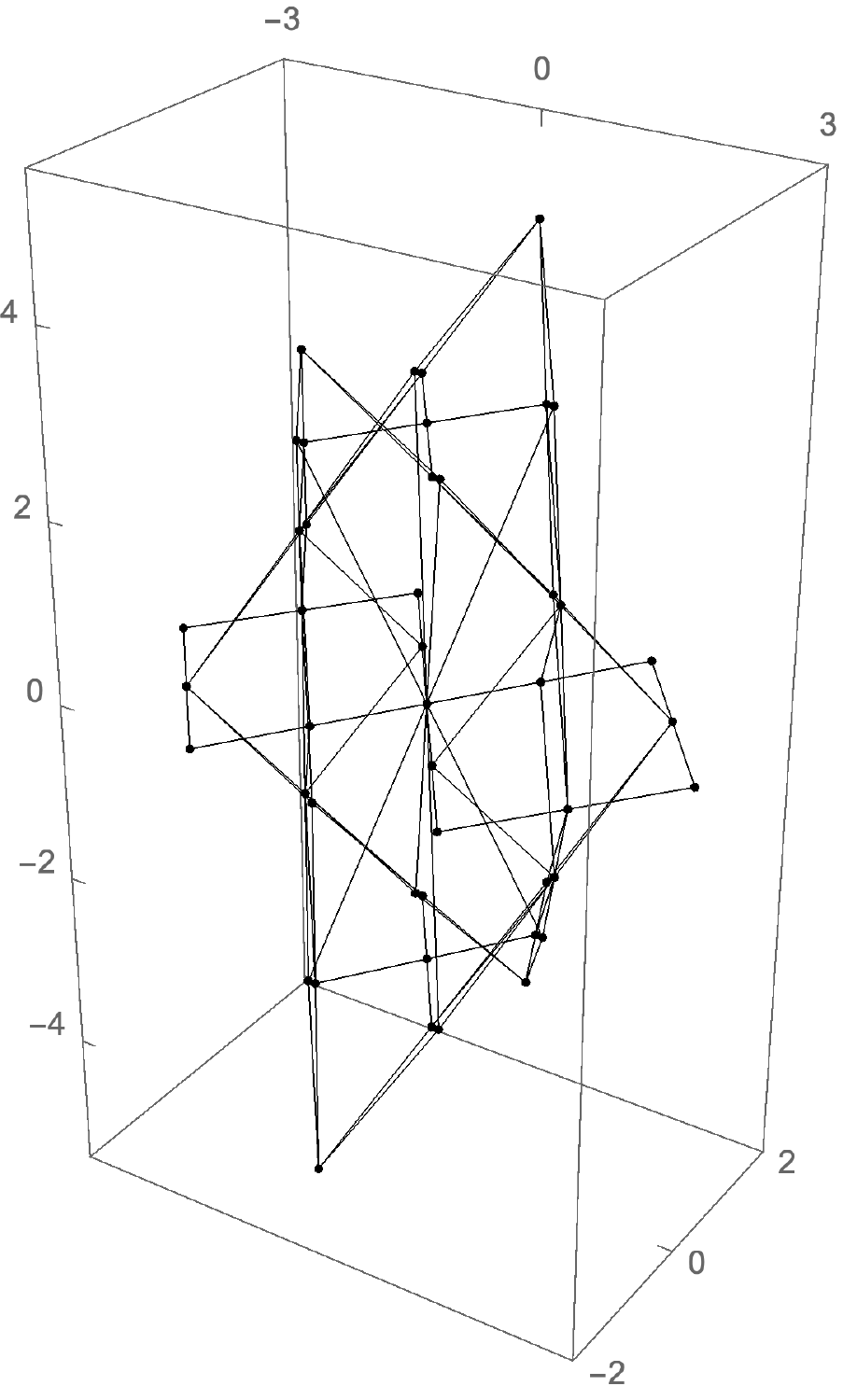}}
\hfill
\subfigure[$N=6$]{\includegraphics[height=6cm,keepaspectratio]%
{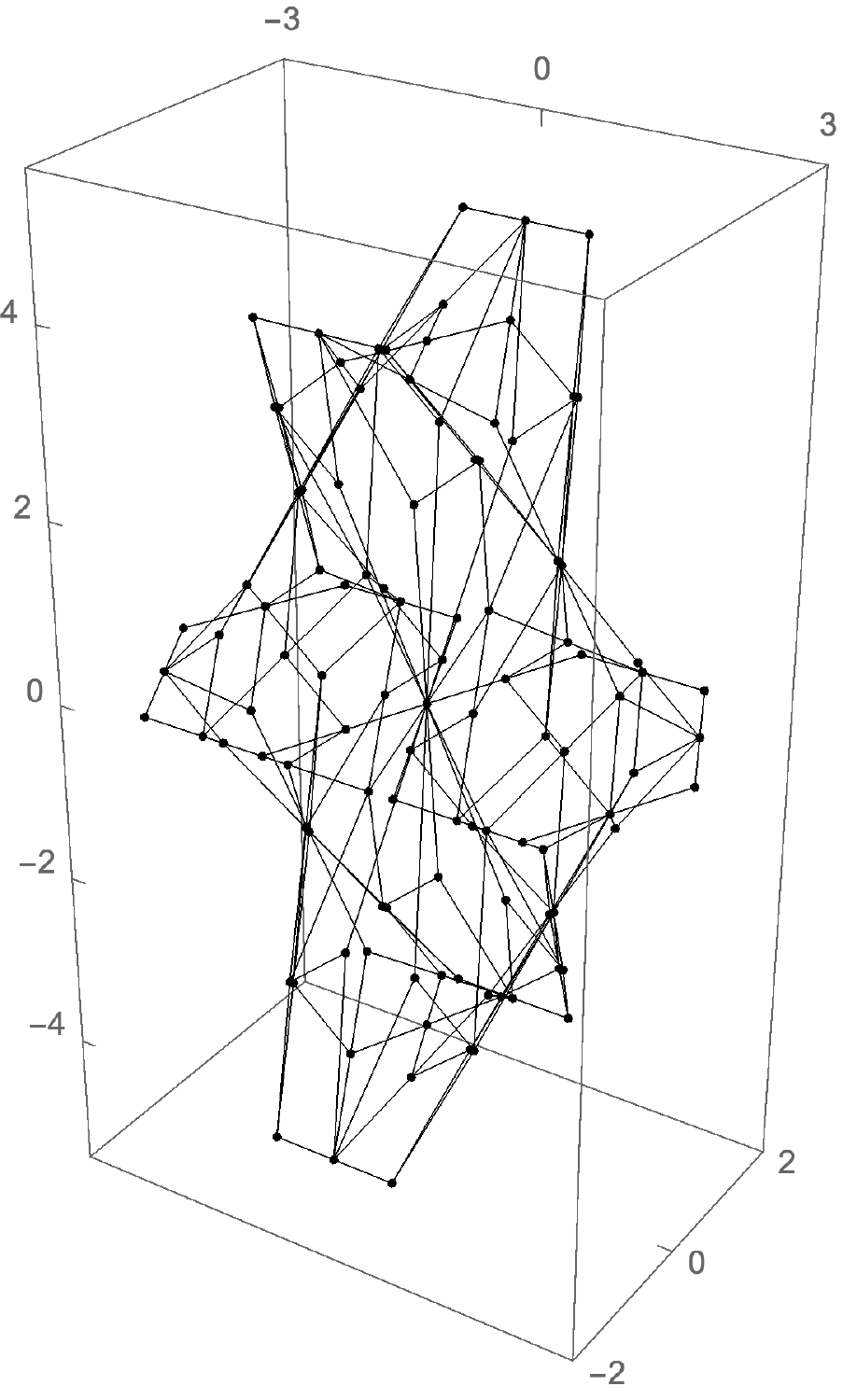}}
\hfill
\subfigure[$N=12$]{\includegraphics[height=6cm,keepaspectratio]%
{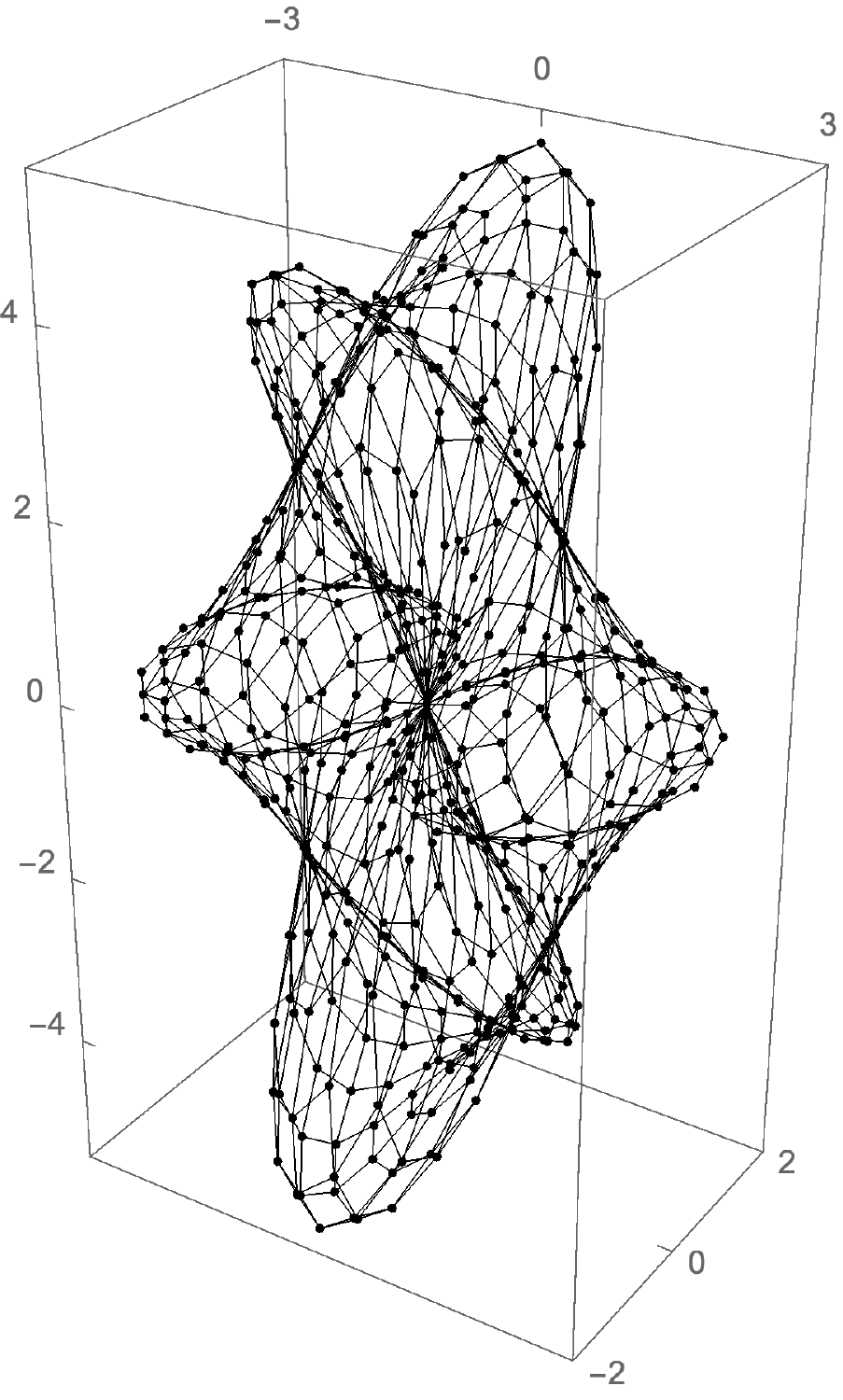}}
\hfill
\hfill
\caption{Discrete indefinite improper affine spheres $f^m_n$
with $\epsilon = \delta = 1$.}
\label{fig:d-exam2-3}
\end{figure}

\end{enumerate}
\end{Example}

\def\cprime{$'$}
% \bibliographystyle{amsplain}
% \bibliography{./reference}
\providecommand{\bysame}{\leavevmode\hbox to3em{\hrulefill}\thinspace}
\providecommand{\MR}{\relax\ifhmode\unskip\space\fi MR }
\providecommand{\MRhref}[2]{%
  \href{http://www.ams.org/mathscinet-getitem?mr=#1}{#2}
}
\providecommand{\href}[2]{#2}

\end{document}